\documentclass[11pt]{preprint}
\usepackage{difftrees} %Added%
\usepackage[full]{textcomp}
\usepackage[osf]{newtxtext} 
\usepackage[cal=boondoxo]{mathalfa}
\usepackage{colortbl}

\usepackage{tikz-cd} %Added%

\usetikzlibrary{cd} %Added%
\usepackage[all,cmtip]{xy}
\usepackage{comment}

\usepackage{amssymb}
\usepackage{mathtools}
\usepackage{hyperref}
\usepackage{breakurl}
\usepackage{mhenvs}
\usepackage{mhequ} 
\usepackage{mhsymb}
\usepackage{booktabs}
\usepackage{tikz}
\usepackage{tcolorbox}
\usepackage{mathrsfs}
\usepackage[utf8]{inputenc}
\usepackage{longtable}
\usepackage{wrapfig}
\usepackage{rotating} %Added%
\usepackage{subcaption}
\usepackage{mathrsfs}
\usepackage{epsfig}
\usepackage{microtype}
\usepackage{comment}
\usepackage{wasysym}
\usepackage{centernot}
\usepackage{enumitem}
\usepackage{bm}
\usepackage{stackrel}

\usepackage{graphicx}
\usepackage{axodraw}
\usepackage{xspace}
\usepackage{subcaption} %Added%
\usepackage{epsfig} %Added%
\usepackage{axodraw} %Added%
\usepackage{xspace} %Added%
\usepackage[toc,page]{appendix} %Added%
\usepackage[all,cmtip]{xy} %Added%
\usepackage{relsize} %Added%
\usepackage{shuffle} %Added%
\makeatletter
\newcommand{\globalcolor}[1]{%
  \color{#1}\global\let\default@color\current@color
}
\makeatother

\usetikzlibrary{calc}
\usetikzlibrary{decorations}
\usetikzlibrary{positioning}
\usetikzlibrary{shapes}
\usetikzlibrary{external}

\definecolor{blush}{rgb}{0.87, 0.36, 0.51}
	\definecolor{brightcerulean}{rgb}{0.11, 0.67, 0.84}
	\definecolor{greenryb}{rgb}{0.4, 0.69, 0.2}

\newif\ifdark
\darkfalse

\ifdark
\definecolor{darkred}{rgb}{0.9,0.2,0.2}
\definecolor{darkblue}{rgb}{0.7,0.3,1}
\definecolor{darkgreen}{rgb}{0.1,0.9,0.1}
\definecolor{franck}{rgb}{0,0.8,1}
\definecolor{pagebackground}{rgb}{.15,.21,.18}
\definecolor{pageforeground}{rgb}{.84,.84,.85}
\pagecolor{pagebackground}
\AtBeginDocument{\globalcolor{pageforeground}}
\definecolor{symbols}{rgb}{0,0.7,1}
\colorlet{connection}{red!80!black}
\colorlet{boxcolor}{blue!50}

\else

\definecolor{darkred}{rgb}{0.7,0.1,0.1}
\definecolor{darkblue}{rgb}{0.4,0.1,0.8}
\definecolor{darkgreen}{rgb}{0.1,0.7,0.1}
\definecolor{franck}{rgb}{0,0,1}
\definecolor{pagebackground}{rgb}{1,1,1}
\definecolor{pageforeground}{rgb}{0,0,0}
\colorlet{symbols}{blue!90!black}
\colorlet{connection}{red!30!black}
\colorlet{boxcolor}{blue!50!black}

\fi

\def\slash{\leavevmode\unskip\kern0.18em/\penalty\exhyphenpenalty\kern0.18em}
\def\dash{\leavevmode\unskip\kern0.18em--\penalty\exhyphenpenalty\kern0.18em}

\DeclareMathAlphabet{\mathbbm}{U}{bbm}{m}{n}

\DeclareFontFamily{U}{BOONDOX-calo}{\skewchar\font=45 }
\DeclareFontShape{U}{BOONDOX-calo}{m}{n}{
  <-> s*[1.05] BOONDOX-r-calo}{}
\DeclareFontShape{U}{BOONDOX-calo}{b}{n}{
  <-> s*[1.05] BOONDOX-b-calo}{}
\DeclareMathAlphabet{\mcb}{U}{BOONDOX-calo}{m}{n}
\SetMathAlphabet{\mcb}{bold}{U}{BOONDOX-calo}{b}{n}
\setlist{noitemsep,topsep=4pt,leftmargin=1.5em}

\DeclareMathAlphabet{\mathbbm}{U}{bbm}{m}{n}

\DeclareMathAlphabet{\mcb}{U}{BOONDOX-calo}{m}{n}
\SetMathAlphabet{\mcb}{bold}{U}{BOONDOX-calo}{b}{n}
\DeclareFontFamily{U}{mathx}{\hyphenchar\font45}
\DeclareFontShape{U}{mathx}{m}{n}{
      <5> <6> <7> <8> <9> <10>
      <10.95> <12> <14.4> <17.28> <20.74> <24.88>
      mathx10
      }{}
\DeclareSymbolFont{mathx}{U}{mathx}{m}{n}
\DeclareMathSymbol{\bigtimes}{1}{mathx}{"91}

\def\s{\mathfrak{s}}

\def\hattimes{\mathbin{\hat\otimes}}
\providecommand{\figures}{false}
{ \ifthenelse{\equal{\figures}{false}} {#1}{\[ {\rm Figure \ missing !} \]} }{}
\def\id{\mathrm{id}}

\def\CH{\mathcal{H}}
\def\CP{\mathcal{P}}

\def\CT{\mathcal{T}}

\tikzstyle{tinydots}=[dash pattern=on \pgflinewidth off \pgflinewidth]
\tikzstyle{superdense}=[dash pattern=on 4pt off 1pt]

\newcommand{\beq}{\begin{equation}}
\newcommand{\eeq}{\end{equation}}

\usepackage{empheq}

\newcommand{\mfT}{\mathfrak{T}}

\newcommand{\mfl}{\mathfrak{l}}

\def\Labe{\mathfrak{e}}

\def\Labn{\mathfrak{n}}

\def\Labhom{\mathfrak{t}}

\def\Lab{\mathfrak{L}}

\def\${|\!|\!|}

\newcounter{theorem}

\newtheorem{ex}[theorem]{Example}

\newenvironment{DIFnomarkup}{}{} % see man latexdiff

\theorembodyfont{\rmfamily}

\newcommand{\rrightarrow}{{\to\hskip -4.9mm\raise 1pt\hbox{$\to$}}}

\newfont{\indic}{bbmss12}

\def\Nabla_#1{\nabla_{\!#1}}

\makeatletter
\pgfdeclareshape{crosscircle}
{
  \inheritsavedanchors[from=circle] % this is nearly a circle
  \inheritanchorborder[from=circle]
  \inheritanchor[from=circle]{north}
  \inheritanchor[from=circle]{north west}
  \inheritanchor[from=circle]{north east}
  \inheritanchor[from=circle]{center}
  \inheritanchor[from=circle]{west}
  \inheritanchor[from=circle]{east}
  \inheritanchor[from=circle]{mid}
  \inheritanchor[from=circle]{mid west}
  \inheritanchor[from=circle]{mid east}
  \inheritanchor[from=circle]{base}
  \inheritanchor[from=circle]{base west}
  \inheritanchor[from=circle]{base east}
  \inheritanchor[from=circle]{south}
  \inheritanchor[from=circle]{south west}
  \inheritanchor[from=circle]{south east}
  \inheritbackgroundpath[from=circle]
  \foregroundpath{
    \centerpoint%
    \pgf@xc=\pgf@x%
    \pgf@yc=\pgf@y%
    \pgfutil@tempdima=\radius%
    \pgfmathsetlength{\pgf@xb}{\pgfkeysvalueof{/pgf/outer xsep}}%  
    \pgfmathsetlength{\pgf@yb}{\pgfkeysvalueof{/pgf/outer ysep}}%  
    \ifdim\pgf@xb<\pgf@yb%
      \advance\pgfutil@tempdima by-\pgf@yb%
    \else%
      \advance\pgfutil@tempdima by-\pgf@xb%
    \fi%
    \pgfpathmoveto{\pgfpointadd{\pgfqpoint{\pgf@xc}{\pgf@yc}}{\pgfqpoint{-0.707107\pgfutil@tempdima}{-0.707107\pgfutil@tempdima}}}
    \pgfpathlineto{\pgfpointadd{\pgfqpoint{\pgf@xc}{\pgf@yc}}{\pgfqpoint{0.707107\pgfutil@tempdima}{0.707107\pgfutil@tempdima}}}
    \pgfpathmoveto{\pgfpointadd{\pgfqpoint{\pgf@xc}{\pgf@yc}}{\pgfqpoint{-0.707107\pgfutil@tempdima}{0.707107\pgfutil@tempdima}}}
    \pgfpathlineto{\pgfpointadd{\pgfqpoint{\pgf@xc}{\pgf@yc}}{\pgfqpoint{0.707107\pgfutil@tempdima}{-0.707107\pgfutil@tempdima}}}
  }
}
\makeatother

\def\symbol#1{\textcolor{symbols}{#1}}

\def\decorate#1#2{
        \ifnum#2>0
    		\foreach \count in {1,...,#2}{
	       	let
				\p1 = (sourcenode.center),
                \p2 = (sourcenode.east),
				\n1 = {\x2-\x1},
				\n2 = {1mm},
				\n3 = {(1.3+0.6*(\count-1))*\n1},
				\n4 = {0.7*\n1}
			in 
        		node[rectangle,fill=symbols,rotate=30,inner sep=0pt,minimum width=0.2*\n2,minimum height=\n2] at ($(sourcenode.center) + (\n3,\n4)$) {}
				}
		\fi
        \ifnum#1>0
    		\foreach \count in {1,...,#1}{
	       	let
				\p1 = (sourcenode.center),
                \p2 = (sourcenode.east),
				\n1 = {\x2-\x1},
				\n2 = {1mm},
				\n3 = {(1.3+0.6*(\count-1))*\n1},
				\n4 = {0.7*\n1}
			in 
        		node[rectangle,fill=symbols,rotate=-30,inner sep=0pt,minimum width=0.2*\n2,minimum height=\n2] at ($(sourcenode.center) + (-\n3,\n4)$) {}
				}
		\fi
}

\tikzset{
    dectriangle/.style 2 args={
        triangle,
        alias=sourcenode,
        append after command={\decorate{#1}{#2}}
    },
    dectriangle/.default={0}{0},
}

\tikzset{
	cross/.style={path picture={ 
  		\draw[symbols]
			(path picture bounding box.south east) -- (path picture bounding box.north west) (path picture bounding box.south west) -- (path picture bounding box.north east);
		}},
root/.style={circle,fill=green!50!black,inner sep=0pt, minimum size=1.2mm},
        dot/.style={circle,fill=pageforeground,inner sep=0pt, minimum size=1mm},
        dotred/.style={circle,fill=pageforeground!50!pagebackground,inner sep=0pt, minimum size=2mm},
        var/.style={circle,fill=pageforeground!10!pagebackground,draw=pageforeground,inner sep=0pt, minimum size=3mm},
        kernel/.style={semithick,shorten >=2pt,shorten <=2pt},
        kernels/.style={snake=zigzag,shorten >=2pt,shorten <=2pt,segment amplitude=1pt,segment length=4pt,line before snake=2pt,line after snake=5pt,},
        rho/.style={densely dashed,semithick,shorten >=2pt,shorten <=2pt},
           testfcn/.style={dotted,semithick,shorten >=2pt,shorten <=2pt},
        renorm/.style={shape=circle,fill=pagebackground,inner sep=1pt},
        labl/.style={shape=rectangle,fill=pagebackground,inner sep=1pt},
        xic/.style={very thin,circle,draw=symbols,fill=symbols,inner sep=0pt,minimum size=1.2mm},
        g/.style={very thin,rectangle,draw=symbols,fill=symbols!10!pagebackground,inner sep=0pt,minimum width=2.5mm,minimum height=1.2mm},
        xi/.style={very thin,circle,draw=symbols,fill=symbols!10!pagebackground,inner sep=0pt,minimum size=1.2mm},
	xies/.style={very thin,rectangle,fill=green!50!black!25,draw=symbols,inner sep=0pt,minimum size=1.1mm},
	xiesf/.style={very thin,rectangle,fill=green!50!black,draw=symbols,inner sep=0pt,minimum size=1.1mm},
        xix/.style={very thin,crosscircle,fill=symbols!10!pagebackground,draw=symbols,inner sep=0pt,minimum size=1.2mm},
        X/.style={very thin,cross,rectangle,fill=pagebackground,draw=symbols,inner sep=0pt,minimum size=1.2mm},
	xib/.style={thin,circle,fill=symbols!10!pagebackground,draw=symbols,inner sep=0pt,minimum size=1.6mm},
	xie/.style={thin,circle,fill=green!50!black,draw=symbols,inner sep=0pt,minimum size=1.6mm},
	xid/.style={thin,circle,fill=symbols,draw=symbols,inner sep=0pt,minimum size=1.6mm},
	xibx/.style={thin,crosscircle,fill=symbols!10!pagebackground,draw=symbols,inner sep=0pt,minimum size=1.6mm},
	kernels2/.style={very thick,draw=connection,segment length=12pt},
	keps/.style={thin,draw=symbols,->},
	kepspr/.style={thick,draw=connection,->},
	krho/.style={thin,draw=symbols,superdense,->},
	krhopr/.style={thick,draw=connection,superdense},
	triangle/.style = { regular polygon, regular polygon sides=3},
	not/.style={thin,circle,draw=connection,fill=connection,inner sep=0pt,minimum size=0.5mm},
	diff/.style = {very thin,draw=symbols,triangle,fill=red!50!black,inner sep=0pt,minimum size=1.6mm},
	diff1/.style = {very thin,dectriangle={1}{0},fill=red!50!black,draw=symbols,inner sep=0pt,minimum size=1.6mm},
	diff2/.style = {very thin,dectriangle={1}{1},fill=red!50!black,draw=symbols,inner sep=0pt,minimum size=1.6mm},
		diffmini/.style = {very thin,rectangle,fill=black,draw=black,inner sep=0pt,minimum size=0.75mm},
	 kernelsmod/.style={very thick,draw=connection,segment length=12pt},
	 rec/.style = {very thin,rectangle,fill=black,draw=black,inner sep=0pt,minimum size=2mm},
	cerc/.style={very thin,circle,draw=black,fill=symbols,inner sep=0pt,minimum size=2mm},
	stars/.style={very thin,star,star points=6,star point ratio=0.5, draw=black,fill=red,inner sep=0pt,minimum size=0.7mm},
	>=stealth,
        }
        \tikzset{
root/.style={circle,fill=black!50,inner sep=0pt, minimum size=3mm},
        circ/.style={circle,fill=white,draw=black,very thin,inner sep=.5pt, minimum size=1.2mm},
        round1/.style={fill=white,outer sep = 0,inner sep=2pt,rounded corners=1mm,draw,text=black,thin,minimum size=1.2mm},
          circ1/.style={circle,fill=red!10,draw=red,very thin,inner sep=.5pt, minimum size=1.2mm},
        rect/.style={fill=white,outer sep = 0,inner sep=2pt,rectangle,draw,text=black,thin,minimum size=1.2mm},
        rect1/.style={fill=white,outer sep = 0,inner sep=2pt,rectangle,draw,text=black,thin,minimum size=1.2mm},
        round2/.style={fill=red!10,outer sep = 0,inner sep=2pt,rounded corners=1mm,draw,text=black,thin,minimum size=1.2mm},
       round3/.style={fill=blue!10,outer sep = 0,inner sep=2pt,rounded corners=1mm,draw,text=black,thin,minimum size=1.2mm}, 
        rect2/.style={fill=black!10,outer sep = 0,inner sep=2pt,rectangle,draw,text=black,thin,minimum size=1.2mm},
        dot/.style={circle,fill=black,inner sep=0pt, minimum size=1.2mm},
        dotred/.style={circle,fill=black!50,inner sep=0pt, minimum size=2mm},
        var/.style={circle,fill=black!10,draw=black,inner sep=0pt, minimum size=3mm},
        kernel/.style={semithick,shorten >=2pt,shorten <=2pt},
         diag/.style={thin,shorten >=4pt,shorten <=4pt},
        kernel1/.style={thick},
        kernels/.style={snake=zigzag,shorten >=2pt,shorten <=2pt,segment amplitude=1pt,segment length=4pt,line before snake=2pt,line after snake=5pt,},
		kernels1/.style={snake=zigzag,segment amplitude=0.5pt,segment length=2pt},
		rho1/.style={densely dotted,semithick},
        rho/.style={densely dashed,semithick,shorten >=2pt,shorten <=2pt},
           testfcn/.style={dotted,semithick,shorten >=2pt,shorten <=2pt},
           visible/.style={draw, circle, fill, inner sep=0.25ex},
        renorm/.style={shape=circle,fill=white,inner sep=1pt},
        labl/.style={shape=rectangle,fill=white,inner sep=1pt},
        xic/.style={very thin,circle,fill=symbols,draw=black,inner sep=0pt,minimum size=1.2mm},
        xi/.style={very thin,circle,fill=blue!10,draw=black,inner sep=0pt,minimum size=1.2mm},
	xib/.style={very thin,circle,fill=blue!10,draw=black,inner sep=0pt,minimum size=1.6mm},
	xie/.style={very thin,circle,fill=green!50!black,draw=black,inner sep=0pt,minimum size=1mm},
	xid/.style={very thin,circle,fill=symbols,draw=black,inner sep=0pt,minimum size=1.6mm},
	edgetype/.style={very thin,circle,draw=black,inner sep=0pt,minimum size=5mm},
	nodetype/.style={very thick,circle,draw=black,inner sep=0pt,minimum size=5mm},
	kernels2/.style={very thick,draw=connection,segment length=12pt},
clean/.style={thin,circle,fill=black,inner sep=0pt,minimum size=1mm},	not/.style={thin,circle,fill=symbols,draw=connection,fill=connection,inner sep=0pt,minimum size=0.8mm},
	>=stealth,
        }

\makeatletter
\def\DeclareSymbol#1#2#3{%
	\expandafter\gdef\csname MH@symb@#1\endcsname{\tikzsetnextfilename{symbol#1}%
	\tikz[baseline=#2,scale=0.15,draw=symbols,line join=round]{#3}}%
	\expandafter\gdef\csname MH@symb@#1s\endcsname{\scalebox{0.75}{\tikzsetnextfilename{symbol#1}%
	\tikz[baseline=#2,scale=0.15,draw=symbols,line join=round]{#3}}}%
	\expandafter\gdef\csname MH@symb@#1ss\endcsname{\scalebox{0.65}{\tikzsetnextfilename{symbol#1}%
	\tikz[baseline=#2,scale=0.15,draw=symbols,line join=round]{#3}}}%
	}
\def\<#1>{\ifthenelse{\boolean{mmode}}{\mathchoice{\csname MH@symb@#1\endcsname}{\csname MH@symb@#1\endcsname}{\csname MH@symb@#1s\endcsname}{\csname MH@symb@#1ss\endcsname}}{\csname MH@symb@#1\endcsname}}
\makeatother

\DeclareSymbol{Xi22}{0.5}{\draw (0,0) node[xi] {} -- (-1,1) node[not] {} -- (0,2) node[xi] {};} % 1 not used in the text
\DeclareSymbol{Xi2}{-2}{\draw (-1,-0.25) node[xi] {} -- (0,1) node[xi] {};} % 2
\DeclareSymbol{Xi2b}{-2}{\draw (-1,-0.25) node[xic] {} -- (0,1) node[xic] {};} % 2
\DeclareSymbol{Xi2g}{-2}{\draw (-1,-0.25) node[xies] {} -- (0,1) node[xi] {};} % 2
\DeclareSymbol{Xi2g2}{-2}{\draw (-1,-0.25) node[xi] {} -- (0,1) node[xies] {};} % 2
\DeclareSymbol{cXi2}{-2}{\draw (0,-0.25) node[xi] {} -- (-1,1) node[xic] {};}%3
\DeclareSymbol{Xi3}{0}{\draw (0,0) node[xi] {} -- (-1,1) node[xi] {} -- (0,2) node[xi] {};}%4
\DeclareSymbol{XiIIXi}{0}{\draw (0,0) node[xi] {} -- (-1,1); \draw[kernels2] (-1,1) node[not] {} -- (0,2) node[xi] {};}%4

\DeclareSymbol{Xi4}{2}{\draw (0,0) node[xi] {} -- (-1,1) node[xi] {} -- (0,2) node[xi] {} -- (-1,3) node[xi] {};}%5
\DeclareSymbol{Xi4_1}{2}{\draw (0,0) node[xic] {} -- (-1,1) node[xic] {} -- (0,2) node[xi] {} -- (-1,3) node[xi] {};}%6
\DeclareSymbol{Xi4_2}{2}{\draw (0,0) node[xic] {} -- (-1,1) node[xi] {} -- (0,2) node[xi] {} -- (-1,3) node[xic] {};}%7
\DeclareSymbol{Xi2X}{-2}{\draw (0,-0.25) node[xi] {} -- (-1,1) node[xix] {};}%8
\DeclareSymbol{XXi2}{-2}{\draw (0,-0.25) node[xix] {} -- (-1,1) node[xi] {};}%9
\DeclareSymbol{IIXi}{0}{\draw (0,-0.25) node[not] {} -- (-1,1) node[xi] {} -- (0,2) node[xi] {};}%10
\DeclareSymbol{IXi^2}{-1}{\draw (-1,1) node[xi] {} -- (0,0) node[not] {} -- (1,1) node[xi] {};}%11
\DeclareSymbol{IIXi^2}{-4}{\draw (0,-1.5) node[not] {} -- (0,0);
\draw[kernels2] (-1,1) node[xi] {} -- (0,0) node[not] {} -- (1,1) node[xi] {};}%12
\DeclareSymbol{XiX}{-2.8}{\node[xibx] {};}%13
\DeclareSymbol{tauX}{-2.8}{ \node[X] {};}%14
\DeclareSymbol{Xi}{-2.8}{\node[xib] {};}%15

\DeclareSymbol{IXiX}{-1}{\draw (0,-0.25) node[not] {} -- (-1,1) node[xix] {};}%18
\DeclareSymbol{IXi3}{2}{\draw (0,-0.25) node[not] {} -- (-1,1) node[xi] {} -- (0,2) node[xi] {} -- (-1,3) node[xi] {};}%20
\DeclareSymbol{IXi}{-2}{\draw (0,-0.25) node[not] {} -- (-1,1) node[xi] {};}%21
\DeclareSymbol{XiI}{-2}{\draw (0,-0.25) node[xi] {} -- (-1,1) node[not] {};}%22

\DeclareSymbol{Xi4b}{0}{\draw(0,1.5) node[xi] {} -- (0,0); \draw (-1,1) node[xi] {} -- (0,0) node[xi] {} -- (1,1) node[xi] {};}%23
\DeclareSymbol{Xi4b'}{0}{\draw(0,1.5) node[xi] {} -- (0,-0.2); \draw (-1,1) node[xi] {} -- (0,-0.2) node[not] {} -- (1,1) node[xi] {};}%24 not used
\DeclareSymbol{Xi4c}{0}{\draw (0,1) -- (0.8,2.2) node[xi] {};\draw (0,-0.25) node[xi] {} -- (0,1) node[xi] {} -- (-0.8,2.2) node[xi] {};}%25
\DeclareSymbol{Xi4d}{-4.5}{\draw (0,-1.5) node[not] {} -- (0,0); \draw (-1,1) node[xi] {} -- (0,0) node[xi] {} -- (1,1) node[xi] {};}%26 not used
\DeclareSymbol{Xi4e}{0}{\draw (0,2) node[xi] {} -- (-1,1) node[xi] {} -- (0,0) node[xi] {} -- (1,1) node[xi] {};}%27
\DeclareSymbol{Xi4e'}{0}{\draw (0,2) node[xi] {} -- (-1,1) node[xi] {} -- (0,-0.2) node[not] {} -- (1,1) node[xi] {};}%28 not used

\DeclareSymbol{Xitwo}%29
{0}{\draw[kernels2] (0,0) node[not] {} -- (-1,1) node[not] {}
-- (-2,2) node[not]{} -- (-3,3) node[xi]  {};
\draw[kernels2] (0,0) -- (1,1) node[xi] {};
\draw[kernels2] (-1,1) -- (0,2) node[xi] {};
\draw[kernels2] (-2,2) -- (-1,3) node[xi] {};}

\DeclareSymbol{IXitwo}%30
{0}{\draw (-.7,1.2) node[xi] {} -- (0,-0.2) -- (.7,1.2) node[xi] {};}
\DeclareSymbol{I1Xitwo}%30
{0}{\draw[kernels2] (0,0) node[not] {} -- (-1,1) node[xi] {};
\draw[kernels2] (0,0) -- (1,1) node[xi] {};}

\DeclareSymbol{I1Xitwobis}%30
{0}{\draw[kernels2] (0,0) node[not] {} -- (-1,1) node[xies] {};
\draw[kernels2] (0,0) -- (1,1) node[xies] {};}

\DeclareSymbol{I1Xitwog}%30
{0}{\draw[kernels2] (0,0) node[not] {} -- (-1,1) node[xies] {};
\draw[kernels2] (0,0) -- (1,1) node[xi] {};}

\DeclareSymbol{cI1Xitwo}%31
{0}{\draw[kernels2] (0,0) node[not] {} -- (-1,1) node[xic] {};
\draw[kernels2] (0,0) -- (1,1) node[xi] {};}

\DeclareSymbol{I1IXi3}{0}{\draw (0,0) node[xi] {} -- (-1,1) ; %32
\draw[kernels2] (-1,1) node[not] {} -- (0,2) node[xi] {};
\draw[kernels2] (-1,1) node[not] {} -- (-2,2) node[xi] {};}

\DeclareSymbol{I1Xi3c}{-1}{\draw[kernels2](0,1.5) node[xi] {} -- (0,0) node[not] {}; \draw (-1,1) node[xi] {} -- (0,0) ; \draw[kernels2] (0,0) -- (1,1) node[xi] {};}%33

\DeclareSymbol{I1Xi3cbis}{-1}{\draw[kernels2](0,1.5) node[xies] {} -- (0,0) node[not] {}; \draw (-1,1) node[xies] {} -- (0,0) ; \draw[kernels2] (0,0) -- (1,1) node[xies] {};}%33

\DeclareSymbol{I1IXi3b}{0}{\draw[kernels2] (0,0) node[not] {} -- (-1,1) ; \draw[kernels2] (0,0)   -- (1,1) node[xi] {} ;
\draw (-1,1) node[xi] {} -- (0,2) node[xi] {};
}%34

\DeclareSymbol{I1IXi3c}{0}{\draw[kernels2] (0,0) node[not] {} -- (-1,1) ; \draw[kernels2] (0,0)   -- (1,1) node[xi] {} ;
\draw[kernels2] (-1,1) node[not] {} -- (0,2) node[xi] {};
\draw[kernels2] (-1,1) node[not] {} -- (-2,2) node[xi] {};}%35

\DeclareSymbol{I1IXi3cbis}{0}{\draw[kernels2] (0,0) node[not] {} -- (-1,1) ; \draw[kernels2] (0,0)   -- (1,1) node[xies] {} ;
\draw[kernels2] (-1,1) node[not] {} -- (0,2) node[xies] {};
\draw[kernels2] (-1,1) node[not] {} -- (-2,2) node[xies] {};}%35

\DeclareSymbol{I1Xi}{0}{\draw[kernels2] (0,0) node[not] {} -- (-1,1)  node[xi] {} ;}%36

\DeclareSymbol{I1Xi4a}{2}{\draw[kernels2] (0,0) node[not] {} -- (-1,1) ; \draw[kernels2] (0,0) node[not] {} -- (1,1) node[xi] {} ;%37
\draw (-1,1) node[xi] {} -- (0,2) node[xi] {} -- (-1,3) node[xi] {};}
\DeclareSymbol{cI1Xi4a}{2}{\draw[kernels2] (0,0) node[not] {} -- (-1,1) ; \draw[kernels2] (0,0) node[not] {} -- (1,1) node[xic] {} ;%39
\draw (-1,1) node[xic] {} -- (0,2) node[xi] {} -- (-1,3) node[xi] {};}
\DeclareSymbol{I1Xi4b}{2}{\draw (0,0) node[xi] {} -- (-1,1) node[xi] {} -- (0,2) ; \draw[kernels2] (0,2) node[not] {} -- (-1,3) node[xi] {};\draw[kernels2] (0,2)  -- (1,3) node[xi] {};
}%41

\DeclareSymbol{cI1Xi4b}{2}{\draw (0,0) node[xic] {} -- (-1,1) node[xic] {} -- (0,2) ; \draw[kernels2] (0,2) node[not] {} -- (-1,3) node[xi] {};\draw[kernels2] (0,2)  -- (1,3) node[xi] {};
}%42

\DeclareSymbol{I1Xi4c}{2}{\draw (0,0) node[xi] {} -- (-1,1) node[not] {}; \draw[kernels2] (-1,1) -- (0,2) ; 
\draw[kernels2] (-1,1) -- (-2,2) node[xi] {} ;
\draw (0,2) node[xi] {} -- (-1,3) node[xi] {};}%43

\DeclareSymbol{cI1Xi4c}{2}{\draw (0,0) node[xic] {} -- (-1,1) node[not] {}; \draw[kernels2] (-1,1) -- (0,2) ; 
\draw[kernels2] (-1,1) -- (-2,2) node[xic] {} ;
\draw (0,2) node[xi] {} -- (-1,3) node[xi] {};}%44

\DeclareSymbol{I1Xi4ab}{2}{\draw[kernels2] (0,0) node[not] {} -- (-1,1) ; \draw[kernels2] (0,0) node[not] {} -- (1,1) node[xi] {};\draw (-1,1) node[xi] {} -- (0,2) ; \draw[kernels2] (0,2) node[not] {} -- (-1,3) node[xi] {};\draw[kernels2] (0,2)  -- (1,3) node[xi] {}; }%45

\DeclareSymbol{cI1Xi4ab}{2}{\draw[kernels2] (0,0) node[not] {} -- (-1,1) ; \draw[kernels2] (0,0) node[not] {} -- (1,1) node[xic] {};\draw (-1,1) node[xic] {} -- (0,2) ; \draw[kernels2] (0,2) node[not] {} -- (-1,3) node[xi] {};\draw[kernels2] (0,2)  -- (1,3) node[xi] {}; }%46

\DeclareSymbol{I1Xi4bc}{2}{\draw (0,0) node[xi] {} -- (-1,1) node[not] {}; \draw[kernels2] (-1,1) -- (0,2) ; %47
\draw[kernels2] (-1,1) -- (-2,2) node[xi] {} ; \draw[kernels2] (0,2) node[not] {} -- (-1,3) node[xi] {};\draw[kernels2] (0,2)  -- (1,3) node[xi] {};
}

\DeclareSymbol{cI1Xi4bc}{2}{\draw (0,0) node[xic] {} -- (-1,1) node[not] {}; \draw[kernels2] (-1,1) -- (0,2) ; %48
\draw[kernels2] (-1,1) -- (-2,2) node[xic] {} ; \draw[kernels2] (0,2) node[not] {} -- (-1,3) node[xi] {};\draw[kernels2] (0,2)  -- (1,3) node[xi] {};
}

\DeclareSymbol{I1Xi4abcc1}{2}{\draw[kernels2] (0,0) node[not] {} -- (-1,1) node[not] {}%50
-- (-2,2) node[not]{} -- (-3,3) node[xic]  {};
\draw[kernels2] (0,0) -- (1,1) node[xic] {};
\draw[kernels2] (-1,1) -- (0,2) node[xi] {};
\draw[kernels2] (-2,2) -- (-1,3) node[xi] {};
}

\DeclareSymbol{I1Xi4abcc1b}{2}{\draw[kernels2] (0,0) node[not] {} -- (-1,1) node[not] {}%50
-- (-2,2) node[not]{} -- (-3,3) node[xi]  {};
\draw[kernels2] (0,0) -- (1,1) node[xic] {};
\draw[kernels2] (-1,1) -- (0,2) node[xic] {};
\draw[kernels2] (-2,2) -- (-1,3) node[xi] {};
}

\DeclareSymbol{I1Xi4abcc2}{2}{\draw[kernels2] (0,0) node[not] {} -- (-1,1) node[not] {}%51
-- (-2,2) node[not]{} -- (-3,3) node[xic]  {};
\draw[kernels2] (0,0) -- (1,1) node[xi] {};
\draw[kernels2] (-1,1) -- (0,2) node[xi] {};
\draw[kernels2] (-2,2) -- (-1,3) node[xic] {};
}

\DeclareSymbol{I1Xi4ac}{2}{\draw[kernels2] (0,0) node[not] {} -- (-1,1) ; \draw[kernels2] (0,0) node[not] {} -- (1,1) node[xi] {}; 
\draw[kernels2] (-1,1) node[not] {} -- (0,2) ; %52
\draw[kernels2] (-1,1) -- (-2,2) node[xi] {} ;
\draw (0,2) node[xi] {} -- (-1,3) node[xi] {};}

\DeclareSymbol{cI1Xi4ac}{2}{\draw[kernels2] (0,0) node[not] {} -- (-1,1) ; \draw[kernels2] (0,0) node[not] {} -- (1,1) node[xic] {}; 
\draw[kernels2] (-1,1) node[not] {} -- (0,2) ; %53
\draw[kernels2] (-1,1) -- (-2,2) node[xic] {} ;
\draw (0,2) node[xi] {} -- (-1,3) node[xi] {};}

\DeclareSymbol{I1Xi4acc1}{2}{\draw[kernels2] (0,0) node[not] {} -- (-1,1) ; \draw[kernels2] (0,0) node[not] {} -- (1,1) node[xic] {}; 
\draw[kernels2] (-1,1) node[not] {} -- (0,2) ; %54
\draw[kernels2] (-1,1) -- (-2,2) node[xi] {} ;
\draw (0,2) node[xic] {} -- (-1,3) node[xi] {};}

\DeclareSymbol{I1Xi4acc2}{2}{\draw[kernels2] (0,0) node[not] {} -- (-1,1) ; \draw[kernels2] (0,0) node[not] {} -- (1,1) node[xic] {}; 
\draw[kernels2] (-1,1) node[not] {} -- (0,2) ; %55
\draw[kernels2] (-1,1) -- (-2,2) node[xi] {} ;
\draw (0,2) node[xi] {} -- (-1,3) node[xic] {};}

\DeclareSymbol{2I1Xi4}{2}{\draw[kernels2] (0,0) node[not] {} -- (-1,1) node[not] {};
\draw[kernels2] (0,0) -- (1,1) node[not] {};%56
\draw[kernels2] (-1,1) -- (-1.5,2.5) node[xi] {};
\draw[kernels2] (-1,1) -- (-0.5,2.5) node[xi] {};
\draw[kernels2] (1,1) -- (0.5,2.5) node[xi] {};
\draw[kernels2] (1,1) -- (1.5,2.5) node[xi] {};
}

\DeclareSymbol{2I1Xi4dis}{2}{\draw[kernels2] (0,0) node[not] {} -- (-1,1) node[not] {};
\draw[kernels2] (0,0) -- (1,1) node[not] {};%56
\draw[kernels2] (-1,1) -- (-1.5,2.5) node[xies] {};
\draw[kernels2] (-1,1) -- (-0.5,2.5) node[xies] {};
\draw[kernels2] (1,1) -- (0.5,2.5) node[xies] {};
\draw[kernels2] (1,1) -- (1.5,2.5) node[xies] {};
}

\DeclareSymbol{2I1Xi4c1}{2}{\draw[kernels2] (0,0) node[not] {} -- (-1,1) node[not] {};%57
\draw[kernels2] (0,0) -- (1,1) node[not] {};
\draw[kernels2] (-1,1) -- (-1.5,2.5) node[xic] {};
\draw[kernels2] (-1,1) -- (-0.5,2.5) node[xi] {};
\draw[kernels2] (1,1) -- (0.5,2.5) node[xic] {};
\draw[kernels2] (1,1) -- (1.5,2.5) node[xi] {};
}

\DeclareSymbol{2I1Xi4c2}{2}{\draw[kernels2] (0,0) node[not] {} -- (-1,1) node[not] {};%58
\draw[kernels2] (0,0) -- (1,1) node[not] {};
\draw[kernels2] (-1,1) -- (-1.5,2.5) node[xic] {};
\draw[kernels2] (-1,1) -- (-0.5,2.5) node[xic] {};
\draw[kernels2] (1,1) -- (0.5,2.5) node[xi] {};
\draw[kernels2] (1,1) -- (1.5,2.5) node[xi] {};
}

\DeclareSymbol{2I1Xi4b}{2}{\draw[kernels2] (0,0) node[not] {} -- (-1,1) ;%59
\draw[kernels2] (0,0) -- (1,1);
\draw (-1,1) node[xi] {} -- (-1,2.5) node[xi] {};
\draw (1,1)  node[xi] {} -- (1,2.5) node[xi] {};
}

\DeclareSymbol{2I1Xi4bb}{2}{\draw[kernels2] (0,0) node[not] {} -- (-1,1) ;%59
\draw[kernels2] (0,0) -- (1,1);
\draw (-1,1) node[xi] {} -- (-1,2.5) node[xiesf] {};
\draw (1,1)  node[xi] {} -- (1,2.5) node[xic] {};
}

\DeclareSymbol{2I1Xi4c}{2}{\draw[kernels2] (0,0) node[not] {} -- (-1,1);%60
\draw[kernels2] (0,0) -- (1,1) node[not] {};
\draw (-1,1)  node[xi] {} -- (-1,2.5) node[xi] {};
\draw[kernels2] (1,1) -- (0.4,2.5) node[xi] {};
\draw[kernels2] (1,1) -- (1.6,2.5) node[xi] {};
}

\DeclareSymbol{2I1Xi4cc1}{2}{\draw[kernels2] (0,0) node[not] {} -- (-1,1);%61
\draw[kernels2] (0,0) -- (1,1) node[not] {};
\draw (-1,1)  node[xic] {} -- (-1,2.5) node[xi] {};
\draw[kernels2] (1,1) -- (0.4,2.5) node[xic] {};
\draw[kernels2] (1,1) -- (1.6,2.5) node[xi] {};
}

\DeclareSymbol{2I1Xi4cc2}{2}{\draw[kernels2] (0,0) node[not] {} -- (-1,1);%62
\draw[kernels2] (0,0) -- (1,1) node[not] {};
\draw (-1,1)  node[xic] {} -- (-1,2.5) node[xic] {};
\draw[kernels2] (1,1) -- (0.4,2.5) node[xi] {};
\draw[kernels2] (1,1) -- (1.6,2.5) node[xi] {};
}

\DeclareSymbol{Xi4ba}{0}{\draw(-0.5,1.5) node[xi] {} -- (0,0); \draw (-1.5,1) node[xi] {} -- (0,0) node[not] {}; \draw[kernels2] (0,0) -- (1.5,1) node[xi] {};
\draw[kernels2] (0,0) -- (0.5,1.5) node[xi] {} ;}%63

\DeclareSymbol{Xi4badis}{0}{\draw(-0.5,1.5) node[xies] {} -- (0,0); \draw (-1.5,1) node[xies] {} -- (0,0) node[not] {}; \draw[kernels2] (0,0) -- (1.5,1) node[xies] {};
\draw[kernels2] (0,0) -- (0.5,1.5) node[xies] {} ;}%63

\DeclareSymbol{Xi4ba1}{0}{\draw(-0.5,1.5) node[xi] {} -- (0,0); \draw (-1.5,1) node[xi] {} -- (0,0) node[not] {}; \draw[kernels2] (0,0) -- (1.5,1) node[xic] {};
\draw[kernels2] (0,0) -- (0.5,1.5) node[xic] {} ;}%64

\DeclareSymbol{Xi4ba1b}{0}{\draw(-0.5,1.5) node[xic] {} -- (0,0); \draw (-1.5,1) node[xic] {} -- (0,0) node[not] {}; \draw[kernels2] (0,0) -- (1.5,1) node[xi] {};
\draw[kernels2] (0,0) -- (0.5,1.5) node[xi] {} ;}%64

\DeclareSymbol{Xi4ba1bdiff}{0}{\draw(-0.5,1.5) node[xic] {} -- (0,0); \draw (-1.5,1) node[xic] {} -- (0,0) node[not] {}; \draw (0,0) -- (1.5,1) node[xi] {};
\draw (0,0) -- (0.5,1.5) node[xi] {};
\draw(0,0) node[diff] {};}%64

\DeclareSymbol{Xi4ba1bb}{0}{\draw(-0.5,1.5) node[xic] {} -- (0,0); \draw (-1.5,1) node[xiesf] {} -- (0,0) node[not] {}; \draw[kernels2] (0,0) -- (1.5,1) node[xi] {};
\draw[kernels2] (0,0) -- (0.5,1.5) node[xi] {} ;}%64

\DeclareSymbol{Xi4ba2}{0}{\draw(-0.5,1.5) node[xi] {} -- (0,0); \draw (-1.5,1) node[xic] {} -- (0,0) node[not] {}; \draw[kernels2] (0,0) -- (1.5,1) node[xi] {};
\draw[kernels2] (0,0) -- (0.5,1.5) node[xic] {} ;}%65

\DeclareSymbol{Xi4ba2b}{0}{\draw(-0.5,1.5) node[xi] {} -- (0,0); \draw (-1.5,1) node[xic] {} -- (0,0) node[not] {}; \draw[kernels2] (0,0) -- (1.5,1) node[xi] {};
\draw[kernels2] (0,0) -- (0.5,1.5) node[xiesf] {} ;}%65

\DeclareSymbol{Xi4ca}{0}{\draw (0,1) -- (-1,2.2) node[xi] {};\draw (0,-0.25) node[xi] {} -- (0,1) ; \draw[kernels2] (0,1) node[not] {} -- (1,2.2) node[xi] {};%67
\draw[kernels2] (0,1) {} -- (0,2.7) node[xi] {};
}

\DeclareSymbol{Xi4cb}{0}{\draw (-1,1) -- (-2,2) node[xi] {};\draw[kernels2] (0,0)  -- (-1,1) node[xi] {} ; \draw[kernels2] (0,0) node[not] {} -- (1,1) node[xi] {} ; 
\draw (-1,1) node[xi] {} -- (0,2) node[xi] {};}

\DeclareSymbol{Xi4cbb}{0}{\draw (-1,1) -- (-2,2) node[xiesf] {};\draw[kernels2] (0,0)  -- (-1,1) node[xi] {} ; \draw[kernels2] (0,0) node[not] {} -- (1,1) node[xi] {} ; 
\draw (-1,1) node[xi] {} -- (0,2) node[xic] {};}

\DeclareSymbol{Xi4cbc1}{0}{\draw (-1,1) -- (-2,2) node[xic] {};\draw[kernels2] (0,0)  -- (-1,1) node[xic] {} ; \draw[kernels2] (0,0) node[not] {} -- (1,1) node[xi] {} ; 
\draw (-1,1) node[xic] {} -- (0,2) node[xi] {};}

\DeclareSymbol{Xi4cbc2}{0}{\draw (-1,1) -- (-2,2) node[xi] {};\draw[kernels2] (0,0)  -- (-1,1) node[xi] {} ; \draw[kernels2] (0,0) node[not] {} -- (1,1) node[xic] {} ; 
\draw (-1,1) node[xic] {} -- (0,2) node[xi] {};}

\DeclareSymbol{Xi4cab}{0}{\draw (-1,1) -- (-2,2) node[xi] {};\draw[kernels2] (0,0)  -- (-1,1); \draw[kernels2] (0,0) node[not] {} -- (1,1) node[xi] {} ; 
\draw[kernels2] (-1,1)  {} -- (0,2) node[xi] {};
\draw[kernels2] (-1,1) node[not] {} -- (-1,2.5) node[xi] {};
}%69

\DeclareSymbol{Xi4cabdis}{0}{\draw (-1,1) -- (-2,2) node[xies] {};\draw[kernels2] (0,0)  -- (-1,1); \draw[kernels2] (0,0) node[not] {} -- (1,1) node[xies] {} ; 
\draw[kernels2] (-1,1)  {} -- (0,2) node[xies] {};
\draw[kernels2] (-1,1) node[not] {} -- (-1,2.5) node[xies] {};
}%69

\DeclareSymbol{Xi4cabc1}{0}{\draw (-1,1) -- (-2,2) node[xi] {};\draw[kernels2] (0,0)  -- (-1,1); \draw[kernels2] (0,0) node[not] {} -- (1,1) node[xic] {} ; 
\draw[kernels2] (-1,1)  {} -- (0,2) node[xic] {};
\draw[kernels2] (-1,1) node[not] {} -- (-1,2.5) node[xi] {};
}%69

\DeclareSymbol{Xi4cabc2}{0}{\draw (-1,1) -- (-2,2) node[xic] {};\draw[kernels2] (0,0)  -- (-1,1); \draw[kernels2] (0,0) node[not] {} -- (1,1) node[xic] {} ; 
\draw[kernels2] (-1,1)  {} -- (0,2) node[xi] {};
\draw[kernels2] (-1,1) node[not] {} -- (-1,2.5) node[xi] {};
}%69

\DeclareSymbol{Xi4ea}{1.5}{\draw (-1,2.5) node[xi] {} -- (-1,1) node[xi] {} -- (0,0); %71
 \draw[kernels2] (0,0)  -- (1,1) node[xi] {};
\draw[kernels2] (0,0) node[not] {} -- (0,1.5) node[xi] {}; }

\DeclareSymbol{Xi4eac1}{1.5}{\draw (-1,2.5) node[xic] {} -- (-1,1) node[xi] {} -- (0,0); %72
 \draw[kernels2] (0,0)  -- (1,1) node[xic] {};
\draw[kernels2] (0,0) node[not] {} -- (0,1.5) node[xi] {}; }

\DeclareSymbol{Xi4eac1b}{1.5}{\draw (-1,2.5) node[xic] {} -- (-1,1) node[xi] {} -- (0,0); %72
 \draw[kernels2] (0,0)  -- (1,1) node[xiesf] {};
\draw[kernels2] (0,0) node[not] {} -- (0,1.5) node[xi] {}; }

\DeclareSymbol{Xi4eac2}{1.5}{\draw (-1,2.5) node[xic] {} -- (-1,1) node[xic] {} -- (0,0); %73
 \draw[kernels2] (0,0)  -- (1,1) node[xi] {};
\draw[kernels2] (0,0) node[not] {} -- (0,1.5) node[xi] {}; }

\DeclareSymbol{Xi4eact1}{1.5}{\draw (-1,2.5) node[xic] {} -- (-1,1) node[xi] {} -- (0,0); %74
 \draw (0,0)  -- (1,1) node[xic] {};
\draw[rho] (0,0) node[not] {} -- (0,1.5) node[xi] {}; }

\DeclareSymbol{Xi4eact2}{1.5}{\draw[rho] (-1,2.5) node[xic] {} -- (-1,1) node[xi] {} -- (0,0); %75
 \draw (0,0)  -- (1,1) node[xic] {};
\draw (0,0) node[not] {} -- (0,1.5) node[xi] {}; }

\DeclareSymbol{Xi4eabis}{1.5}{\draw (-1,2.5) node[xi] {} -- (-1,1) ; \draw[kernels2] (-1,1) node[xi] {} -- (0,0); 
 \draw (0,0)  -- (1,1) node[xi] {};%76
\draw[kernels2] (0,0) node[not] {} -- (0,1.5) node[xi] {}; }

\DeclareSymbol{Xi4eabisc1}{1.5}{\draw (-1,2.5) node[xic] {} -- (-1,1) ; \draw[kernels2] (-1,1) node[xi] {} -- (0,0); 
 \draw (0,0)  -- (1,1) node[xi] {};%77
\draw[kernels2] (0,0) node[not] {} -- (0,1.5) node[xic] {}; }

\DeclareSymbol{Xi4eabisc1b}{1.5}{\draw (-1,2.5) node[xic] {} -- (-1,1) ; \draw[kernels2] (-1,1) node[xi] {} -- (0,0); 
 \draw (0,0)  -- (1,1) node[xi] {};%77
\draw[kernels2] (0,0) node[not] {} -- (0,1.5) node[xiesf] {}; }

\DeclareSymbol{Xi4eabisc1bis}{1.5}{\draw (-1,2.5) node[xi] {} -- (-1,1) ; \draw[kernels2] (-1,1) node[xi] {} -- (0,0); 
 \draw (0,0)  -- (1,1) node[xi] {};%78
\draw[kernels2] (0,0) node[not] {} -- (0,1.5) node[xi] {};
\draw (-2,1) node[] {\tiny{$i$}};
\draw (-2,2.5) node[] {\tiny{$\ell$}};
\draw (2,1) node[] {\tiny{$k$}};
\draw (0,2.5) node[] {\tiny{$j$}};
 }

\DeclareSymbol{Xi4eabisc1tris}{1.5}{\draw (-1,2.5) node[xi] {} -- (-1,1) ; \draw[kernels2] (-1,1) node[xi] {} -- (0,0); 
 \draw (0,0)  -- (1,1) node[xi] {};%78
\draw[kernels2] (0,0) node[not] {} -- (0,1.5) node[xi] {};
\draw (-2,1) node[] {\tiny{i}};
\draw (-2,2.5) node[] {\tiny{j}};
\draw (2,1) node[] {\tiny{j}};
\draw (0,2.5) node[] {\tiny{i}};
 }

\DeclareSymbol{Xi4eabisc1quater}{1.5}{\draw (-1,2.5) node[xic] {} -- (-1,1) ; \draw[kernels2] (-1,1) node[xi] {} -- (0,0); 
 \draw (0,0)  -- (1,1) node[xic] {};%78
\draw[kernels2] (0,0) node[not] {} -- (0,1.5) node[xi] {};
 }

\DeclareSymbol{Xi4eabisc2}{1.5}{\draw (-1,2.5) node[xic] {} -- (-1,1) ; \draw[kernels2] (-1,1) node[xi] {} -- (0,0); %79
 \draw (0,0)  -- (1,1) node[xic] {};
\draw[kernels2] (0,0) node[not] {} -- (0,1.5) node[xi] {}; }

\DeclareSymbol{Xi4eabisc2l}{1.5}{\draw (-1,2.5) node[xiesf] {} -- (-1,1) ; \draw[kernels2] (-1,1) node[xi] {} -- (0,0); %79
 \draw (0,0)  -- (1,1) node[xic] {};
\draw[kernels2] (0,0) node[not] {} -- (0,1.5) node[xi] {}; }

\DeclareSymbol{Xi4eabisc2r}{1.5}{\draw (-1,2.5) node[xic] {} -- (-1,1) ; \draw[kernels2] (-1,1) node[xi] {} -- (0,0); %79
 \draw (0,0)  -- (1,1) node[xiesf] {};
\draw[kernels2] (0,0) node[not] {} -- (0,1.5) node[xi] {}; }

\DeclareSymbol{Xi4eabisc3}{1.5}{\draw (-1,2.5) node[xic] {} -- (-1,1) ; \draw[kernels2] (-1,1) node[xic] {} -- (0,0); %80
 \draw (0,0)  -- (1,1) node[xi] {};
\draw[kernels2] (0,0) node[not] {} -- (0,1.5) node[xi] {}; }

\DeclareSymbol{Xi4eb}{0}{%81
\draw[kernels2] (0,2) node[xi] {} -- (-1,1) ; \draw[kernels2] (-2,2)  node[xi] {} -- (-1,1) ; \draw (-1,1)  node[not] {} -- (0,0); 
 \draw (0,0) node[xi] {}  -- (1,1) node[xi] {};
}

\DeclareSymbol{Xi4eab}{1.5}{\draw[kernels2] (-1,2.5) node[xi] {} -- (-1,1) ; \draw[kernels2] (-2,2)  node[xi] {} -- (-1,1) ; \draw (-1,1)  node[not] {} -- (0,0); %82
 \draw[kernels2] (0,0)  -- (1,1) node[xi] {};
\draw[kernels2] (0,0) node[not] {} -- (0,1.5) node[xi] {}; 
}

\DeclareSymbol{Xi4eabdis}{1.5}{\draw[kernels2] (-1,2.5) node[xies] {} -- (-1,1) ; \draw[kernels2] (-2,2)  node[xies] {} -- (-1,1) ; \draw (-1,1)  node[not] {} -- (0,0); %82
 \draw[kernels2] (0,0)  -- (1,1) node[xies] {};
\draw[kernels2] (0,0) node[not] {} -- (0,1.5) node[xies] {}; 
}

\DeclareSymbol{Xi4eabc1}{1.5}{\draw[kernels2] (-1,2.5) node[xic] {} -- (-1,1) ; \draw[kernels2] (-2,2)  node[xi] {} -- (-1,1) ; \draw (-1,1)  node[not] {} -- (0,0); %83
 \draw[kernels2] (0,0)  -- (1,1) node[xic] {};
\draw[kernels2] (0,0) node[not] {} -- (0,1.5) node[xi] {}; 
}

\DeclareSymbol{Xi4eabc2}{1.5}{\draw[kernels2] (-1,2.5) node[xi] {} -- (-1,1) ; \draw[kernels2] (-2,2)  node[xi] {} -- (-1,1) ; \draw (-1,1)  node[not] {} -- (0,0); %84
 \draw[kernels2] (0,0)  -- (1,1) node[xic] {};
\draw[kernels2] (0,0) node[not] {} -- (0,1.5) node[xic] {}; 
}

\DeclareSymbol{Xi4eabbis}{1.5}{\draw[kernels2] (-1,2.5) node[xi] {} -- (-1,1) ; \draw[kernels2] (-2,2)  node[xi] {} -- (-1,1) ; \draw[kernels2] (-1,1)  node[not] {} -- (0,0); %85
 \draw (0,0)  -- (1,1) node[xi] {};
\draw[kernels2] (0,0) node[not] {} -- (0,1.5) node[xi] {}; 
}

\DeclareSymbol{Xi4eabbisc1}{1.5}{\draw[kernels2] (-1,2.5) node[xic] {} -- (-1,1) ; \draw[kernels2] (-2,2)  node[xi] {} -- (-1,1) ; \draw[kernels2] (-1,1)  node[not] {} -- (0,0); %86
 \draw (0,0)  -- (1,1) node[xic] {};
\draw[kernels2] (0,0) node[not] {} -- (0,1.5) node[xi] {}; 
}

\DeclareSymbol{Xi4eabbisc1perm}{1.5}{\draw[kernels2] (-1,2.5) node[xi] {} -- (-1,1) ; \draw[kernels2] (-2,2)  node[xic] {} -- (-1,1) ; \draw[kernels2] (-1,1)  node[not] {} -- (0,0); %86
 \draw (0,0)  -- (1,1) node[xic] {};
\draw[kernels2] (0,0) node[not] {} -- (0,1.5) node[xi] {}; 
}

\DeclareSymbol{Xi4eabbisc2}{1.5}{\draw[kernels2] (-1,2.5) node[xi] {} -- (-1,1) ; \draw[kernels2] (-2,2)  node[xi] {} -- (-1,1) ; \draw[kernels2] (-1,1)  node[not] {} -- (0,0); %87
 \draw (0,0)  -- (1,1) node[xic] {};
\draw[kernels2] (0,0) node[not] {} -- (0,1.5) node[xic] {}; 
}

\DeclareSymbol{Xi2cbis}{0}{\draw[kernels2] (0,1) -- (0.8,2.2) node[xi] {};\draw[kernels2] (0,-0.25) node[not] {} -- (0,1); \draw[kernels2] (0,1) node[not] {} -- (-0.8,2.2) node[xi] {};}%88

\DeclareSymbol{Xi2cbis1}{0}{\draw (0,1) -- (-0.8,2.2) node[xi] {};\draw[kernels2] (0,-0.25) node[not] {} -- (0,1) node[xi] {}; }%89

\DeclareSymbol{Xi2Xbis}{-2}{\draw[kernels2] (0,-0.25)  -- (-1,1) ; \draw (-1,1) node[xix] {};%91
\draw[kernels2] (0,-0.25) node[not] {} -- (1,1) node[xi] {};}

\DeclareSymbol{XXi2bis}{-2}{\draw[kernels2] (0,-0.25) -- (-1,1) node[xi] {};%92
\draw[kernels2] (0,-0.25) node[X] {} -- (1,1) node[xi] {};}

\DeclareSymbol{I1XiIXi}{0}{\draw[kernels2] (0,-0.25) -- (1,1) node[xi] {};%93
\draw (0,-0.25) node[not] {} -- (-1,1) node[xi] {};}

\DeclareSymbol{I1XiIXib}{0}{\draw  (0,-0.25) node[xi] {} -- (0,1) node[not] {};
\draw[kernels2] (0,1) -- (0,2.25) ; \draw (0,2.25) node[xi]{}; }%94

\DeclareSymbol{I1XiIXic}{0}{%95
\draw[kernels2] (0,0) -- (1,1) node[xi] {} ; 
\draw[kernels2] (0,0) node[not] {}  -- (-1,1) node[not] {} -- (0,2) node[xi] {};
}

\DeclareSymbol{thin}{1.4}{\draw[pagebackground] (-0.3,0) -- (0.3,0); \draw  (0,0) -- (0,2);}
\DeclareSymbol{thin2}{1.4}{\draw[pagebackground] (-0.3,0) -- (0.3,0); \draw[tinydots]  (0,0) -- (0,2);}

\DeclareSymbol{thick}{1.4}{\draw[pagebackground] (-0.3,0) -- (0.3,0); \draw[kernels2]  (0,0) -- (0,2);}

\DeclareSymbol{thick2}{1.4}{\draw[pagebackground] (-0.3,0) -- (0.3,0); \draw[kernels2,tinydots]  (0,0) -- (0,2);}

\DeclareSymbol{Xi4ind}{2}{\draw (0,0) node[xi,label={[label distance=-0.2em]right: \scriptsize  $ i $}]  { } -- (-1,1) node[xi,label={[label distance=-0.2em]left: \scriptsize  $ j $}] {} -- (0,2) node[xi,label={[label distance=-0.2em]right: \scriptsize  $ k $}] {} -- (-1,3) node[xi,label={[label distance=-0.2em]left: \scriptsize  $ \ell $}] {};}%115

\DeclareSymbol{Xi4c1}{2}{\draw (0,0) node[xic] {} -- (-1,1) node[xi] {} -- (0,2) node[xic] {} -- (-1,3) node[xi] {};} %119
\DeclareSymbol{IXi2ex}{0}{\draw (0,-0.25) node[xie] {} -- (-1,1) node[xi] {} ; \draw (0,-0.25)-- (1,1) node[xi] {};}
\DeclareSymbol{IXi2ex1}{0}{\draw (0,-0.25) node[xie] {} -- (-1,1) node[xi] {} -- (0,2) node[xi] {};}

\DeclareSymbol{Xi4b1}{0}{\draw(0,1.5) node[xic] {} -- (0,0); \draw (-1,1) node[xic] {} -- (0,0) node[xi] {} -- (1,1) node[xi] {};}

\DeclareSymbol{Xi4ec1}{0}{\draw (0,2) node[xi] {} -- (-1,1) node[xic] {} -- (0,0) node[xic] {} -- (1,1) node[xi] {};}
\DeclareSymbol{Xi4ec2}{0}{\draw (0,2) node[xic] {} -- (-1,1) node[xi] {} -- (0,0) node[xic] {} -- (1,1) node[xi] {};}
\DeclareSymbol{Xi4ec3}{0}{\draw (0,2) node[xic] {} -- (-1,1) node[xic] {} -- (0,0) node[xi] {} -- (1,1) node[xi] {};}

\DeclareSymbol{I1Xi4ac1}{2}{\draw[kernels2] (0,0) node[not] {} -- (-1,1) ; \draw[kernels2] (0,0) node[not] {} -- (1,1) node[xic] {} ;
\draw (-1,1) node[xi] {} -- (0,2) node[xic] {} -- (-1,3) node[xi] {};}%161

\DeclareSymbol{I1Xi4ac2}{2}{\draw[kernels2] (0,0) node[not] {} -- (-1,1) ; \draw[kernels2] (0,0) node[not] {} -- (1,1) node[xic] {} ;
\draw (-1,1) node[xi] {} -- (0,2) node[xi] {} -- (-1,3) node[xic] {};}

\DeclareSymbol{I1Xi4bp}{2}{\draw (0,0) node[not] {} -- (-1,1) node[xi] {} -- (0,2) ; \draw[kernels2] (0,2) node[not] {} -- (-1,3) node[xi] {};\draw[kernels2] (0,2)  -- (1,3) node[xi] {};
}

\DeclareSymbol{I1Xi4bc1}{2}{\draw (0,0) node[xic] {} -- (-1,1) node[xi] {} -- (0,2) ; \draw[kernels2] (0,2) node[not] {} -- (-1,3) node[xi] {};\draw[kernels2] (0,2)  -- (1,3) node[xic] {};
}%165

\DeclareSymbol{I1Xi4bc2}{2}{\draw (0,0) node[xic] {} -- (-1,1) node[xi] {} -- (0,2) ; \draw[kernels2] (0,2) node[not] {} -- (-1,3) node[xic] {};\draw[kernels2] (0,2)  -- (1,3) node[xi] {};
}

\DeclareSymbol{I1Xi4cp}{2}{\draw (0,0) node[not] {} -- (-1,1) node[not] {}; \draw[kernels2] (-1,1) -- (0,2) ; 
\draw[kernels2] (-1,1) -- (-2,2) node[xi] {} ;
\draw (0,2) node[xi] {} -- (-1,3) node[xi] {};}

\DeclareSymbol{I1Xi4cc1}{2}{\draw (0,0) node[xic] {} -- (-1,1) node[not] {}; \draw[kernels2] (-1,1) -- (0,2) ; 
\draw[kernels2] (-1,1) -- (-2,2) node[xi] {} ;
\draw (0,2) node[xic] {} -- (-1,3) node[xi] {};}%169

\DeclareSymbol{I1Xi4cc2}{2}{\draw (0,0) node[xic] {} -- (-1,1) node[not] {}; \draw[kernels2] (-1,1) -- (0,2) ; 
\draw[kernels2] (-1,1) -- (-2,2) node[xi] {} ;
\draw (0,2) node[xi] {} -- (-1,3) node[xic] {};}

\DeclareSymbol{I1Xi4abc1}{2}{\draw[kernels2] (0,0) node[not] {} -- (-1,1) ; \draw[kernels2] (0,0) node[not] {} -- (1,1) node[xic] {};\draw (-1,1) node[xi] {} -- (0,2) ; \draw[kernels2] (0,2) node[not] {} -- (-1,3) node[xic] {};\draw[kernels2] (0,2)  -- (1,3) node[xi] {}; }

\DeclareSymbol{I1Xi4abc2}{2}{\draw[kernels2] (0,0) node[not] {} -- (-1,1) ; \draw[kernels2] (0,0) node[not] {} -- (1,1) node[xic] {};\draw (-1,1) node[xi] {} -- (0,2) ; \draw[kernels2] (0,2) node[not] {} -- (-1,3) node[xi] {};\draw[kernels2] (0,2)  -- (1,3) node[xic] {}; }%173

\DeclareSymbol{R1}{0}{\draw (-1,1) node[xi] {} -- (0,0) node[not] {};%131
\draw[kernels2] (0,1.5) node[xic] {} -- (0,0) -- (1,1) node[xic] {};}
\DeclareSymbol{R2}{0}{\draw (-1,1) node[xic] {} -- (0,0) node[not] {};
\draw[kernels2] (0,1.5)  {} -- (0,0) -- (1,1)  {};
\draw (0,1.5) node[xi] {};
\draw (1,1) node[xic] {};
}%132
\DeclareSymbol{R3}{1}{\draw[kernels2] (-1,1.5)  {} -- (0,0) node[not] {} -- (1,1.5);%133
\draw (-1,1.5) node[xi] {};
\draw[kernels2] (0,3) {} -- (1,1.5) -- (2,3)  {};
\draw  (0,3) node[xic] {} ;
\draw (2,3) node[xic] {};}
\DeclareSymbol{R4}{1}{\draw[kernels2] (-1,1.5) node[xic] {} -- (0,0) node[not] {} -- (1,1.5);%134
\draw[kernels2] (0,3) {} -- (1,1.5) -- (2,3) node[xic] {};
\draw (0,3) node[xi] {};}

\DeclareSymbol{I1Xi4bcp}{2}{\draw (0,0) node[not] {} -- (-1,1) node[not] {}; \draw[kernels2] (-1,1) -- (0,2) ; %175
\draw[kernels2] (-1,1) -- (-2,2) node[xi] {} ; \draw[kernels2] (0,2) node[not] {} -- (-1,3) node[xi] {};\draw[kernels2] (0,2)  -- (1,3) node[xi] {};
}

\DeclareSymbol{I1Xi4bcc1}{2}{\draw (0,0) node[xic] {} -- (-1,1) node[not] {}; \draw[kernels2] (-1,1) -- (0,2) ; 
\draw[kernels2] (-1,1) -- (-2,2) node[xi] {} ; \draw[kernels2] (0,2) node[not] {} -- (-1,3) node[xi] {};\draw[kernels2] (0,2)  -- (1,3) node[xic] {};
}

\DeclareSymbol{I1Xi4bcc2}{2}{\draw (0,0) node[xic] {} -- (-1,1) node[not] {}; \draw[kernels2] (-1,1) -- (0,2) ; 
\draw[kernels2] (-1,1) -- (-2,2) node[xi] {} ; \draw[kernels2] (0,2) node[not] {} -- (-1,3) node[xic] {};\draw[kernels2] (0,2)  -- (1,3) node[xi] {};
} %177

\DeclareSymbol{2I1Xi4bc1}{2}{\draw[kernels2] (0,0) node[not] {} -- (-1,1) ;
\draw[kernels2] (0,0) -- (1,1);
\draw (-1,1) node[xic] {} -- (-1,2.5) node[xi] {};
\draw (1,1)  node[xic] {} -- (1,2.5) node[xi] {};
}

\DeclareSymbol{2I1Xi4bc2}{2}{\draw[kernels2] (0,0) node[not] {} -- (-1,1) ;
\draw[kernels2] (0,0) -- (1,1);
\draw (-1,1) node[xi] {} -- (-1,2.5) node[xic] {};
\draw (1,1)  node[xic] {} -- (1,2.5) node[xi] {};
}

\DeclareSymbol{diff2I1Xi4bc2}{2}{\draw (0,0) node[diff] {} -- (-1,1) ;
\draw (0,0) -- (1,1);
\draw (-1,1) node[xi] {} -- (-1,2.5) node[xic] {};
\draw (1,1)  node[xic] {} -- (1,2.5) node[xi] {};
}

\DeclareSymbol{2I1Xi4bc3}{2}{\draw[kernels2] (0,0) node[not] {} -- (-1,1) ;
\draw[kernels2] (0,0) -- (1,1);
\draw (-1,1) node[xic] {} -- (-1,2.5) node[xic] {};
\draw (1,1)  node[xi] {} -- (1,2.5) node[xi] {};
}

\DeclareSymbol{Xi41}{0}{\draw (0,1) -- (0.8,2.2) node[xic] {};\draw (0,-0.25) node[xi] {} -- (0,1) node[xi] {} -- (-0.8,2.2) node[xic] {};} 

\DeclareSymbol{Xi42}{0}{\draw (0,1) -- (0.8,2.2) node[xi] {};\draw (0,-0.25) node[xic] {} -- (0,1) node[xi] {} -- (-0.8,2.2) node[xic] {};}

\DeclareSymbol{Xi4ca1}{0}{\draw (0,1) -- (-1,2.2) node[xic] {};\draw (0,-0.25) node[xi] {} -- (0,1) ; \draw[kernels2] (0,1) node[not] {} -- (1,2.2) node[xic] {};
\draw[kernels2] (0,1) {} -- (0,2.7) node[xi] {};
}

\DeclareSymbol{Xi4ca2}{0}{\draw (0,1) -- (-1,2.2) node[xi] {};\draw (0,-0.25) node[xi] {} -- (0,1) ; \draw[kernels2] (0,1) node[not] {} -- (1,2.2) node[xic] {};
\draw[kernels2] (0,1) {} -- (0,2.7) node[xic] {};
}

\DeclareSymbol{Xi4cap}{0}{\draw (0,1) -- (-1,2.2) node[xi] {};\draw (0,-0.25) node[not] {} -- (0,1) ; \draw[kernels2] (0,1) node[not] {} -- (1,2.2) node[xi] {};
\draw[kernels2] (0,1) {} -- (0,2.7) node[xi] {};
}

\DeclareSymbol{Xi3a}{0}{
 \draw (-1,1)  node[xi] {} -- (0,0); 
 \draw (0,0) node[xi] {}  -- (1,1) node[xi] {};
}

\DeclareSymbol{Xi4ebc1}{0}{
\draw[kernels2] (0,2) node[xi] {} -- (-1,1) ; \draw[kernels2] (-2,2)  node[xic] {} -- (-1,1) ; \draw (-1,1)  node[not] {} -- (0,0); 
 \draw (0,0) node[xic] {}  -- (1,1) node[xi] {};
}

\DeclareSymbol{Xi4ebc2}{0}{
\draw[kernels2] (0,2) node[xi] {} -- (-1,1) ; \draw[kernels2] (-2,2)  node[xi] {} -- (-1,1) ; \draw (-1,1)  node[not] {} -- (0,0); 
 \draw (0,0) node[xic] {}  -- (1,1) node[xic] {};
}

\DeclareSymbol{Xi2cbispex}{0}{\draw[kernels2] (0,1) -- (0.8,2.2) node[xi] {};\draw (0,-0.25) node[xie] {} -- (0,1); \draw[kernels2] (0,1) node[not] {} -- (-0.8,2.2) node[xi] {};}

\DeclareSymbol{Xi2cbis1p}{0}{\draw (0,1) -- (-0.8,2.2) node[xi] {};\draw (0,-0.25) node[not] {} -- (0,1) node[xi] {}; }

\DeclareSymbol{Xi2Xp}{-2}{\draw (0,-0.25) node[not] {} -- (-1,1) node[xix] {};} % 229 not used

\DeclareSymbol{I1XiIXib}{0}{\draw  (0,-0.25) node[xi] {} -- (0,1) node[not] {};
\draw[kernels2] (0,1) -- (0,2.25) ; \draw (0,2.25) node[xi]{}; }

\DeclareSymbol{IXi2b}{0}{\draw  (0,-0.25) node[xi] {} -- (0,1) node[not] {};
\draw (0,1) -- (0,2.25) ; \draw (0,2.25) node[xi]{}; }

\DeclareSymbol{IXi2bex}{0}{\draw  (0,-0.25) node[xi] {} -- (0,1) node[xie] {};
\draw (0,1) -- (0,2.25) ; \draw (0,2.25) node[xi]{}; }

 \def\1{\mathbf{\symbol{1}}}

\def\one{\mathbf{1}}

\DeclareSymbol{diff}{0}{
\draw (0,0.5) node[diff] {};
}

\DeclareSymbol{diff1}{0}{
\draw (0,0.5) node[diff1] {};
}

\DeclareSymbol{diff2}{0}{
\draw (0,0.5) node[diff2] {};
}

\DeclareSymbol{geo}{0}{
\draw (0,0) node[diff] {};
\draw (0.3,0) node[diff] {};
}

\DeclareSymbol{generic}{0}{
\draw (0,0.6) node[xi] {};
}

\DeclareSymbol{g}{0}{
\draw (0,0.6) node[g] {};
}

\DeclareSymbol{Ito}{0}{
\draw (0,0.6) node[xies] {};
}

\DeclareSymbol{Itob}{0}{
\draw (0,0.6) node[xiesf] {};
}

\DeclareSymbol{greycirc}{0}{
\draw (0,0.3) node[xi] {};
}

\DeclareSymbol{not}{0}{
\draw (0,0.6) node[not] {};
\draw[tinydots] (0,0.6) circle (0.8);
}

\DeclareSymbol{genericb}{0}{
\draw (0,0.6) node[xic] {};
}

\DeclareSymbol{bluecirc}{0}{
\draw (0,0.3) node[xic] {};
}

\DeclareSymbol{genericxix}{0}{
\draw (0,0.6) node[xix] {};
}

\DeclareSymbol{genericX}{0}{
\draw (0,0.6) node[X] {};
}

\DeclareSymbol{diffIto}{1}{
\draw  (0,2.5) -- (0,0) ;
\draw (0,-0.1) node[diff] {};
\draw (0,2.5) node[xies] {};
}
\DeclareSymbol{Itodiff}{2}{
\draw(0,2.9) -- (0,-0.2);
\draw (0,2.9) node[diff] {};
\draw (0,-0.1) node[xies] {};
}

\DeclareSymbol{diffgeneric}{1}{
\draw  (0,2.5) -- (0,0) ;
\draw (0,-0.1) node[diff] {};
\draw (0,2.5) node[xi] {};
}

\DeclareSymbol{genericdiff}{2}{
\draw(0,2.9) -- (0,-0.2);
\draw (0,2.9) node[diff] {};
\draw (0,-0.1) node[xi] {};
}

\DeclareSymbol{diffdot}{2}{
\draw  (0,3) -- (0,-0.1) ;
\draw (0,3) node[not] {};
\draw (0,-0.1) node[diff] {};
}

\DeclareSymbol{diffdotmini}{0}{
\draw  (0,0) -- (0,1.2) ;
\draw (0,1.2) node[not] {};
\draw (0,0) node[diffmini] {};
}

\DeclareSymbol{dotdiff}{2}{
\draw[kernelsmod]  (0,3) -- (0,-0.1) ;
\draw (0,3) node[diff] {};
\draw (0,-0.1) node[not] {};
}

\DeclareSymbol{dotdiff1}{2}{
\draw[kernelsmod]  (0,3) -- (0,-0.1) ;
\draw (0,3) node[diff1] {};
\draw (0,-0.1) node[not] {};
}

\DeclareSymbol{dotdiff1mini}{0}{
\draw[kernelsmod]  (0,1.2) -- (0,0) ;
\draw (0,1.2) node[diffmini] {};
\draw (0,0) node[not] {};
}

\DeclareSymbol{dotdiff2}{2}{
\draw (0,3) -- (0,-0.1) ;
\draw (0,3) node[diff] {};
\draw (0,-0.1) node[not] {};
}

\DeclareSymbol{dotdiff2mini}{0}{
\draw (0,1.2) -- (0,0) ;
\draw (0,1.2) node[diffmini] {};
\draw (0,0) node[not] {};
}

\DeclareSymbol{dotdiffstraight}{0}{
\draw  (0,3) -- (0,-0.1) ;
\draw (0,3) node[diff] {};
\draw (0,-0.1) node[not] {};
}

\DeclareSymbol{arbre1}{0}{
\draw  (0,0) -- (1.5,1.5) ;
\draw (1.5,1.5) node[not] {};
\draw (0,0) node[not] {};
}

\DeclareSymbol{arbre2}{0}{
\draw  (0,0) -- (1.5,1.5) ;
\draw[kernelsmod] (0,0) -- (-1.5,1.5);
\draw (1.5,1.5) node[not] {};
\draw (0,0) node[not] {};
\draw (-1.5,1.5) node[xi] {};
}

\DeclareSymbol{arbre3}{0}{
\draw  (0,0) -- (1.5,1.5) ;
\draw[kernelsmod] (1.5,1.5) -- (0,3);
\draw (0,0) node[not] {};
\draw (1.5,1.5) node[not] {};
\draw (0,3) node[xi] {};
}

\DeclareSymbol{treeeval}{0}{
\draw (0,0) -- (1,1);
\draw (0,0) node[xi] {};
\draw (1.25,1.25) node[xi] {};
\draw (-0.6,0.6) node[]{\tiny{$i$}};
\draw (0.65,1.85) node[]{\tiny{$j$}};
}

\DeclareSymbol{testeval}{0}{
\draw (0,0) -- (1,1);
\draw (0,0) -- (-1,1);
\draw (0,0) node[xi] {};
\draw (1.25,1.25) node[xi] {};
\draw (-1.25,1.25) node[xi] {};
\draw (-0.6,-0.6) node[]{\tiny{$i$}};
\draw (0.65,1.85) node[]{\tiny{$j$}};
\draw (-1.95,1.85) node[]{\tiny{$k$}};
}

\DeclareSymbol{treeeval2}{0}{
\draw[kernelsmod] (-0.25,-1) -- (1,0.5) ;
\draw[kernelsmod] (1,0.5) -- (-0.25,2);
\draw (1,0.5) node[diff2] {};
\draw (-0.25,-1) node[not] {};
\draw (-0.25,2) node[xi] {};
\draw (-0.6,1.2) node[]{\tiny{1}};
}

\DeclareSymbol{arbreact}{1}{
\draw (0,0) node[not] {};
\draw[kernelsmod] (0,0) -- (1,1);
\draw[kernelsmod] (0,0) -- (-1,1);
\draw (-1,1) node[xic] {};
\draw  (0,2) -- (1,1) ;
\draw (0,2) node[xic] {};
\draw (1,1) node[xi] {};
}

\DeclareSymbol{arbreact1}{0}{
\draw (0,-1.5) -- (0,0);
\draw[kernelsmod] (0,0) -- (1,1);
\draw[kernelsmod] (0,0) -- (-1,1);
\draw  (0,2) -- (1,1) ;
\draw (0,-1.5) node[diff] {};
\draw (0,0) node[not] {};
\draw (-1,1) node[xic] {};
\draw (0,2) node[xic] {};
\draw (1,1) node[xi] {};
}

\DeclareSymbol{arbreact2}{0}{
\draw (0,-0.75) -- (-1,0.5); 
\draw (0,-0.75) -- (1,0.5);
\draw (0,1.5) -- (1,0.5);
\draw (0,1.5) node[xic] {};
\draw (1,0.5) node[xi] {};
\draw (-1,0.5) node[xic] {};
\draw (0,-0.75) node[diff] {};
}

\DeclareSymbol{arbreact3}{0}{
\draw[kernelsmod] (0,-0.75) -- (-1,0.5); 
\draw[kernelsmod] (0,-0.75) -- (1,0.5);
\draw (0,1.75) -- (1,0.5);
\draw (2,1.75) -- (1,0.5);
\draw (0,1.75) node[xic] {};
\draw (1,0.5) node[diff] {};
\draw (-1,0.5) node[xic] {};
\draw (2,1.75) node[xi] {};
\draw (0,-0.75) node[not] {};
}

\DeclareSymbol{pre_im_I1Xitwo}{0}{
\draw[kernels2] (0,-0.3) node[not] {} -- (-0.6,0.7) ;
\draw[kernels2] (0,-0.3) -- (0.6,0.7);
\draw (0,0.9) node[g] {};
}

\DeclareSymbol{pre_im_cI1Xi4ab}{2}{
\draw[kernels2] (0,-1) node[not] {} -- (-0.6,0) ;
\draw[kernels2] (0,-1) -- (0.6,0);
\draw (0,0.2) node[g] {};
\draw (0,0.6) -- (0,1.5);
\draw[kernels2] (0,1.5) node[not] {} -- (-0.6,2.5) ;
\draw[kernels2] (0,1.5) -- (0.6,2.5);
\draw (0,2.7) node[g] {};
}%46

\DeclareSymbol{pre_im_I1Xi4acc2}{0}{
\draw[kernels2] (-1,-0.5) node[not] {} -- (-1.6,0.5) ;
\draw[kernels2] (-1,-0.5) -- (-0.4,0.5);
\draw[kernels2] (-1,-0.5) -- (0.2,-1.5) node[not] {} ;
\draw (-1,1.1) -- (-1,2);
\draw[kernels2] (0.2,-1.5) -- (0.2,2);
\draw (-1,0.7) node[g] {};
\draw (-0.3,2.2) node[g] {};
}

\DeclareSymbol{pre_im_I1Xi4abcc2}{2}{
\draw[kernels2] (0,-1) node[not] {} -- (-1,0) node[not] {};
\draw[kernels2] (-1,1.2) node[not] {} -- (-1,0);
\draw[kernels2] (-1,1.2) -- (-1.5,2.5);
\draw[kernels2] (-1,1.2) -- (-0.5,2.5);
\draw[kernels2] (-1,0) -- (0.7,2.5);
\draw[kernels2] (0,-1) -- (1.5,2.5);
\draw (-1,2.7) node[g] {};
\draw (1,2.7) node[g] {};
}

\DeclareSymbol{pre_im_2I1Xi4c1}{2}{
\draw[kernels2] (0,-0.5) node[not] {} -- (-1,0.5) node[not] {};
\draw[kernels2] (0,-0.5) -- (1,0.5) node[not] {};
\draw[kernels2] (-1,0.5) node[not] {}-- (-1.7,2);
\draw[kernels2]  (-1,2) -- (1,0.5);
\draw[kernels2] (-1,0.5) -- (1,2);
\draw[kernels2] (1,0.5) -- (1.7,2);
\draw (-1.2,2.2) node[g] {};
\draw (1.2,2.2) node[g] {};
}

\DeclareSymbol{pre_im_Xi4eabisc2}{2}{
\draw[kernels2] (1.2,-0.5) node[not] {} -- (-0.7,0.8) ;
\draw[kernels2] (1.2,-0.5) -- (0.4,0.8);
\draw (0,1.4)  -- (0,2.2);
\draw (1.2,2.2) -- (1.2,-0.6);
\draw (0,1) node[g] {};
\draw (0.6,2.4) node[g] {};
}

\DeclareSymbol{pre_im_Xi4eabisc22}{2}{
\draw (1.2,-0.5) node[not] {} -- (-0.7,0.8) ;
\draw[kernels2] (1.2,-0.5) -- (0.4,0.8);
\draw (0,1.4)  -- (0,2.2);
\draw[kernels2] (1.2,2.2) -- (1.2,-0.6);
\draw (0,1) node[g] {};
\draw (0.6,2.4) node[g] {};
}

\DeclareSymbol{pre_im_Xi4eabisc222}{2}{
\draw[kernels2] (0.4,-0.5) node[not] {} -- (-0.6,1) ;
\draw[kernels2] (1.2,0) -- (0.3,1);
\draw (0,1.1)  -- (0,2.5);
\draw[kernels2] (1.2,2.5) -- (1.2,0) node[not] {} -- (0.4,-0.6);
\draw (0,1.2) node[g] {};
\draw (0.6,2.5) node[g] {};
}

\DeclareSymbol{pre_im_Xi4eabc2}{2}{
\draw (0,-0.5) node[not] {} -- (-1,0.5) node[not] {};
\draw[kernels2] (-1,0.5) -- (-1.5,2);
\draw[kernels2] (0,-0.5)  -- (0.7,2);
\draw[kernels2] (-1,0.5) -- (-0.5,2);
\draw[kernels2] (0,-0.5) -- (1.5,2);
\draw (-1,2.2) node[g] {};
\draw (1,2.2) node[g] {};
}

\DeclareSymbol{pre_im_Xi4eabbisc2}{2}{
\draw[kernels2] (0,-0.5) node[not] {} -- (-1,0.5) node[not] {};
\draw[kernels2] (-1,0.5) -- (-1.5,2);
\draw[kernels2] (0,-0.5)  -- (0.7,2);
\draw[kernels2] (-1,0.5) -- (-0.5,2);
\draw (0,-0.5) -- (1.5,2);
\draw (-1,2.2) node[g] {};
\draw (1,2.2) node[g] {};
}

\DeclareSymbol{pre_im_I1Xi4abcc1}{2}{
\draw[kernels2] (0,-1) node[not] {} -- (-1,0) node[not] {};
\draw[kernels2] (0,1.1) node[not] {} -- (-1,0);
\draw[kernels2] (-1,0) -- (-1.5,2.5);
\draw[kernels2] (0,1.1) node[not] {} -- (-0.5,2.5);
\draw[kernels2] (0,1.1) -- (0.5,2.5);
\draw[kernels2] (0,-1) -- (1.5,2.5);
\draw (-1,2.7) node[g] {};
\draw (1,2.7) node[g] {};
}

\DeclareSymbol{pre_im_Xi4eabc1}{2}{
\draw (0,-0.5) node[not] {} -- (-1,0.5) node[not] {};
\draw[kernels2] (-1,0.5) -- (-1.5,2);
\draw[kernels2] (0,-0.5)  -- (-0.5,2);
\draw[kernels2] (-1,0.5) -- (0.8,2);
\draw[kernels2] (0,-0.5) -- (1.5,2);
\draw (-1,2.2) node[g] {};
\draw (1,2.2) node[g] {};
}

\DeclareSymbol{pre_im_Xi4ba1b}{2}{
\draw[kernels2] (0,0) node[not] {}  -- (1.8,1.5);
\draw[kernels2] (0,0) -- (0.8,1.5);
\draw (0,-0.1) -- (-1.8,1.5);
\draw (0,-0.1) -- (-0.8,1.5);
\draw (-1,1.7) node[g] {};
\draw (1,1.7) node[g] {};
}

\DeclareSymbol{pre_im_Xi4ba2}{2}{
\draw (0,-0.1) node[not] {}  -- (1.8,1.5);
\draw[kernels2] (0,0) -- (0.8,1.5);
\draw (0,-0.1) -- (-1.8,1.5);
\draw[kernels2] (0,0) -- (-0.8,1.5);
\draw (-1,1.7) node[g] {};
\draw (1,1.7) node[g] {};
}

\DeclareSymbol{pre_im_Xi4cabc2}{2}{
\draw[kernels2] (0,-0.5) node[not] {} -- (-1,0.5) node[not] {};
\draw[kernels2] (-1,0.5) -- (-1.5,2);
\draw (-1,0.5)  -- (0.7,2);
\draw[kernels2] (-1,0.5) -- (-0.5,2);
\draw[kernels2] (0,-0.5) -- (1.5,2);
\draw (-1,2.2) node[g] {};
\draw (1,2.2) node[g] {};
}

\DeclareSymbol{pre_im_Xi4cabc1}{2}{
\draw[kernels2] (0,-0.5) node[not] {} -- (-1,0.5) node[not] {};
\draw (-1,0.5) -- (-1.5,2);
\draw[kernels2] (-1,0.5)  -- (0.7,2);
\draw[kernels2] (-1,0.5) -- (-0.5,2);
\draw[kernels2] (0,-0.5) -- (1.5,2);
\draw (-1,2.2) node[g] {};
\draw (1,2.2) node[g] {};
}

\DeclareSymbol{pre_im_Xi4eabbisc1}{2}{
\draw[kernels2] (0,-0.5) node[not] {} -- (-1,0.5) node[not] {};
\draw[kernels2] (-1,0.5) -- (-1.5,2);
\draw[kernels2] (0,-0.5)  -- (-0.5,2);
\draw[kernels2] (-1,0.5) -- (0.8,2);
\draw (0,-0.5) -- (1.5,2);
\draw (-1,2.2) node[g] {};
\draw (1,2.2) node[g] {};
}

\DeclareSymbol{pre_im_1}{0}{
\draw[kernels2] (0,-0.5) node[not] {} -- (-0.6,0.5) ;
\draw[kernels2] (0,-0.5) -- (0.6,0.5);
\draw (0,1.1)  -- (-0.55,2);
\draw (0,1.1)  -- (0.55,2);
\draw (0,0.7) node[g] {};
\draw (0,2.2) node[g] {};
}

\DeclareSymbol{disconnect}{0}{
\draw[kernels2] (0,-0.5) node[not] {} -- (-0.6,0.5) ;
\draw[kernels2] (0,-0.5) -- (0.6,0.5);
\draw (-0.55,1.1)  -- (-0.55,2.3);
\draw (0.55,2.3) -- (0.55,1.5) -- (1.2,1.5) -- (1.2,3.5) -- (0.55,3.5) -- (0.55,2.7);
\draw (0,0.7) node[g] {};
\draw (0,2.5) node[g] {};
}

\DeclareSymbol{pre_im_2}{2}{\draw[kernels2] (0,0) node[not] {} -- (-1,1) node[not] {};
\draw[kernels2] (0,0) -- (1,1) node[not] {};
\draw[kernels2] (-1,1) -- (-1.5,2.5);
\draw[kernels2] (-1,1) -- (-0.5,2.5);
\draw[kernels2] (1,1) -- (0.5,2.5);
\draw[kernels2] (1,1) -- (1.5,2.5);
\draw (-1,2.7) node[g] {};
\draw (1,2.7) node[g] {};
}

\DeclareSymbol{CX_rec}{0}{
\draw [black] (-0.3,1) to (-0.3,-0.3);
\draw [black] (0.3,1) to (0.3,-0.3);
\draw [black] (-0.3,1) to (-0.3,2.3);
\draw [black] (0.3,1) to (0.3,2.3);
\draw (0,1) node[rec] {};
}

\DeclareSymbol{CX_cerc}{0}{
\draw [black] (0,1) to (0,-0.3);
\draw (0,1) node[cerc] {};
}

\DeclareMathAlphabet{\mathpzc}{OT1}{pzc}{m}{it}

\def\eqref#1{(\ref{#1})}

\makeatletter % Stolen from the internet to make a fat \cdot which isn't as fat as a \bullet
\newcommand*{\bigcdot}{}% Check if undefined
\DeclareRobustCommand*{\bigcdot}{%
  \mathbin{\mathpalette\bigcdot@{}}%
}
\newcommand*{\bigcdot@scalefactor}{.5}
\newcommand*{\bigcdot@widthfactor}{1.15}
\newcommand*{\bigcdot@}[2]{%
  \sbox0{$#1\vcenter{}$}% math axis
  \sbox2{$#1\cdot\m@th$}%
  \hbox to \bigcdot@widthfactor\wd2{%
    \hfil
    \raise\ht0\hbox{%
      \scalebox{\bigcdot@scalefactor}{%
        \lower\ht0\hbox{$#1\bullet\m@th$}%
      }%
    }%
    \hfil
  }%
}
\makeatother

\tcbset
{colframe=boxcolor,colback=symbols!7!pagebackground,coltext=pageforeground,
fonttitle=\bfseries,nobeforeafter,center title,size=fbox,boxsep=1.5pt,
top=0mm,bottom=0mm,boxsep=0mm,tcbox raise base}

\def\two{{\<generic>\kern0.05em\<genericb>}}
\def\twoI{{\<Ito>\kern0.05em\<Itob>}}

\def\st{\mathsf{fgt}}
\def\mail#1{\burlalt{#1}{mailto:#1}}

\usepackage{thmtools} %Added%

\begin{document}

\def\st{\mathsf{fgt}}
\def\mail#1{\burlalt{#1}{mailto:#1}}
\title{Composition and substitution of Regularity Structures B-series}
\author{Yvain Bruned$^1$}
\institute{ 
 IECL (UMR 7502), Université de Lorraine
 \\
Email:\ \begin{minipage}[t]{\linewidth}
\mail{yvain.bruned@univ-lorraine.fr}
\end{minipage}}
%Added by Foivos - Shuffle Product symbol
\def\dsqcup{\sqcup\mathchoice{\mkern-7mu}{\mkern-7mu}{\mkern-3.2mu}{\mkern-3.8mu}\sqcup}
%Added by Foivos - Shuffle Product symbol

\maketitle

\begin{abstract}

\ \ \ \ In this work, we introduce Regularity Structures B-series which are used for describing solutions of singular stochastic partial differential equations (SPDEs). We define composition and substitutions of these B-series and as in the context of B-series for ordinary differential equations, these operations can be rewritten via products and Hopf algebras which have been used for building up renormalised models. These models provide a suitable topology for solving singular SPDEs.
This new construction sheds new light on these products and opens interesting perspectives for the study of singular SPDEs in connection with B-series.
\end{abstract}

\setcounter{tocdepth}{1}
\tableofcontents

\section{ Introduction }

Numerical methods like Runge-Kutta methods for ordinary differential equations (ODE) of the type
\begin{equs}
\partial_t y = F(y), \quad y(0) = y_0 \in \mathbb{R}^{d}\;,
\end{equs}
  can be described via tree-series called B-series (Butcher series). They are of the form
\begin{equs} \label{B-series-form}
	B(\alpha, F,h)(y) = \alpha(\mathbbm{1}) y + \sum_{\tau \in\mathbf{T} } h^{|\tau|} \alpha(\tau) \frac{F[\tau]}{S(\tau)}(y) 
\end{equs}
where $ h $ is the time step, $ \mathbf{T} $ corresponds to the set of rooted trees, $ \mathbbm{1} $ is the empty tree, $ |\tau| $ is the number of nodes of a given tree $ \tau $ and $ S(\tau) $ is its symmetry factor. The coefficients $ F[\tau] $ are elementary differentials depending on the non-linearities appearing in the ODE considered. The last term $ \alpha $ is a function from $ \mathbf{T} $ into $ \mathbb{R} $ that specifies the numerical method. 
These series were originally introduced by Butcher in \cite{Butcher72}. The term of B-series appeared in \cite{HW74}. There are two fundamental operations on these series:
\begin{itemize}
	\item Composition of B-series was originally introduced in \cite{Butcher72,HW74}. It is defined from the composition of a numerical scheme with a smooth function. It is an essential tool as one likes to compose numerical methods to get new ones. 
	\item Substitution of B-series introduced in \cite{CHV05,CHV07} is used mainly for the backward error analysis. The main point is to write the modified equation that a numerical scheme will solve exactly which implies an action onto the non-linearity of the ODE. The substitution is then the replacement of the non-linearity $ F $ by a well-chosen B-series.
\end{itemize}
A survey on the two previous operations is given in \cite{CHV10}.
One of the crucial theorems in the context of B-series is to show that these two operations give rise to B-series whose coefficients can be described with the help of two products $ \star_{\text{\tiny{BCK}}} $ and $\star_{\text{\tiny{EC}}}$:
\begin{equs}
	\label{composition_substitution}
	\begin{aligned}
		B(\alpha_1,F,h) \circ B(\alpha_2, F, h) & = B( \alpha_1 \star_{\tiny{\text{BCK}}} \alpha_2, F,h)
		\\ 	B(\beta_2, \frac{1}{h} B(h,F,\beta_1),h) & = B( \beta_2 \star_{\text{\tiny{EC}}} \beta_1, F,h).
		\end{aligned}
\end{equs}
where $ \circ $ is the composition product and the $ \alpha_i, \beta_i $ are maps from $ \mathbf{T} $ into $ \mathbb{R} $. The products $  \star_{\text{\tiny{BCK}}} $ and $\star_{\text{\tiny{EC}}}$  are defined on forests that are finite sets of trees.
One can think of the $ \alpha_i, \beta_i $ to be extended multiplicatively for the forest product which is the free commutative product on trees. They are called characters. The products $  \star_{\text{\tiny{BCK}}} $ and $\star_{\text{\tiny{EC}}}$ are convolution-type products which means that they are defined via coproducts:
\begin{equs} 
\alpha_1 \star_{\tiny{\text{BCK}}} \alpha_2 = \left( \alpha_1 \otimes \alpha_2 \right) \Delta_{\tiny{\text{BCK}}}, \quad \beta_1 \star_{\tiny{\text{EC}}} \beta_2 = \left( \beta_1 \otimes \beta_2 \right) \Delta_{\tiny{\text{EC}}}
	\end{equs}
where $ \Delta_{\text{\tiny{BCK}}} $ is the Butcher-Connes-Kreimer coproduct also used in Quantum Field Theory for renormalising Feynman diagrams and non-commutative geometry \cite{CK1,CK2}. The second coproduct $ \Delta_{\text{\tiny{EC}}} $ is the extraction-contraction coproduct introduced in \cite{CEM}. These coproducts are in co-interaction (see \cite{CEM}) which at a product level translates into
\begin{equs}
(\beta \star_{\tiny{\text{EC}}}	\alpha_1) \star_{\tiny{\text{BCK}}} (\beta \star_{\tiny{\text{EC}}} \alpha_2) =  \beta \star_{\tiny{\text{EC}}} ( \alpha_1 \star_{\tiny{\text{BCK}}}  \alpha_2) 
\end{equs}
where $ \beta, \alpha_1,\alpha_2 $ are characters.

The main result of this paper is to provide a similar perspective on the resolution of  singular stochastic partial differential equations (SPDEs) via the theory of Regularity Structures. The singular SPDEs considered in the present work are PDEs on $ \mathbb{R}_+ \times \mathbb{R}^d $ with the following form:
\begin{equs} \label{main_equation_bis}
	\partial_t u_{\mathfrak{t}} - \mathcal{L}_{\mathfrak{t}} u_{\mathfrak{t}}  = \sum_{\mathfrak{l} \in \Lab_-} F_{\mathfrak{t}}^{\mathfrak{l}}({\bf{u}}) \xi_{\mathfrak{l}}, \quad \mathfrak{t} \in \Lab_+,
\end{equs}
where $ \Lab_+, \Lab_- $ are finite sets, the $ \xi_{\mathfrak{l}} $ are space-time noises, $ {\bf{u}} $ is the collection of the $ u_{\mathfrak{t}} $ and some of their derivatives. The terms  $ \mathcal{L}_{\mathfrak{t}} $ are differential operators and the $ F_{\mathfrak{t}}^{\mathfrak{l}}({\bf{u}}) $ are non-linearities. The theory of Regularity Structures invented by Martin Hairer in \cite{reg} tells us that the solution of the previous equation could be described via a new Taylor-type expansion:
\begin{equs} \label{expansion_solution}
		u_{\mathfrak{t}} = \sum_{\tau \in \mathfrak{T}_{\mathfrak{t}}} c_{\mathfrak{t},\tau}(z) \, u_{\mathfrak{t},z,\tau} + R_{\mathfrak{t},\mathfrak{T}_{\mathfrak{t}},z}, \quad  |u_{\mathfrak{t},z,\tau}(z')| \lesssim |z'-z|^{\alpha_{\tau}}, 
	\end{equs}
where the $ u_{\mathfrak{t},z,\tau} $ are iterated integrals recentered around a point $ z \in \mathbb{R}_+ \times \mathbb{R}^d $ and the sets $ \mathfrak{T}_{\mathfrak{t}}$ are finite set of decorated trees. The coefficients $c_{\mathfrak{t},\tau}(z)$ could be understood as some type of derivatives. The $ \alpha_{\tau} \in \mathbb{R} $ is the order of the local behaviour  of $  u_{\mathfrak{t},z,\tau}  $ around the point $ z $. The term $ R_{\mathfrak{t},\mathfrak{T}_{\mathfrak{t}},z} $ is a Taylor remainder that behaves better than the iterated integrals in the sense that there exists an $ \alpha_{\mathfrak{t}} \in \mathbb{R} $ such that for every $ \tau \in \mathfrak{T}_{\mathfrak{t}} $
\begin{equs}
	|R_{\mathfrak{t},\mathfrak{T}_{\mathfrak{t}},z}(z')| \lesssim |z'-z|^{\alpha_{\mathfrak{t}}}, \quad \alpha_{\mathfrak{t}} \geq \alpha_{\tau}.
\end{equs}
We recommend the reader the following references \cite{FrizHai,BaiHos} for surveys on Regularity Structures.
From \cite{reg}, the key point for getting the solution is the construction of the integrals $ u_{\mathfrak{t},z,\tau} $ that may require a renormalisation procedure. The general definition for these terms has been given in \cite{BHZ} that constructed the renormalised integrals with the help of two Hopf algebras in co-interaction. The convergence of such renormalised integrals has been proved in \cite{CH16}. A Butcher series formalism for the expansion \eqref{expansion_solution} has been initiated in \cite{BCCH} and it has been used for understanding how the renormalisation of the iterated integrals changed the equation we started with by adding counter-terms. The method of \cite{BCCH} has been simplified in \cite{BaiHos} and in \cite{BB21,BB21b} a simple proof of the renormalised equation has been given that also works in the context of non-translation invariance where renormalisation constants are replaced by functions. The present work aims to push forward the formalism of $ \cite{BCCH} $ by introducing the notion of Regularity Structures Butcher series. Then, one gets an equivalent of \eqref{composition_substitution} and the following correspondence: 
\begin{itemize}
	\item Recentering iterated integrals is similar to the composition of Regularity Structures B-series.
	\item Renormalisation of iterated integrals is similar to the substitution of Regularity Structures B-series.
\end{itemize}
To make this statement precise, we introduce two types of B-series. We first need to rewrite \eqref{main_equation_bis} in integral form for $ \mathfrak{t} \in \Lab_+ $: 
\begin{equs}
	u_{\mathfrak{t}} = K_{\mathfrak{t}} * v_{\mathfrak{t}}, \quad v_{\mathfrak{t}} = \sum_{\mathfrak{l} \in \Lab_-} F_{\mathfrak{t}}^{\mathfrak{l}}({\bf{u}}) \xi_{\mathfrak{l}}.
\end{equs}
One can define two Regularity Structures B-series that will approximate well the solution $ u_{\mathfrak{t}} $ and the right hand side $ v_{\mathfrak{t}} $ of the equation \eqref{main_equation}:
\begin{equs}
	u_{\mathfrak{t}} \approx	B_+(\alpha,F, \mathfrak{t}, {\bf{u}}(x)), \quad v_{\mathfrak{t}} \approx B_-(\beta,F_{\mathfrak{t}}, {\bf{u}}(x))
\end{equs}
where the two series introduced have a similar form as $ \eqref{B-series-form} $. The main difference is the index set that now will be decorated trees as introduced in \cite{BHZ} for coding stochastic iterated integrals. The series $ B_+ $ (resp. $ B_- $) is indexed by the set of decorated trees $ \mathfrak{T}_+ $ (resp. $ \mathfrak{T} $). The two sets are quite different as $ \mathfrak{T}_+ $ corresponds to classical monomials and to the iterated integrals that describe $ u_{\mathfrak{t}} $. The latter starts with a convolution with the kernel $ K_{\mathfrak{t}} $ whereas the iterated integrals encoded by $ \mathfrak{T} $ contain products with the noises $ \xi_{\mathfrak{l}} $. This is a notable difference in comparison to the classical Butcher series where one does not make this distinction.
One defines new symmetry factors and elementary differentials for these trees. The small term in $ h $ disappears in the definition and now the characters $ \alpha : \CT_+ \rightarrow \mathcal{D}'(\mathbb{R}^{d+1},\mathbb{R}) $ and $ \beta : \CT \rightarrow \mathcal{D}'(\mathbb{R}^{d+1},\mathbb{R})$ take values in Schwartz distributions. In this work, we will consider smooth noises and therefore the characters take values in  $ \mathcal{C}^{\infty}(\mathbb{R}^{d+1},\mathbb{R}) $. Here $ \CT $ and $ \CT_+ $ are the linear span of $ \mathfrak{T} $ and $ \mathfrak{T}_+  $.  The character $ \alpha $ will be small when it is evaluated at a point close to the fixed point $ z $ in the sense that there exist $ (a_{\tau})_{\tau \in \mathfrak{T}} $ such that
\begin{equs}
	|\alpha(\tau)(z')| \lesssim |z'-z|^{a_{\tau}}.
\end{equs}
We extend the definition of the composition $\circ$ and the substitution $ \circ_s$ of B-series to the series $ B_+ $  and $ B_- $  that we denote by
\begin{equs}
	&	B_-(\alpha_1,F, \cdot) \circ B_+(\alpha_2,F, \cdot, {\bf{u}}(z)),
	\\
	&	B_-(\beta_1,\cdot,{\bf{u}}(z)) \circ_s B_-(\beta_2,F, {\bf{u}}(z)).  \end{equs}
Below, we give the main result of this paper which is an extension of \eqref{composition_substitution} to the Regularity Structures case:
\begin{theorem} \label{main_result} One has
\begin{equs}
	&	B_-(\alpha_1,F, \cdot) \circ B_+(\alpha_2,F, \cdot, {\bf{u}}(z)) = 	B_-(\alpha_2 \star_{2} \alpha_1, F, {\bf{u}}(z))
	\\
	&	B_-(\beta_1,\cdot,{\bf{u}}(z)) \circ_s B_-(\beta_2,F, {\bf{u}}(z)) = B_-(\beta_2 \star_1 \beta_1 ,F,{\bf{u}}(z)) 
\end{equs}
where the product $ \star_2 $ (resp. $ \star_1 $) is a deformation of $ \star_{\text{\tiny{BCK}}}  $ ( $ \star_{\text{\tiny{EC}}}  $) in the sense of \cite{BM22}.
  \end{theorem}
This theorem is split into two theorems in Section~\ref{sec::4} of this paper: Theorem \ref{main_composition} and Theorem \ref{thm_substition}. In these theorems, we will consider a slightly more general context where $ B_+(\alpha_2,F,(\mathfrak{t},m),\mathbf{u}(z)) $ depends on $ (\mathfrak{t},m) = \Lab_+ \times \mathbb{N}^{d+1} $. This corresponds to local expansion for $ D^{m} u_{\mathfrak{t}} $. Here, the $D^{m}$ are derivatives. 
The products $ \star_1 $ and $ \star_2 $ have been originally introduced in \cite{BHZ} for describing renormalised models in the context of Regularity Structures. The product $ \star_1 $ is used for the renormalisation by extracting and contracting subforests. This renormalisation is essential for proving the general convergence theorem for renormalised model in \cite{CH16}. It is very much inspired from the BPHZ schemes developed for renormalising Feynman diagrams in Quantum Field Theory (see \cite{BP,Hepp,Zim}). 

Before outlining the content of the paper, we give a couple of remarks on the scope and the perspectives offered by the Regularity Structures B-series.
\begin{remark}
B-series are used in numerical analysis for finding numerical schemes having some structures preserving properties. For example, one has algebraic identities that the character $ \alpha $ in \eqref{B-series-form} must satisfy for having a symplectic scheme. Also, symmetric schemes are well-understood. We recommend the book \cite{HLW} for an overview of these properties. Let us mention that decorated trees have been used in \cite{BS} for deriving low regularity schemes in the context of dispersive equations with rough initial data. These decorated trees are very similar to the ones used for singular SPDEs. The scheme is based on a tree series and recently a large class of symmetric schemes has been found in \cite{ABMS23}. One can wonder how one can extract ideas from B-series to understand better symmetries of solutions of singular SPDEs. So far the main symmetries obtained are chain rule and Itô Isometry for the stochastic geometric heat equations in \cite{BGHZ} and Gauge invariance for the stochastic Yang-Mills equation in two and three dimensions in \cite{CCHS22, CCHS23}. One can hope to try to connect the search of symmetries to Regularity Structures B-series. 
	\end{remark}
\begin{remark}
	A natural extension of B-series in numerical analysis is the Lie Butcher-series developed in \cite{ML08,ML13}. These are series that take  values in a Lie group. Composition and substitution have been understood in this case via Hopf algebras on planar trees. These Hopf algebras come from a post-Lie structure. A post-Lie product was first observed at the level of partition of posets (see \cite{Val}). Substitution has even been understood at an operadic level in \cite{Rahm1} when one uses the post-Lie operad (see also \cite{ChaLiv,CH04} for an operadic interpretation in the case of a pre-Lie product on non-planar rooted trees). One can wonder if such an extension is possible for singular SPDEs. The extension of the algebraic part is performed in \cite{Rahm2} but the analytical operations like substitution and composition are not defined. One is missing an example of a singular SPDE where such a structure could be used.
	Recently, Regularity Structures on manifolds have been worked out in \cite{HS23}. Regularity Structures B-series could be also described  in this context.
	\end{remark}
\begin{remark}
	One can think about replacing the indexed set of decorated trees by something else. Indeed, multi-indices have been introduced in the context of quasi-linear SPDEs in \cite{OSSW} and they inspired a non-diagrammatic proof for the convergence of the model in \cite{LOTT}. One has similar Hopf algebras and products as in the case of decorated trees. See \cite{LOT} for the algebraic construction  and \cite{BL23} where the authors define general multi-indices and the renormalisation of their associated models.
	The first step will be definitely to write numerical schemes with the help of multi-indices for ordinary differential equations as the theory is missing in this case. This step has been completed in \cite{BEFH24} after the first arxiv version of the present paper. The proofs of composition and substitution of multi-indice B-series are directly inspired by the proofs of this paper. Rough paths multi-indices have been treated in \cite{L23}. One of the main advantage of multi-indices is to provide a compression  of the coding as several decorated trees 	are now encoded via the same combinatorial object.
	\end{remark}
\begin{remark}
	Regularity Structures is not the only efficient tool for treating singular SPDEs. Paracontrolled calculus has been introduced as the same time in \cite{GIP}. For a while, it was limited to small expansions but with \cite{BB19} one can have higher expansions. 
	Algebraic structures for dealing with  the para-linearisation and all the commutators produced are not well-understood although the correspondence between paracontrolled and Regularity Structures has been firmly established in \cite{BH21,BH21b}. One of the main advantages of paracontrolled calculus is that it is widely used for dispersive SPDEs. Therefore, an understanding of the solution via a Butcher-series could provide
	a new perspective on this expansion.
	\end{remark}

Let us outline the paper by summarising the content of its sections. In Section~\ref{sec::2}, we recall the definition of Butcher-series for describing numerical methods for ordinary differential equations. After defining the composition with a smooth function in \eqref{composition_B_series_smooth}, we introduce the composition between two B-series \eqref{composition_B_series} that is essential for composing numerical methods. 
Theorem~\ref{composition_Butcher} makes the connection between this composition and the Butcher-Connes-Kreimer coproduct defined by \eqref{def_BCK}.
The second main operation is substitution defined in \eqref{substitution_def}. It is used in the context of backward error analysis. The second main theorem is Theorem~\ref{comp_B_series} that connects this operation with extraction-contraction coproduct given in \eqref{EC_coproduct}. We cloture the section with a presentation of the co-interaction \eqref{cointeraction_BCK_EC} and the Remark~\ref{remark_rp} on the extension of these series to Rough Paths.  

In Section~\ref{sec::3}, we present the main combinatorial objects for describing solutions of singular SPDEs that are decorated trees with both decorations on the vertices and on the edges. Then, we proceed by the introduction of two operations on decorated trees that are deformed grafting products in \eqref{deformed_grafting_a} and increasing node decorations in \eqref{increasing_node}. These two operations are enough for defining an associative product $ \star_2 $ in \eqref{def_star_2} which is the dual of the coproduct $ \Delta_2 $ introduced in \cite{BHZ}.  The second product $ \star_1 $ in this section is defined  from the coproduct $ \Delta_1 $.
This coproduct is introduced via $ \Delta_2 $ in \eqref{recursive_def_delta1}. A direct recursive definition of $ \star_2 $ from $ \star_1 $ is also given. The two crucial properties of this section are the multiplicative property of the map  $ M^*_{\alpha}$ in \eqref{productM} wich is a consequence of the co-interaction bewteen $ \Delta_1 $ and $ \Delta_2 $. One also gets a second property for $ M^*_{\alpha}$  which is \eqref{consequence_M}. The presentation  of these products draws its inspiration from \cite{BM22}.

In Section~\ref{sec::4}, we introduce the new concept of Regularity Structures B-series in Definition \ref{Def_B_-}. They are series built upon the decorated trees introduced before. One has to deal with two types of B-series that correspond to the solution of the SPDEs considered or an expansion of the right hand side of the equation (see Definition~\ref{def_B_+}). We define composition in \eqref{compostion_B_Series} from the definition used in \cite{reg,BCCH} for composing a modelled distribution with a smooth function. This can be viewed as a generalisation of the composition of B-series: one has to move from a Taylor expansion in time for ODE to a Taylor-type expansion in space-time for PDE. The main theorem for composition is Theorem \ref{main_composition} that establishes the relation with the product $ \star_2 $.
Then, we introduce substitution in \eqref{substitution_B_series} and the connection with the product $ \star_1 $ is given in
 Theorem \ref{thm_substition}. We finish the section with Theorem \ref{thm_root_substitution} which proposes a link between substitution and composition. 
Let us mention that the proof of Theorem \ref{main_composition} is a generalisation of the proof given \cite{BB21} for the renormalised equation.
The proof of Theorem \ref{thm_substition} deviates from the one given for the renormalised equation in \cite{BCCH} by using a morphism property for the product $ \star_2 $.
The last Theorem \ref{thm_root_substitution} is connected to one of the key arguments in the short proof of the renormalised equation given in \cite{BB21} which is to use a character multiplicative for the tree product that carries part of the renormalisation. Then, the missing renormalisation can be interpreted as a composition.

\subsection*{Acknowledgements}

{\small
Y. B. gratefully acknowledges funding support from the European Research Council (ERC) through the ERC Starting Grant Low Regularity Dynamics via Decorated Trees (LoRDeT), grant agreement No.\ 101075208.
Y. B. thanks the Max Planck Institute for Mathematics in the Sciences (MiS) in Leipzig for having supported his research via a long stay in Leipzig from January to June 2022. The research course " Decorated Trees for singular (stochastic) PDEs" given during that stay puts the basis of this work.     Y. B. thanks the participants of the workshop "Hopf algebras, operads, deformations for singular dynamics"  funded by the ANR via the project LoRDeT (Dynamiques de faible régularité via les arbres décorés) from the projects call T-ERC\_STG. Their talks give some inspirations to this work. Finally, Y. B. thanks the Centre for Advanced Study (CAS) at The Norwegian Academy of Science and Letters for the nice environment offered during a long stay in September 2023 for the programme "Signature for Images" when most of this work was written. 
} 

\section{B-series in Numerical Analysis}

\label{sec::2}
 
Given an ordinary differential equation (ODE)
\begin{equs}\label{eq: ode}
	\partial_t y = F(y), \quad y(0) = y_0 \in \mathbb{R}^{d}\;,
\end{equs}
a B-series associated to it is a discrete time-stepping method $ y_k \mapsto y_{k+1}  $ given by an expansion over rooted trees:
\begin{equs}\label{eq: B-series formula}
	y_{k+1} - y_k = \sum_{\tau \in \mathbf{T}}  h^{|\tau|} \alpha(\tau) \frac{F[\tau](y_k)}{S(\tau)},
\end{equs}
where $\mathbf{T}$ is the set of rooted combinatorial trees. In the sequel, we also consider forests on these trees as unordered sets of trees. The forest product could be seen as the disjoint union of two forests. In the right hand side of~\eqref{eq: B-series formula}, $h > 0$ is the step size, $ |\tau| $ denotes the number of nodes of $\tau$, $ S(\tau) $ denotes its symmetry factor. It is defined for $ \tau
= \prod_{i=1}^n \CI(\tau_i)^{\beta_i} $
where the $ \tau_i $ are disjoint rooted trees by:
\begin{equs}
	S(\tau) = \prod_{i=1}^n \beta_i ! \, S(\tau_i)^{\beta_i}.
\end{equs}
The notation  $ \prod_{i=1}^n \CI(\tau_i)^{\beta_i} $ means that the $ \tau_i $ are connected to the same root via an edge connecting each root of the $ \tau_i $ to a new root. The map $ \alpha : \mathbf{T} \rightarrow \mathbb{R}$ is a function which determines the time stepping method.
	One natural choice of $\alpha$ is setting
	 \begin{equs} \label{choice_Taylor}
	 	\alpha(\tau) = 1/\gamma(\tau)
	 \end{equs}
  where $\gamma(\tau)$ is given for $ \tau = \prod_{i=1}^n \CI(\tau_i) $ by 
	\begin{equs}
		\gamma(\bullet) = 1, \quad \gamma(\tau) = |\tau| \prod_{j=1}^{n}\gamma(\tau_{j}).
	\end{equs}
An alternative notation is to use the $ \mathcal{B}_+ $ operator going from the forests into the rooted trees which is such that
\begin{equs}
	\mathcal{B}_+(\tau_1 \cdots \tau_n) = \prod_{i=1}^n \CI(\tau_i)
\end{equs}
where $ \tau_1 \cdots \tau_n $ denotes the forest formed of the trees $ \tau_i $.
	With the choice \eqref{choice_Taylor}, the formula given in~\eqref{eq: B-series formula} coincides with the Taylor expansion  of the exact solution to the initial value problem. 
The elementary differentials $ F(\tau) $ are defined recursively for $ \tau = \prod_{i=1}^n \CI(\tau_i)$ by  
\begin{equs}
	F[\bullet] = F, \quad 	F[\tau](y)
	=
	F^{(n)}(y)
	\prod_{j=1}^{n} F\big[
	\tau_j
	\big](y)\;.
\end{equs}
One uses the following notation for describing a B-series:
\begin{equs}
	B(\alpha,F,h) = \alpha(\mathbbm{1}) \id + \sum_{\tau \in\mathbf{T} } h^{|\tau|} \alpha(\tau) \frac{F[\tau]}{S(\tau)}
\end{equs}
where $ \mathbbm{1} $ is the empty tree and the arguments of $ B $ gather the main parameters of the numerical method.
One important notion about these B-series is the composition with a smooth function $ G : \mathbb{R}^d \rightarrow \mathbb{R}^d $ defined up  to order $ m $ by:
\begin{equs} \label{composition_B_series_smooth}
	\left( G \circ_m	B(\alpha,F,h) \right)(u) = \sum_{k \leq m } \left(  B(\alpha, F,h)(u) - u \right)^k \frac{G^{(k)}(u)}{k!}
\end{equs}
This truncation is needed for having a finite series and it is natural in a numerical analysis context where the order of the approximation is known. In the sequel, we will assume that the parameter $ m $ has been fixed for the composition and we will not mention  it. Therefore, we will denote the product as $ \circ $ instead of $ \circ_m $.
Then, using this definition, one can define the composition of B-series (see \cite{Butcher72,HW74}), that is the following term:
\begin{equs}
	\label{composition_B_series}
		B(\alpha,F,h) \circ B(\beta,F,h) = B(\beta,F,h)  + \sum_{\tau \in\mathbf{T} } h^{|\tau|}  \frac{\alpha(\tau)}{S(\tau)} F[\tau] \circ  B(\beta, F,h)
	\end{equs}
where we have assumed that $ \alpha(\mathbbm{1}) = 1 $.
We want to rewrite \eqref{composition_B_series} in a simple form as a B-series. To complete this purpose, one has to first introduce
  the Butcher-Connes-Kreimer coproduct $ \Delta_{\text{\tiny{BCK}}} $ defined on rooted trees by
\begin{equs} \label{def_BCK}
	\Delta_{\tiny{\text{BCK}}} \tau = \sum_{ c \in \scriptsize{\text{Adm}}(\tau) } R^c(\tau) \otimes P^{c}(\tau) + \tau \otimes \mathbbm{1}.
\end{equs}
where $ \text{Adm}(\tau) $ are admissible edge cuts that are subsets of edges of $ \tau $ which contain at most one edge from
each path in $ \tau $ that starts in the root and ends in a leaf. Here, we have used $ P^c(\tau)$ to denote the pruned forest that is formed by collecting all
the edges above the cut. The term  $ R^c(\tau)$ corresponds to the "trunk", that is the subforest formed
by the edges not lying above the ones upon which the cut was performed. The coproduct  $ \Delta_{\text{\tiny{BCK}}} $ is extended multiplicatively for the forest product. Given two characters $ \alpha, \beta $ that are multiplicative maps for the forest product into $ \mathbb{R} $, we set:
\begin{equs} \label{convolution_product_BCK}
\alpha \star_{\tiny{\text{BCK}}} \beta = \left( \alpha \otimes \beta \right)	\Delta_{\tiny{\text{BCK}}}
\end{equs}
where we have omitted the multiplication in the end as we identify $ \mathbb{R} \otimes \mathbb{R} $ with $ \mathbb{R} $. The product $ \star_{\text{\tiny{BCK}}} $ is called a convolution product defined from the coproduct $ \Delta_{\text{\tiny{BCK}}} $. Equipped with the coproduct $ \Delta_{\text{\tiny{BCK}}}  $ and the forest product, the linear span of forests forms the Butcher-Connes-Kreimer Hopf algebra that we denote by $ \mathcal{H}_{\text{\tiny{BCK}}} $.

\begin{theorem} \label{composition_Butcher}
	Let $ \alpha, \beta $ be characters of  $ \mathcal{H}_{\text{\tiny{BCK}}} $ then the composition of B-series satisfies
	\begin{equs}
	B(\beta,F,h) \circ B(\alpha, F,h) = B( \beta \star_{\tiny{\text{BCK}}} \alpha, F,h).
	\end{equs}
	\end{theorem}

A second main operation has been defined on B-series which is the substitution of a B-series into another (see \cite{CHV05,CHV07}). It is given by
\begin{equs} \label{substitution_def}
		B(\beta, F,h) \circ_s B(\alpha,F,h) = 	B(\beta, \frac{1}{h} B(\alpha,F,h),h). 
\end{equs}
It boils down to replacing the vector field $ F $ by a B-series with the constraint that $ \alpha $ should be zero on $ \mathbbm{1} $. We need an extra coproduct for building a new convolution product that will allows us to obtain a more precise description of this operation.

This new  coproduct  $ \Delta_{\text{\tiny{EC}}} $ is defined by extraction and contraction of subtrees. Given a rooted tree $ \tau $, one has
\begin{equs} \label{EC_coproduct}
	\Delta_{\tiny{\text{EC}}} \tau = \sum_{\tau_1 \cdots \tau_n \subset \tau}  \tau_1 \cdots \tau_n  \otimes \tau / \tau_1 \cdots \tau_n
\end{equs}
where the sum is performed on all spanning subforests $ \tau_1 \cdots \tau_n $ of $ \tau $. A spanning subforest of $ \tau $ is composed of subtrees of $ \tau $ such that each vertex of $ \tau $ is contained exactly in one of these trees.
 The term $ \tau / \tau_1 \cdots \tau_n  $ corresponds to the contraction  of  each subtree to a
single vertex. This coproduct is extended multiplicatively to the forest product. Equipped with the coproduct $ \Delta_{\text{\tiny{EC}}}  $ and the forest product, the linear span of forests formed the extraction-contraction bialgebra that we denote by $ \mathcal{H}_{\text{\tiny{EC}}} $. As before, given two characters $ \alpha, \beta $, we define a convolution product out of $ \Delta_{\text{\tiny{EC}}}  $:
\begin{equs}
	\alpha \star_{\tiny{\text{EC}}} \beta = \left( \alpha \otimes \beta \right)	\Delta_{\tiny{\text{EC}}}
\end{equs} 

\begin{theorem} \label{comp_B_series}
Let $ \alpha, \beta : \mathbf{T} \oplus \mathbb{R} \mathbbm{1}  \rightarrow \mathbb{R} $ be linear maps satisfying $ \alpha(\mathbbm{1}) = 0 $. Extend $ \alpha $ to $ \CH_{\text{\tiny{EC}}} $  multiplicatively.
Then the substitution of B-series satisfies
\begin{equs}
	B(\beta, \frac{1}{h} B(\alpha,F,h),h) = B(\alpha \star_{\text{\tiny{EC}}} \beta, F,h).
\end{equs}
\end{theorem}

It has been observed in \cite{CEM} that the  two coproducts $\Delta_{\text{\tiny{BCK}}}$ and  $\Delta_{\text{\tiny{EC}}}$ are in co-interaction which means that they satisfy the following identity:
	\begin{equs} \label{cointeraction_BCK_EC}
		\mathcal{M}^{(13)(2)(4)}\left(  \Delta_{\tiny{\text{EC}}} \otimes  \Delta_{\tiny{\text{EC}}}\right)
		\Delta_{\tiny{\text{BCK}}} = \left( \id \otimes \Delta_{\tiny{\text{BCK}}} \right) \Delta_{\tiny{\text{EC}}}, 
	\end{equs}
	where for $ \tau_1, \tau_2, \tau_3, \tau_4 \in \CF $
	\begin{equs}
		\mathcal{M}^{(13)(2)(4)}\left( \tau_1 \otimes \tau_2 \otimes  \tau_3 
		\otimes \tau_4 \right) =  \tau_1 \cdot  \tau_3 \otimes 
		\tau_2 \otimes \tau_4.
	\end{equs}
Here $ \mathcal{F} $ denotes the linear span of forests. 
At a character level, the previous identity is
\begin{equs}
	(\beta \star_{\tiny{\text{EC}}}	\alpha_1) \star_{\tiny{\text{BCK}}} (\beta \star_{\tiny{\text{EC}}} \alpha_2) =  \beta \star_{\tiny{\text{EC}}} ( \alpha_1 \star_{\tiny{\text{BCK}}}  \alpha_2) 
\end{equs}
where $ \beta, \alpha_1,\alpha_2 $ are characters on $ \CF $.
\begin{remark} \label{remark_rp}
One can define B-series for more than one non-linearity and consider a finite  family $ (F_a)_{a \in A}$  of those. Here, $ A $ is a finite set. This formalism is useful when one works with systems. 
\begin{equs}\label{eq: ode}
	\partial_t y_a = F_a({\bf y}), \quad y_a(0) = y_{a,0} \in \mathbb{R}\;,
\end{equs}	
where now $ {\bf y} = (y_a)_{a \in A} $ and the non-linearities $ F_a $ depend on these variables. In terms of notation, the B-series are denoted by 
\begin{equs}
	a \mapsto	B(\alpha,F,a,h) = \alpha(\mathbbm{1}) \id + \sum_{\tau \in \mathbf{T}_a } h^{|\tau|} \alpha(\tau) \frac{F[\tau]}{S(\tau)} 
\end{equs}
where now $ \mathbf{T}_a  $ are vertex-decorated trees whose root is decorated by $ a $. The definition of the symmetry factor and the elementary differential can be extended to this new case. For example, one has
\begin{equs}
		F[\mathcal{B}^a_+(\tau_1 \cdots \tau_n)](y)
	=
	F_a^{(n)}(y)
	\prod_{j=1}^{n} F\big[
	\tau_j
	\big](y)\;
	\end{equs}
Here, $ \mathcal{B}^a_+ $ is defined similarly as $ \mathcal{B}_+ $. Now the new root created is decorated by $a$.
 These B-series can be viewed as  functions on the parameter $a$. Another source of complexity can come from adding more drivers or noises which is the case in the context of Rough Paths theory where this formalism of B-series has been extensively used.
In~\cite{Gub06}, the author introduced similar tree expansions instead of words (see \cite{Lyons98,Gubinelli2004}) for solving rough differential equations of the form
\begin{equation}
	\label{eq:CODE}
	dY_t = \sum_{\mfl=0}^m F_{\mfl}(Y_t)dX^{\mfl}_t,
\end{equation}
where the $F_\mfl$ are vector fields on $\mathbb{R}^{d}$ and $X:[0,T] \rightarrow \mathbb{R}^{m+1}$ is a driving signal. Non-noisy term can be encoded by using the convention that $X_{0}(t) = t$. 
The solution to~\eqref{eq:CODE} is again given by a tree expansion
\begin{equation}
	\label{eq:CODEsolution}
	Y_t =  Y_s + \sum_{\tau \in \mathbf{T}_A} \frac{F[\tau](Y_s)}{S(\tau)} X^{\tau}_{st}\;,
\end{equation}
where $ \mathbf{T}_A  $ are vertex-decorated trees whose vertices decorations belong to $A = \{0,\ldots,m\}$. 
In the sequel, we will proceed similarly to these two extensions by looking at decorated trees with decorations on both the edges and the vertices that will encode both systems and multiple drivers or noises. 
The composition rule for these B-series will be different as now one has to consider space-time monomials. These monomials appear inside a decorated tree via decorations on the nodes.   
\end{remark}

\section{Algebraic structures on decorated trees}
\label{sec::3}

Decorated trees as introduced in
\cite{BHZ} are described in the following way. We suppose given two finite sets $  \mathfrak{L}_{+} $ and $ \mathfrak{L}_{-} $. The main motivation behind these sets is to encode a system of (stochastic) partial differential equations that are given by:
\begin{equs} \label{main_equation}
	\partial_t u_{\mathfrak{t}} - \mathcal{L}_{\mathfrak{t}} u_{\mathfrak{t}}  = \sum_{\mathfrak{l} \in \Lab_-} F_{\mathfrak{t}}^{\mathfrak{l}}({\bf{u}}) \xi_{\mathfrak{l}}, \quad \mathfrak{t} \in \Lab_+
\end{equs}
where the $ \xi_{\mathfrak{l}} $ are space-time noises, $ {\bf{u}} $ is the collection of the $ u_{\mathfrak{t}} $ and some of their derivatives. The terms  $ \mathcal{L}_{\mathfrak{t}} $ are differential operators and one can rewrite \eqref{main_equation} into a mild form given by
\begin{equs} \label{mild_formulation}
 u_{\mathfrak{t}}   = K_{\mathfrak{t}} * \left(  \sum_{\mathfrak{l} \in \Lab_-} F_{\mathfrak{t}}^{\mathfrak{l}}({\bf{u}}) \xi_{\mathfrak{l}} \right), \quad \mathfrak{t} \in \Lab_+
\end{equs}
where the $ K_{\mathfrak{t}} $ are the kernels associated to the operators $ (\partial_t - \mathcal{L}_{\mathfrak{t}}) $ and $ * $ is the space-time convolution. Therefore, on our decorated trees, we set  $ \mathcal{D} := \Lab \times \mathbb{N}^{d+1}$ where $ \Lab = \Lab_+ \cup \Lab_- $ to be the set of edge decorations. These symbols represent a convolution with the kernels $ D^m K_{\mathfrak{t}} $ ($D^m$ are derivatives) for $ (\mathfrak{t},m) \in \Lab_+ \times \mathbb{N}^{d+1} $ and a noise term $ \xi_{\mathfrak{l}}  $ for $ (\mathfrak{l},m) \in \Lab_- \times \mathbb{N}^{d+1}  $.  We suppose that $0 \in \Lab_-$ and $ \xi_0 = 1 $.  Decorated trees over $ \mathcal{D} $ are of the form  $T_{\Labe}^{\Labn} =  (T,\Labn,\Labe) $ where $T$ is a non-planar rooted tree with node set $N_T$, edge set $E_T$ and root $ \varrho_T $. We suppose that the edges decorated by $ (\mathfrak{l},m) \in \Lab_- \times \mathbb{N}^{d+1}$    are terminal edges which means that one of their extremities is a leaf. Also from every node, only one edge with such a decoration can be connected to this node. We assume that $ m=0 $ and the node decoration of that leaf is zero. The maps $\Labn : N_T \rightarrow \mathbb{N}^{d+1}$ and $\Labe : E_T \rightarrow \mathcal{D}$ are node, respectively edge, decorations.
The tree product is defined by 
\begin{equation*} 
	(T,\Labn,\Labe) \cdot  (T',\Labn',\Labe') 
	= (T \cdot T',\Labn + \Labn', \Labe + \Labe')\;, 
\end{equation*} 
where $T \cdot T'$ is the rooted tree obtained by identifying the roots of $ T$ and $T'$. The sums $ \Labn + \Labn'$ mean that decorations are added at the root and extended to the disjoint union by setting them to vanish on the other tree.  
We make the connection with  symbolic notation introduced in the previous part.

\begin{enumerate}
	\item[--] An edge decorated by  $ a = (\mathfrak{t},m) \in \mathcal{D} $ with $ \mathfrak{t} \in \Lab_+ $  is denoted by $ \CI_{a} $. The symbol $  \CI_{a} $ is also viewed as  the operation that grafts a tree onto a new root via a new edge with edge decoration $ a $. The new root at hand remains decorated with $0$. 
	
	\item[--]  An edge decorated by $ (\mathfrak{l},0) \in \mathcal{D} $ with $ \mathfrak{l} \in \Lab_- $ is denoted by $  \Xi_{\mathfrak{l}} $.
	
	\item[--] A factor $ X^{\ell}$   encodes a single node  $ \bullet^{\ell} $ decorated by $ \ell \in \mathbb{N}^{d+1}$. We write $ X_i$, $ i \in \lbrace 0,1,\ldots,d\rbrace $, to denote $ X^{e_i}$. Here, $ e_i $ corresponds to the vector of $ \mathbb{N}^{d+1} $ with $ 1 $ in $i$th position and $ 0 $ otherwise. The elements $ X^0 $ and $\Xi_0$ are identified with $ \mathbf{1} $ the tree with one node and no decoration.
\end{enumerate}
Using this symbolic notation, given a decorated tree $ \tau $ there exist decorated trees $ \tau_1, ..., \tau_r $ such that
\begin{equs}
	\tau = X^{\ell}  \Xi_{\mathfrak{l}}  \prod_{i=1}^r \CI_{a_i}(\tau_i)
\end{equs}
where $ \prod_i $ is the tree product, $ \ell \in \mathbb{N}^{d+1} $, $ m, r \in \mathbb{N} $. A tree of the form $ \CI_a(\tau) $ is called a planted tree as there is only one edge connecting the root to the rest of the tree.
Below, we present an example of a decorated tree with its associated iterated integral coming from \eqref{mild_formulation}:
\begin{equs} 
	\tau = X^{\alpha} \Xi_{\mathfrak{l}_1} \CI_a(X^{\beta} \Xi_{\mathfrak{l}_2})  =   \begin{tikzpicture}[scale=0.2,baseline=0.1cm]
		\node at (0,0)  [dot,label= {[label distance=-0.2em]below: \scriptsize  $  \alpha   $} ] (root) {};
		\node at (2,4)  [dot,label={[label distance=-0.2em]right: \scriptsize  $ \beta $}] (right) {};
			\node at (0,8)  [dot,label={[label distance=-0.2em]right: \scriptsize  $  $}] (rightc) {};
		\node at (-2,4)  [dot,label={[label distance=-0.2em]above: \scriptsize  $ $} ] (left) {};
		\draw[kernel1] (right) to
		node [sloped,below] {\small }     (root);
		\draw[kernel1] (right) to
		node [sloped,below] {\small }     (rightc); \draw[kernel1] (left) to
		node [sloped,below] {\small }     (root);
			\node at (1,6) [fill=white,label={[label distance=0em]center: \scriptsize  $ \Xi_{\mathfrak{l}_2} $} ] () {};
		\node at (-1,2) [fill=white,label={[label distance=0em]center: \scriptsize  $ \Xi_{\mathfrak{l}_1} $} ] () {};
		\node at (1,2) [fill=white,label={[label distance=0em]center: \scriptsize  $ a $} ] () {};
	\end{tikzpicture} \equiv X^{\alpha} \xi_{\mathfrak{l}_1}  D^m K_{\mathfrak{t}} * \left( X^{\beta} \xi_{\mathfrak{l}_2} \right)
\end{equs}
where $ X^{\alpha} $ in the previous iterated integral denotes the polynomial function defined on $ \mathbb{R}^{d+1} $ by $ x \in \mathbb{R}^{d+1} \mapsto x^{\alpha} = \prod_{i=0}^d x_i^{\alpha_i} $. One has
$ a = (\mathfrak{t},m) $ and we have put $ \Xi_{\mathfrak{l}_i}$ on the edge to specify that it is decorated by $ (\mathfrak{l}_i,0) $ and to highlight noise type edges. Nodes without decoration mean that their decoration is equal to zero. Pointwise products at the level of the iterated integrals correspond to tree product for decorated trees. We denote the set of decorated trees by $ \mfT $ and its linear span by $ \CT $. We consider the space $ \CT_+ $ which is the linear span of
\begin{equs}
\mathfrak{T}_+ =	\lbrace X^k  \prod_{i=1}^r \CI_{a_i}(\tau_i), \, \tau_i \in \mfT \rbrace.
\end{equs}
The main difference is that we do not have a noise $\Xi_{\mathfrak{l}}$ at the root. The space $\CT_+$ is not a subspace of $\CT$. One has
\begin{equs}
	X^k  \prod_{i=1}^r \CI_{a_i}(\tau_i) \in\CT_+, \quad X^k \Xi_0 \prod_{i=1}^r \CI_{a_i}(\tau_i) \in\CT,
\end{equs}
with the $\tau_i$ belonging to $\mathfrak{T}$.
The space $ \CT_+ $ can be viewed as the symmetric space over terms of the form $ \CI_a(\tau) $ and $ X_i $.  In the sequel, we will consider also a last space denoted by $ \CT_- $ which is the linear span of forests composed of decorated trees in $ \mfT $. The empty forest is denoted by $ \mathbbm{1} $ and the forest product by $\cdot$. For $\bar{\tau} \in F$, we make the following identification: $ \bar{\tau} \cdot \one = \bar{\tau} $. 
We define an inner product on decorated trees as 
\begin{equs}
	\big\langle \sigma, \tau \big\rangle = \delta_{\sigma,\tau} S(\tau). 
\end{equs}
The symmetry factor of the decorated tree $ \tau $ denoted by $ S(\tau) $ is given for
\begin{equs}
	\tau = X^k \Xi_{\mathfrak{l}} \prod_{i=1}^n (\CI_{a_i}(\tau_i))^{m_{i}} 
\end{equs}
where the $ \CI_{a_i}(\tau_i) $ are disjoint, by
\begin{equs} \label{symmetry_factor}
	S(\tau) = (k!) \prod_{i=1}^n (m_i !) S(\tau_i)^{m_i}
\end{equs}

To performe expansions with decorated trees, one has to introduce various combinatorial operations on them. We start by considering a family of products $(\widehat{\curvearrowright}^a)_{a \in \CD_+}$ where $ \CD_+ = \Lab_+ \times \mathbb{N}^{d+1} $  defined by
\begin{equs} \label{deformed_grafting_a}
	\sigma \widehat{\curvearrowright}^a \tau:=\sum_{v\in N_{\tau}}\sum_{\ell\in\mathbb{N}^{d+1}}{\Labn_v \choose \ell} \sigma  \curvearrowright_v^{a-\ell}(\uparrow_v^{-\ell} \tau),
\end{equs}
where $ \Labn_v \in \mathbb{N}^{d+1}$ denotes  the decoration at the vertex $ v $. For $ a =(\mathfrak{t},m) $, one has  $ a-\ell = (\mathfrak{t},m-\ell) $. If $ m-\ell $ does not take value in $ \mathbb{N}^{d+1} $, then the decorated tree with such a decoration is set to be zero. The operator $ \uparrow_v^{-\ell} $ is defined  as subtracting $ \ell $ from the node decoration of $ v $, $\sigma \curvearrowright^a_v \tau$ is obtained by grafting the tree $\sigma$ onto the tree $\tau$ at vertex $v$ by means of a new edge decorated by $a\in \CD_+$. Grafting onto a leaf connected to the rest of the tree via a noise-type edge, that is an edge decorated by $ (\Xi,0) $, is forbidden.
The family of grafting products $ (\widehat{\curvearrowright}^a  )_{a \in \CD_+} $ forms a multi-pre-Lie algebra, in the sense that they satisfy the following identities:
\begin{equs}
	\left( \tau_1  \widehat{\curvearrowright}^a \tau_2 \right)  \widehat{\curvearrowright}^b \tau_3 - 	 \tau_1  \widehat{\curvearrowright}^a (  \tau_2  \widehat{\curvearrowright}^b \tau_3 ) = 	\left( \tau_2  \widehat{\curvearrowright}^b \tau_1 \right)  \widehat{\curvearrowright}^a \tau_3 - 	 \tau_2  \widehat{\curvearrowright}^b (  \tau_1  \widehat{\curvearrowright}^a \tau_3 )
\end{equs}
where the $ \tau_i $ are decorated trees and $ a,b $ belong to $ \CD_+ $.
This is an extension of the notion of a pre-Lie product (which is recovered when the family is reduced to one element) and was first introduced in \cite[Prop. 4.21]{BCCH}.
We use the short hand notation 
\begin{equs}
	\CI_{a}(\sigma) \, \widehat{\curvearrowright} \, \tau:= \sigma \, \widehat{\curvearrowright}^{a} \, \tau.
\end{equs}
We extend the product $ \widehat{\curvearrowright}  $ to the simultaneous grafting of $ \prod_{i=1}^n  \mathcal{I}_{a_i}(\sigma_i) $ onto $ \tau $, denoted by $\prod_{i=1}^n  \mathcal{I}_{a_i}(\sigma_i) \, \widehat{\curvearrowright} \, \tau$. It means that the $\sigma_i$ are simultaneous grafted onto $ \tau $ via the product $ \widehat{\curvearrowright}^{a_i} $.
Below, we provide an example illustrating the product $ \widehat{\curvearrowright} $:
\begin{equs} 
\begin{tikzpicture}[scale=0.2,baseline=0.1cm]
	\node at (0,0)  [dot,label= {[label distance=-0.2em]below: \scriptsize  $     $} ] (root) {};
	\node at (0,8)  [dot,label= {[label distance=-0.2em]above: \scriptsize  $    $} ] (center) {};
	\node at (2,4)  [dot,label={[label distance=-0.2em]right: \scriptsize  $ \alpha $}] (right) {};
	\draw[kernel1] (right) to
	node [sloped,below] {\small }     (root);
	\draw[kernel1] (center) to
	node [sloped,below] {\small }     (right); 
	\node at (1,2) [fill=white,label={[label distance=0em]center: \scriptsize  $ a $} ] () {};
	\node at (1,6) [fill=white,label={[label distance=0em]center: \scriptsize  $ \Xi_{\mathfrak{l}_1} $} ] () {};
\end{tikzpicture}  	 \widehat{\curvearrowright} \begin{tikzpicture}[scale=0.2,baseline=0.1cm]
		\node at (0,0)  [dot,label= {[label distance=-0.2em]right: \scriptsize  $   \gamma  $} ] (root) {};
		\node at (0,8)  [dot,label= {[label distance=-0.2em]above: \scriptsize  $    $} ] (center) {};
		\node at (2,4)  [dot,label={[label distance=-0.2em]right: \scriptsize  $ \beta $}] (right) {};
		\draw[kernel1] (right) to
		node [sloped,below] {\small }     (root);
		\draw[kernel1] (center) to
		node [sloped,below] {\small }     (right); 
		\node at (1,2) [fill=white,label={[label distance=0em]center: \scriptsize  $ b $} ] () {};
		\node at (1,6) [fill=white,label={[label distance=0em]center: \scriptsize  $ \Xi_{\mathfrak{l}_2} $} ] () {};
	\end{tikzpicture}   = \sum_{\ell \in \mathbb{N}^{d+1} }{\gamma \choose \ell} \begin{tikzpicture}[scale=0.2,baseline=0.1cm]
	\node at (0,0)  [dot,label= {[label distance=-0.2em]below: \scriptsize  $  \gamma - \ell   $} ] (root) {};
	\node at (-3,4)  [dot,label= {[label distance=-0.2em]right: \scriptsize  $  \alpha   $} ] (center) {};
	\node at (-3,8)  [dot,label= {[label distance=-0.2em]above: \scriptsize  $    $} ] (centerc) {};
	\node at (3,4)  [dot,label={[label distance=-0.2em]right: \scriptsize  $ \beta $}] (right) {};
	\node at (3,8)  [dot,label={[label distance=-0.2em]above: \scriptsize  $ $} ] (rightc) {};
	\draw[kernel1] (right) to
	node [sloped,below] {\small }     (root); 
	\draw[kernel1] (center) to
	node [sloped,below] {\small }     (root);
	\draw[kernel1] (center) to
	node [sloped,below] {\small }     (centerc);
	\draw[kernel1] (right) to
	node [sloped,below] {\small }     (rightc); 
	\node at (3,6) [fill=white,label={[label distance=0em]center: \scriptsize  $ \Xi_{\mathfrak{l}_2} $} ] () {};
	\node at (-3,6) [fill=white,label={[label distance=0em]center: \scriptsize  $ \Xi_{\mathfrak{l}_1} $} ] () {};
	\node at (1.5,2) [fill=white,label={[label distance=0em]center: \scriptsize  $ b $} ] () {};
	\node at (-1.5,2) [fill=white,label={[label distance=0em]center: \scriptsize  $ a - \ell  $} ] () {};
\end{tikzpicture} +  \sum_{\ell \in \mathbb{N}^{d+1} }{\beta \choose \ell} \, \, \begin{tikzpicture}[scale=0.2,baseline=0.1cm]
		\node at (0,0)  [dot,label= {[label distance=-0.2em]right: \scriptsize  $  \gamma   $} ] (root) {};
		\node at (0,8)  [dot,label= {[label distance=-0.2em]above: \scriptsize  $    $} ] (center) {};
		\node at (4,8)  [dot,label= {[label distance=-0.2em]right: \scriptsize  $ \alpha   $} ] (centerl) {};
		\node at (2,4)  [dot,label={[label distance=-0.2em]right: \scriptsize  $ \beta - \ell $}] (right) {};
			\node at (2,12)  [dot,label= {[label distance=-0.2em]right: \scriptsize  $    $} ] (centerlc) {};
			\draw[kernel1] (centerl) to
			node [sloped,below] {\small }     (centerlc);
		\draw[kernel1] (right) to
		node [sloped,below] {\small }     (root);
		\draw[kernel1] (center) to
		node [sloped,below] {\small }     (right); 
		\draw[kernel1] (centerl) to
		node [sloped,below] {\small }     (right); 
		\node at (1,2) [fill=white,label={[label distance=0em]center: \scriptsize  $ b $} ] () {};
		\node at (3,6) [fill=white,label={[label distance=0em]center: \scriptsize  $ \qquad a - \ell  $} ] () {};
	\node at (1,6) [fill=white,label={[label distance=0em]center: \scriptsize  $ \Xi_{\mathfrak{l}_2} $} ] () {};
	\node at (3,10) [fill=white,label={[label distance=0em]center: \scriptsize  $ \Xi_{\mathfrak{l}_1} $} ] (rightrr) {};
	\end{tikzpicture}.
\end{equs}
Another important operation we will need to define on the space of trees is $ \uparrow^{i} $:
\begin{equs} \label{increasing_node}
	\uparrow^{i} \tau  = \sum_{v \in N_{\tau}} \uparrow^{e_i}_v \tau.  
\end{equs}
Here, the insertion does not happen on top of noise-type edges. We provide below an example of computation:
\begin{equs} 
	\uparrow^{i} \begin{tikzpicture}[scale=0.2,baseline=0.1cm]
		\node at (-2,4)  [dot,label= {[label distance=-0.2em]above: \scriptsize  $    $} ] (center) {};
		\node at (2,4)  [dot,label= {[label distance=-0.2em]right: \scriptsize  $ \alpha   $} ] (centerl) {};
		\node at (0,0)  [dot,label={[label distance=-0.2em]right: \scriptsize  $ \beta $}] (right) {};
		\node at (0,8)  [dot,label= {[label distance=-0.2em]right: \scriptsize  $    $} ] (centerlc) {};
		\draw[kernel1] (centerl) to
		node [sloped,below] {\small }     (centerlc);
		\draw[kernel1] (center) to
		node [sloped,below] {\small }     (right); 
		\draw[kernel1] (centerl) to
		node [sloped,below] {\small }     (right); 
		\node at (1,2) [fill=white,label={[label distance=0em]center: \scriptsize  $ a $} ] () {};
		\node at (-1,2) [fill=white,label={[label distance=0em]center: \scriptsize  $ \Xi_{\mathfrak{l}_2} $} ] () {};
		\node at (1,6) [fill=white,label={[label distance=0em]center: \scriptsize  $ \Xi_{\mathfrak{l}_1} $} ] (rightrr) {};
	\end{tikzpicture}=  \begin{tikzpicture}[scale=0.2,baseline=0.1cm]
	\node at (-2,4)  [dot,label= {[label distance=-0.2em]above: \scriptsize  $    $} ] (center) {};
	\node at (2,4)  [dot,label= {[label distance=-0.2em]right: \scriptsize  $ \alpha   $} ] (centerl) {};
	\node at (0,0)  [dot,label={[label distance=-0.2em]right: \scriptsize  $ \beta + e_i $}] (right) {};
	\node at (0,8)  [dot,label= {[label distance=-0.2em]right: \scriptsize  $    $} ] (centerlc) {};
	\draw[kernel1] (centerl) to
	node [sloped,below] {\small }     (centerlc);
	\draw[kernel1] (center) to
	node [sloped,below] {\small }     (right); 
	\draw[kernel1] (centerl) to
	node [sloped,below] {\small }     (right); 
	\node at (1,2) [fill=white,label={[label distance=0em]center: \scriptsize  $ a $} ] () {};
	\node at (-1,2) [fill=white,label={[label distance=0em]center: \scriptsize  $ \Xi_{\mathfrak{l}_2} $} ] () {};
	\node at (1,6) [fill=white,label={[label distance=0em]center: \scriptsize  $ \Xi_{\mathfrak{l}_1} $} ] (rightrr) {};
\end{tikzpicture} +  \begin{tikzpicture}[scale=0.2,baseline=0.1cm]
\node at (-2,4)  [dot,label= {[label distance=-0.2em]above: \scriptsize  $    $} ] (center) {};
\node at (2,4)  [dot,label= {[label distance=-0.2em]right: \scriptsize  $ \alpha + e_i   $} ] (centerl) {};
\node at (0,0)  [dot,label={[label distance=-0.2em]right: \scriptsize  $ \beta $}] (right) {};
\node at (0,8)  [dot,label= {[label distance=-0.2em]right: \scriptsize  $    $} ] (centerlc) {};
\draw[kernel1] (centerl) to
node [sloped,below] {\small }     (centerlc);
\draw[kernel1] (center) to
node [sloped,below] {\small }     (right); 
\draw[kernel1] (centerl) to
node [sloped,below] {\small }     (right); 
\node at (1,2) [fill=white,label={[label distance=0em]center: \scriptsize  $ a $} ] () {};
\node at (-1,2) [fill=white,label={[label distance=0em]center: \scriptsize  $ \Xi_{\mathfrak{l}_2} $} ] () {};
\node at (1,6) [fill=white,label={[label distance=0em]center: \scriptsize  $ \Xi_{\mathfrak{l}_1} $} ] (rightrr) {};
\end{tikzpicture}.
\end{equs}
	Let $\sigma =  X^{k}  \prod_i \CI_{a_{i}}(\sigma_{i}) $, we define a product $ \star_2 $ given by
\begin{equs} \label{def_star_2}
	\sigma \star_{2} \tau = \tilde{\uparrow}^{k}_{N_{\tau}} \( \prod_i \CI_{a_{i}}(\sigma_{i}) \, \widehat{\curvearrowright} \, \tau \)
\end{equs}
where $ \tilde{\uparrow}^{k}_{N_{ \tau}} $ is given by
\begin{equs} \label{splitting_polynomials}
	\tilde{\uparrow}^{k}_{N_{\tau}} =
	\sum_{k = \sum_{v \in N_{\tau}} k_v} \uparrow^{k_v}_{v}.
\end{equs}
When the previous operation  is happening on all the nodes and not only the set $ N_{\tau} $, we denote this operator by $ \tilde{\uparrow}^{k} $.  Above the  product $\star_2$ is defined for $\sigma \in \CT_+$ and $\tau \in  \CT$ and takes values in $\CT$. One can restrict the second argument $\tau$ to be in $\CT_+$ and get a product on $\CT_+$. 

\begin{proposition} \label{associa_star_2}
	The product $ \star_2 $ is associative on $\CT_+$.
	\end{proposition}
The associativity of the product $ \star_2 $ can be proved using the Guin-Oudom procedure \cite{Guin1,Guin2} applied either on a pre-Lie product called plugging product in \cite{BM22} or a post-Lie product given in \cite{BK} (see \cite{ELM} for a first application of the Guin-Oudom procedure to post-Lie products).
The post-Lie product is given as a way to merge the two operations $ \widehat{\curvearrowright} $ and $ 	\uparrow^{i} $ taking into account the non-commutativity between them. Indeed, one has for all decorated trees $ \sigma, \tau $ and $ a \in \mathbb{N}^{d+1} $, $ i \in \lbrace 0,...,d \rbrace $:
	\begin{equs} \label{non_commutation}
		\uparrow^{i} \left( \sigma \widehat{\curvearrowright}^a \tau  \right) =    ( \uparrow^{i} \sigma ) \widehat{\curvearrowright}^a \, \tau + \sigma \widehat{\curvearrowright}^a \, ( \uparrow^{i} \tau)  - \sigma \widehat{\curvearrowright}^{a-e_i} \tau. 
	\end{equs}

 One of the main results in \cite{BM22} is to connect $ \star_2 $ with a coaction $ \Delta_2 : \CT \rightarrow   \CT_+ \hattimes  \CT$ defined in \cite{BHZ}. This coaction is given  recursively by
\begin{equs}
	\Delta_{2} X_i & = X_i \hattimes \one + \one \hattimes X_i, \quad \Delta_2 \Xi_{\mathfrak{l}} = \one \hattimes \Xi_{\mathfrak{l}} , 
	\\ \Delta_{2} \CI_a(\tau) & =  \left( \id \hattimes \CI_a \right)\Delta_{2} \tau + \sum_{\ell \in \N^{d+1}} \frac{1}{\ell !} \CI_{a + \ell}(\tau) \hattimes X^{\ell},
\end{equs} 
and extended multiplicatively for the tree product. It can also be viewed using the same definition as a coproduct from $ \CT_+ $ into $    \CT_+ \hattimes  \CT_+$.
 We have replaced the standard tensor product $ \otimes $ by $ \hattimes $ to make sense of the infinite sum over $ \ell $. Indeed, the map $ \Delta_2  $  is a triangular map for the bigrading  given in \cite{BHZ} by
\[ 
(|T_{\Labe}^{\Labn}|_{\text{bi}}) = ( |T_{\Labe}^{\Labn}|_{\text{grad},\s}, |N_T \setminus \lbrace \varrho_T \rbrace| + |E_T| ), 
\] 
where $ T^{\Labn}_{\Labe} $ is a rooted tree and $ s:=(s_0,\ldots, s_{d})\in \mathbb{N}^{d+1}$ with $s_i > 0$ is fixed.
We define the grading $ |\cdot|_{\text{grad},\s} $ of a tree $ \tau $  by:
\begin{equation} \label{grading}
	|\tau|_{\text{grad},\s}:=\sum_{e\in E_{\tau}}\big|\Labe(e) \big|_{\s}
\end{equation}
where $ \big|\cdot \big|_{\s} $ is a map from $ \Lab $ into $ \mathbb{R} $ called degree extended as a map from $ \CD \times \mathbb{N}^{d+1} $ as
\begin{equs}
	\big| (\mathfrak{t},m) \big|_{\s} = 	\big| \mathfrak{t}\big|_{\s} - 	\big| m \big|_{\s} 
\end{equs}
where 	$ \big| m \big|_{\s} $ is given by 
\begin{equs}
	|m |_{\s}:=s_0 m_0+\cdots +s_{d}m_{d}.
\end{equs}
The values of $ |\cdot|_{\s} $ on $ \Lab_- $ give the space-time regularity of the noises whereas its values on $ \Lab_+ $ give the Schauder estimate, gain in regularity provided by the convolution with the kernels $ K_{\mathfrak{t}} $.
An important result of $ \cite{BM22} $ is the following:
\begin{proposition}
	The dual of the coproduct $ \Delta_2 $ is the product $ \star_2 $.
	\end{proposition}

\begin{remark}
	One notices that it is nicer to work and manipulate the product $ \star_2 $ instead of $ \Delta_2 $ as it involves only finite sums and does not require the use of a bigrading.
	\end{remark}

\label{star_1}
We define a new coaction from $ \Delta_2 $ that we denote by $ \Delta_1 :  \CT \rightarrow \CT_- \hattimes \CT $ recursively defined as:
\begin{equs} \label{recursive_def_delta1}
	\Delta_1 =  
	\mathcal{M}^{(13)(2)} \left(   \Delta_{\circ}  \hattimes \id \right) \hat{\Delta}_2
\end{equs} 
where $ \hat{\Delta}_2 $ is defined from $\CT $ into $ \CT \hattimes \CT $ as the same as $\Delta_2$ except for $\Xi_{\mathfrak{l}}$:
\begin{equs}
	\hat{\Delta}_2 \Xi_{\mathfrak{l}} = \Xi_{\mathfrak{l}} \hattimes \one + \Xi_{0} \hattimes \Xi_{\mathfrak{l}}.
\end{equs}
One has, for $ \bar \tau_1 \in \CT_- $ and $ \tau_2, \tau_3 \in \CT $
\begin{equs}
	\mathcal{M}^{(13)(2)}\left( \bar\tau_1 \otimes \tau_2 \otimes \tau_3 \right) = \bar \tau_1 \cdot \tau_3 \otimes 
	\tau_2.
\end{equs}
 The map 
 $  \Delta_{\circ} : \CT \rightarrow \CT_- \hattimes \CT   $ is defined multiplicatively by
\begin{equs} \label{Deltacirc}
	\Delta_{\circ} X^k = \one \hattimes X^k, \quad	\Delta_{\circ} \Xi_{\mathfrak{t}} = \Xi_0 \hattimes \Xi_{\mathfrak{t}}, \quad  \Delta_{\circ} \CI_{a}(\tau) = \left( \id \hattimes \mathcal{I}_a \right) \Delta_1 \tau.
\end{equs}
One can observe that  $  \Delta_{\circ} $ sends  $ \CT_+ $ into $ \CT_- \hattimes \CT_+   $.
Starting from \eqref{recursive_def_delta1}, one can extend $ \Delta_1 $ multiplicatively for the forest product to get a coproduct from $ \CT_- $ into $ \CT_- \hattimes \CT_- $. 
One important property is the co-interaction between $ \Delta_1 $ and $ \Delta_2  $ given below:
\begin{equs} \label{cointeraction_-_+}
	\mathcal{M}^{(13)(2)(4)}\left(  \Delta_{\circ} \otimes  \Delta_{1}\right)
	\Delta_2 = \left( \id \otimes \Delta_2 \right) \Delta_1, 
\end{equs}
where for $ \bar \tau_1, \bar \tau_3 \in \CT_- $, $ \tau_2 \in \CT_+ $ and $ \tau_4 \in \CT $
\begin{equs}
	\mathcal{M}^{(13)(2)(4)}\left( \bar\tau_1 \otimes \tau_2 \otimes \bar \tau_3 
	\otimes \tau_4 \right) = \bar \tau_1 \cdot \bar \tau_3 \otimes 
	\tau_2 \otimes \tau_4.
\end{equs}
At the level of characters, it reads for $ \tau \in \CT_+ $ and $ \sigma \in \CT $
\begin{equs} \label{productM}
	M_{\alpha}^* \left(  \tau \star_2 \sigma  \right) = 
		\left( \hat{M}_{\alpha}^* \tau \right) \star_2 	\left( M_{\alpha}^* \sigma \right)  
\end{equs}
where $ \alpha : \CT_- \rightarrow \mathbb{R} $ is a character meaning that it is a multiplicative map for the forest product and $	M_{\alpha}^* $ (resp. $\hat{M}_{\alpha}^*$) on $\CT$ is the adjoint of $M_{\alpha}$ (resp. $ \hat{M}_{\alpha} $) given by
\begin{equs}
\label{M alpha}
M_{\alpha}  = \left( \alpha \,  \hat{\otimes} \, \id \right) \Delta_1, \quad  M_{\alpha}  = \left( \alpha \,  \hat{\otimes} \, \id \right) \Delta_{\circ}.
\end{equs}
We denote by  $ \star_1 $ the convolution product coming from $ \Delta_1 $. By abuse of notation, we use the same notation for $\Delta_{\circ}$. One can also provide a recursive definition given by
\begin{equs}
	\sigma_1 \cdots \sigma_m \star_1  X^k  \prod_{i=1}^n \CI_{a_i}(\tau_i) \Xi_{0} & = 
	\sum_{j=1}^m \sum_{I^j_1,..., I^j_n}  X^k  \prod_{i=1}^n \CI_{a_i}(  \prod_{r \in I^j_i} \sigma_r \star_1 \tau_i)  \star_2 \sigma_j
	\\ & + \sum_{I_1,..., I_n}  X^k  \prod_{i=1}^n \CI_{a_i}(  \prod_{r \in I_i} \sigma_r \star_1 \tau_i)  \Xi_{\mathfrak{l}}
\end{equs}
where the $ I^j_i $ form a partition of $ \lbrace 1,...,m \rbrace \setminus \lbrace j \rbrace $ and  the $ I_i $ form a partition of $ \lbrace 1,...,m \rbrace  $. Some of the $ I^j_i $ and $I^j_i$ could be empty.
For $ \mathfrak{l} \neq 0 $, one  has only
\begin{equs}
		\sigma_1 \cdots \sigma_m \star_1  X^k  \prod_{i=1}^n \CI_{a_i}(\tau_i) \Xi_{\mathfrak{l}} & = 
 \sum_{I_1,..., I_n}  X^k  \prod_{i=1}^n \CI_{a_i}(  \prod_{r \in I_i} \sigma_r \star_1 \tau_i)  \Xi_{\mathfrak{l}}.
\end{equs}
One has the following crucial property:
\begin{proposition} \label{consequence_M}
	One has for every $ \tau = X^k \prod_{i=1}^n \CI_{a_i}(\tau_i) \in \CT_+ $
	\begin{equs} \label{morphism_property_M}
			\hat{M}_{\alpha}^* \left(  X^k \prod_{i=1}^n \CI_{a_i}(\tau_i) \right) = X^k \prod_{i=1}^n \CI_{a_i}( 	M_{\alpha}^*\tau_i).
	\end{equs}
	\end{proposition}
The recursive formulation \eqref{recursive_def_delta1} and \eqref{Deltacirc} has been introduced in \cite{BR18}. It is a decomposition of the extraction-contraction coproduct into two different operations: one is the extraction at the root given by $ \Delta_2 $, the other is the extraction outside the root which is iterated by the map $ \Delta_{\circ} $.
This recursive construction can be generalised for designing renormalisation maps. One just needs to replace $ \Delta_2  $ by a linear map $ R : \CT \rightarrow \CT $ which has some compatibility property with $ \Delta_2 $ that is:
\begin{equs}
	\left( \id \hattimes R\right) \Delta_2 = \Delta_2 R.
\end{equs}
For the adjoint $ R^* $, one gets
\begin{equs}
	R^* (\sigma \star_2 \tau) =\sigma \star_2 
	(R^* \tau).
\end{equs}
Then, one can define a renormalisation map 
\begin{equs}  \label{def::M}
	M = M^{\circ} R, \quad M^{\circ} X^k \tau = X^k M^{\circ}  \tau, \quad M^{\circ} \CI_{a}(\tau) = \CI_a(M \tau),
\end{equs}
where $ M^{\circ} $ is multiplicative for the tree product.
It just corresponds to replacing the extraction with another procedure which is iterated deeper in the tree. One main example of such a map is the map $ R_{\alpha} $ that gives $ M_{\alpha} $. Its adjoint is defined by:
\begin{equs} \label{adjoint_R}
	R_{\alpha}^* \tau = \sum_{\sigma \in \mathfrak{T}} \frac{\alpha(\sigma)}{S(\sigma)} \sigma \star_2 \tau.
\end{equs}

We finish this section by introducing various derivatives on the non-linearities $ F_{\mathfrak{t}}^{\mathfrak{l}}, \mathfrak{l} \in \Lab_-, \, \mathfrak{t} \in \Lab_+ $, appearing on the right hand side of  \eqref{main_equation}.
They are  functions of a finite number of  variables $Z_a$ indexed by $a\in \CD_+$ that represent the solutions of the system \eqref{main_equation} and their derivatives.  We define in the usual way partial derivatives $D_a$ in the variable $u_a $.  For the canonical basis $e_i$ of $\mathbb{N}^{d+1}$, we set
$$ 
\partial^{e_i} = \sum_a Z_{a+e_i}\,D_a.
$$
where for $ a =(\mathfrak{t},p)  $, one has $ a + e_i  = (\mathfrak{t},p+e_i)$.
These operators commute so one can define for $k = \sum_{i=0}^d k_i e_i \in\mathbb{N}^{d+1}$
$$
\partial^k = \prod_{i=0}^{d} (\partial^{e_i})^{k_i}.
$$
These operators will act on functions that depend on finitely many variables $u_a$ so their action will not involve infinite sums. We define inductively a family $F = (F_{\mathfrak{t}})_{\mathfrak{t} \in \Lab_+}$ of  functions of the variables $Z_a$, setting for $ \tau = X^{k} \Xi_{\mathfrak{l}}\prod_{j=1}^{n} \CI_{a_j}(\tau_j) $, with $a_j = (\Labhom_{j},k_j)$, 
\begin{equation} \label{def_upsilon}
	\begin{aligned}
		F_{\mathfrak{t}}(\Xi_{\mathfrak{l}}) = F^{\mathfrak{l}}_{\mathfrak{t}}, \qquad F_{\mathfrak{t}}(\tau) = \Big\{\partial^k D_{a_1} ... D_{a_n} F_{\mathfrak{t}} (\Xi_{\mathfrak{l}})\Big\}\,\prod_{j=1}^n F_{\mathfrak{t}_j}(\tau_j).
\end{aligned} \end{equation}
The family $ F $ plays the role of the elementary differentials in the context of a system of singular SPDEs. The main difference is that one has to deal with various derivatives $ D_a $ corresponding to the variables of the system and their derivatives. The other derivatives correspond to monomials that are used in the local expansion of the solution.

\section{Regularity Structures B-series}
\label{sec::4}

In this section, we introduce Butcher-type series based on the decorated trees coming from the previous section. We also prove our main result Theorem~\ref{main_result} which is split into two theorems: Theorem~\ref{main_composition} for composition and Theorem~\ref{thm_substition} for substitution.

\begin{definition} \label{Def_B_-}
 	A Regularity Structures B-series is of the form:
 	\begin{equs} \label{B_-}
 		B_-(\alpha,F, {\bf{u}}) = \sum_{\tau \in \mathfrak{T}} \frac{\alpha(\tau)}{S(\tau)} F(\tau)({\bf{u}})
 	\end{equs}
 where 
 \begin{itemize}
 	\item the map $ \alpha $ is a linear map from $ \CT $ into $ \mathcal{C}^{\infty}(\mathbb{R}^{d+1},\mathbb{R}) $, we assume that this map has a finite support and $ \alpha(\mathbbm{1}) = 1 $. 
 	\item $ S(\tau) $ is the symmetry factor of the decorated tree $ \tau $ defined in \eqref{symmetry_factor}. 
 \item $ F(\tau) $ are elementary differentials defined by \eqref{def_upsilon}. 
 \end{itemize}
\end{definition}

The series \eqref{B_-} has been used for describing the expansion of the right hand side of \eqref{main_equation}. In this context, $ \alpha $ allows us to consider only a finite number of iterated integrals for describing the singular products of the SPDE. In the next definition, we provide a new series that allows the description of the solutions of \eqref{main_equation}.

\begin{definition} \label{def_B_+}
We defined a specific type of Regularity Structures B-series given by:
\begin{equs} 
	B_+(\alpha,F, a, {\bf{u}}) = \sum_{k \in \mathbb{N}^{d+1}} \alpha(X^k) \frac{u_{(\mathfrak{t},m+k)}}{k!} + \sum_{\CI_a(\tau) \in \CP\CT_{a}} \frac{\alpha(\CI_a(\tau))}{S(\tau)} F_{\mathfrak{t}}(\tau)({\bf{u}})
	\end{equs}
where $ a = (\mathfrak{t},m) $ and $ \CP\CT_{a} $ is the set of decorated trees of the form $ \CI_a(\tau) $ with $ \tau \in \mfT $.
\end{definition}

In the context of \eqref{main_equation}, one can derive a Taylor-type expansion of the solution $ u_{\mathfrak{t}} $ when the noises $ \xi_{\mathfrak{l}} $ are smooth, as
\begin{equs}
	u_{\mathfrak{t}} = 	B_+(\alpha,F, \mathfrak{t}, {\bf{u}}(z)) + R_{\mathfrak{t},z,\gamma}
\end{equs}
where we have used the shorthand notation $ \mathfrak{t} $ instead of $ (\mathfrak{t},0) $ and
 $ \alpha $ is equal to the map $ \Pi_z \circ \mathcal{P}_{\leq \gamma} $ where $ \mathcal{P}_{\leq \gamma} $ is the projection onto the decorated trees of degree at most equal to $ \gamma $. The degree of a decorated tree is defined recursively by
 \begin{equs}
 	\deg(  X^{\ell}  \prod_{i=1}^r \CI_{a_i}(\tau_i) \Xi_{\mathfrak{l}} ) = |\ell|_{\s} + |\mathfrak{l}|_{\s} +  \sum_{i=1}^n ( |a_i|_{\s} +  \deg(  \tau_i  )).
 \end{equs}
 The local subcriticality of the equation guarantees that one has only a finite number of decorated trees with degree at most $ \gamma $ for every $ \gamma \in \mathbb{R} $. This makes the series finite in the end. The map $ \Pi_z $ is recursively defined by
\begin{equs} \label{def_Pi}
	(\Pi_z X^k)(z')  & = (z'-z)^k, \quad (\Pi_z 	\xi_{\mathfrak{l}} )(z') = \xi_{\mathfrak{l}}(z'),
	\\ (\Pi_z \CI_{(\mathfrak{t},m)}(\tau))(z')  & = (D^{m} K_{\mathfrak{t}} * \Pi_z \tau)(z') \\ & - \sum_{|k|_{\s} \leq \deg(\CI_{(\mathfrak{t},m)}(\tau))} \frac{(z'-z)^k}{k!} (D^{m +k} K_{\mathfrak{t}} * \Pi_z \tau)(z)
\end{equs}
and extended multiplicatively for the tree product.
The subtraction of the Taylor expansion in the definition of $ \Pi_z$ is essential to get the correct local behaviour according to the degree of the decorated trees considered that is:
\begin{equs}
		(\Pi_z \tau)(z') \lesssim | z'-z |^{\deg \tau}.
\end{equs}
Morever, the truncation $ \mathcal{P}_{\leq \gamma} $ allows us to have a remainder $ R_{\mathfrak{t},z,\gamma} $ of order $ \gamma $ in the sense that one has
\begin{equs} \label{order_gamma}
	R_{\mathfrak{t},z,\gamma}(z') \lesssim |z'-z|^{\gamma}.
\end{equs}

We define the composition of a B-series $ U_a = B_+(\alpha, F,a, {\bf{u}}) $  with a smooth function $ f   $ depending on a finite number of $ u_a $ 
\begin{equs} \label{composition_f}
	\begin{aligned}
	f_{\gamma}((U_{a_i})_{i})&= \hat{\mathcal{P}}_{\leq \gamma} \sum_{n} \sum_{a_1,...,a_n} \sum_{\beta_1, ..., \beta_n} \\ & \prod_{i=1}^n \frac{1}{\beta_i !} \left(
	U_{a_i} - u_{a_i} \one  \right)^{\beta_i}
	\prod_{i=1}^n (D_{a_i})^{\beta_i} f({\bf{u}})
	\end{aligned}
\end{equs} 
where the $ a_i $'s are disjoint. In the definition, we consider a projection according to the degree in order to get a finite sum. Here $ \hat{\mathcal{P}}_{\leq \gamma} $ removed the terms that have an order bigger than $ \gamma $ in the sense of \eqref{order_gamma}.
The result is itself a B-series as we have 
\begin{equs}
	f_{\gamma}((U_{a_i})_{i})  = B_-(\beta, \hat{F}, {\bf{u}})
\end{equs}
where $ \beta $ is only supported on decorated trees of the form $ X^k \Xi_0 \prod_{i=1}^n \mathcal{I}_{a_i}(\tau_i) $ with degree at most equal to $ \gamma $ with $ a_i =(\Labhom_i,p_i) $. The map $ \hat{F} $ is defined by
\begin{equs}	
\hat{F}(\tau)({\bf{u}}) =  \prod_{i=1}^n F_{\Labhom_{i}}(\tau_i)({\bf{u}}) \,  \partial^k \prod_{i=1}^n D_{a_i} f({\bf{u}}), \quad 
\beta(\tau) = \alpha(X^k) \prod_{i=1}^n \alpha(\CI_{a_i}(\tau)).
	\end{equs}
 
\begin{remark} In the one dimensional case where $ f $ depends only on one variable denoted by $ u $, the set of type $ \Lab_+ $ is a singleton and one can write
\begin{equs}
	f_{\gamma}(U) = \hat{\mathcal{P}}_{\leq \gamma} \sum_{k} (U-u \one)^k\frac{f^k(u)}{k!}
\end{equs}
which corresponds to the composition of modelled distributions defined in \cite{reg}. This composition is very similar to the one used for classical B-series in numerical analysis. The main difference is the polynomial part given by the classical monomials. 
\end{remark}
The composition of B-series is given by
\begin{equs} \label{compostion_B_Series}
	B_-(\alpha,F, \cdot) \circ B_+(\beta,G, \cdot, {\bf{u}}) = 	B_-(\alpha,F, (B_+(\beta,G,a,{\bf{u}}))_{a \in \CD_+})
\end{equs}
where we use the composition defined in \eqref{composition_f} and we omit in the notation the truncation parameter $ \gamma $ that we suppose fixed.
The composition is happening on each elementary differential of the B-series  $	B_-(\alpha,F, \cdot) $ that depends in fact only on a finite number of variables. Therefore, we do not need the full set of decorations $ \CD_+ $ but just a finite subset. The next proposition has been proved in \cite{BB21} and provides a morphism property with respect to the product $ \star_2 $. It is a nice algebraic property that we need to prove the main result of this paper.
\begin{proposition} \label{star_morphism}
	For every $ \mathfrak{t} \in \Lab_+ $, for every $ \sigma\in \mfT$, and $\tau = X^{k} \prod_{j=1}^{n} \CI_{a_j}(\tau_j) \in \CT_+$ with $ a_j = (\Labhom_{j}, k_j) $, one has
	\begin{equation}\label{star2_morphism}
		F_{\mathfrak{t}}\bigg(\Big\{X^{k} \prod_{j=1}^{n} \CI_{a_j}(\tau_j)\Big\} \star_2 \sigma\bigg) = \Big\{\partial^k D_{a_1} ... D_{a_n} F_{\mathfrak{t}}(\sigma)\Big\}\,\prod_{j=1}^n F_{\mathfrak{t}_j}(\tau_j).
	\end{equation}
\end{proposition}

\begin{remark}
	The proposition above was first proved in a simple case when $ \tau $ is either a planted tree $ \CI_a(\bar{\tau}) $ or a monomial $ X^k $ in \cite{BCCH}. One has in this case
	\begin{equs}
		\CI_a(\bar{\tau})  \star_2 \sigma = \bar{\tau} \, \widehat{\curvearrowright}^{a} \sigma, \quad X^k  \star_2 \sigma = \tilde{\uparrow}^k \sigma
		\end{equs}
	which give rise to the following identities:
	\begin{equs}
		F_{\mathfrak{t}}( \bar{\tau} \, \widehat{\curvearrowright}^{a} \sigma )  =
		F_{\bar{\mathfrak{t}}}(\bar{\tau}) D_a F_{\mathfrak{t}}(\sigma),
		\quad	F_{\mathfrak{t}}( \tilde{\uparrow}^k  \sigma )  = \partial^k F_{\mathfrak{t}}(   \sigma )
	\end{equs}
where $ a= (\bar{\mathfrak{t}},m) $
The proof of these identities has been simplified in \cite{BaiHos} and ideas of this proof are implemented for getting \eqref{star_morphism}. The crucial argument is to observe that the operators $ D_a $ and $ \partial^{e_i} $ do not commute as one has the following identity:
\begin{equs}
	D_a \partial^{e_i} = \partial^{e_i} D_a + D_{a - e_i}
\end{equs}
which is quite similar to the identity \eqref{non_commutation}.
	\end{remark}

Below, we present one of our main results in this paper which is a precise description of the composition of Regularity Structures B-series. This result can be seen as an extension of Theorem~\ref{composition_Butcher} in the context of  SPDEs. 

\begin{theorem} \label{main_composition}
	One has for $ \alpha : \CT \rightarrow \mathbb{R} $ and the character $ \beta : \CT_+ \rightarrow \mathbb{R} $
	\begin{equs}
		B_-(\alpha,F, \cdot) \circ B_+(\beta,F, \cdot, {\bf{u}}) = 	B_-(\beta \star_2 \alpha ,F, {\bf{u}}).
	\end{equs}
	\end{theorem}
\begin{proof} 
	One has
	\begin{equs}
	& B_-(\alpha,F, (B_+(\beta,F,a,{\bf{u}}))_{a \in \CD_+})  
	\\ & =  \sum_{\tau \in \mfT} \frac{\alpha(\tau)}{S(\tau)} F(\tau)(  (B_+(\beta,F,a,{\bf{u}}))_{a \in \CD_+} )
	\\ & = \sum_{\tau \in \mfT} \frac{\alpha(\tau)}{S(\tau)}  \sum_{n} \sum_{a_1,...,a_n} \sum_{\beta_1 ... \beta_n} \prod_{i=1}^n \frac{1}{\beta_i !} \left(
	B_{+}(\alpha,F,a_i, {\bf{u}}) - u_{a_i}   \right)^{\beta_i}
	\prod_{i=1}^n (D_{a_i})^{\beta_i} F(\tau)({\bf{u}})
	\end{equs}
where we have used the composition formula given by \eqref{composition_f}. Here, the $ a_i $'s are disjoint. We have omitted the projection according to the degree. 
We use the Fa\`a di Bruno formula from Lemma A.1 in \cite{BCCH}  applied to a function $ G $ depending on a finite number of variables:
\begin{equation}
	\frac{\partial^{k}G}{k!} =  \sum_{b_1,...,b_m} \sum_{k = \sum_{j=1}^m \beta_j k_j} \prod_{j=1}^m \frac{1}{\beta_j !} \left(\frac{u_{b_j + k_j}}{k_j!}\right)^{\beta_j} \prod_{j=1}^m (D_{b_j})^{\beta_j} G 
\end{equation}
and we get for $a_j$'s and $ \tau_i $'s such that $ \CI_{a_i}(\tau_i) \neq \CI_{a_j}(\tau_j) $ for $ i \neq j $, $ a_i= (\mathfrak{t}_i,m_i) $
\begin{equs}
	& B_-(\alpha,F, (B_+(\beta,F,a,{\bf{u}}))_{a \in \CD_+})  \\ & = \sum_{\tau \in \mfT} \frac{\alpha(\tau)}{S(\tau)}  \sum_{\sigma  = X^k \prod_{j=1}^n \mathcal{I}_{a_j}(\tau_j)^{\beta_j}}  \sum_{k} \frac{\beta(X^k)}{k!} \\ & \prod_{j=1}^n \frac{1}{\beta_j !} \left(  \frac{F_{\mathfrak{t}_j}(\tau_j)}{S(\tau_j)} ({\bf{u}})\, \beta(\CI_{a_j}(\tau_j)) \right)^{\beta_j}   \partial^k \prod_{j=1}^n (D_{a_{j}})^{\beta_{j}}  F(\tau)({\bf{u}}).
\end{equs}
By using the definition of the symmetry factor given in \eqref{symmetry_factor}, one gets:
\begin{equs}
	& B_-(\alpha,F, (B_+(\beta,F,a,{\bf{u}}))_{a \in \CD_+})   \\ & = \sum_{\tau \in \mfT} \frac{\alpha(\tau)}{S(\tau)}  \sum_{\sigma  = X^k \prod_{j=1}^n \mathcal{I}_{a_j}(\tau_j)} \sum_{k}\frac{k! \prod_{j=1}^n S(\tau_j)}{S\big(X^k\prod_{j=1}^n\CI_{a_j}(\tau_j)\big)} \, \frac{\beta(X^k)}{k!} \\ & \prod_{j=1}^n  \frac{F_{\mathfrak{t}_j}(\tau_j)}{S(\tau_j)}({\bf{u}})\,\beta(\CI_{a_j}(\tau_j))   \partial^k \prod_{j=1}^n D_{a_{j}}  F(\tau)({\bf{u}})
\end{equs}
where now the $a_j$'s and $ \tau_i $'s are not assumed to satisfy $ \CI_{a_i}(\tau_i) \neq \CI_{a_j}(\tau_j) $ for $ i \neq j $. By using Proposition~\ref{star_morphism}, one gets
\begin{equs}
\prod_{j=1}^n  F_{\mathfrak{t}_j}(\tau_j)	\, \partial^k \prod_{j=1}^n D_{a_{j}} F(\tau) = F(X^k\prod_{j=1}^n\CI_{a_j}(\tau_j) \star_2   \tau).
\end{equs}
Therefore
\begin{equs}
		 & B_-(\alpha,F, (B_+(\beta,F,a,{\bf{u}}))_{a \in \CD_+})  \\ & = \sum_{\tau \in \mfT}   \sum_{\sigma  = X^k \prod_{j=1}^n \mathcal{I}_{a_j}(\tau_j)} \sum_{k} \alpha(\tau) \beta(X^k \prod_{j=1}^n \CI_{a_j}(\tau_j))
		 \\ &   F\left( \frac{X^k\prod_{j=1}^n\CI_{a_j}(\tau_j)}{S\big(X^k\prod_{j=1}^n\CI_{a_j}(\tau_j)\big)} \star_2   \frac{\tau}{S(\tau)} \right)({\bf{u}})
	\end{equs}
where we have used the multiplicativity of $ \beta $. In the end, one has:
\begin{equs}
	 B_-(\alpha,F, (B_+(\beta,F,a,{\bf{u}}))_{a \in \CD_+}) = \sum_{\tau \in \mfT}   \sum_{\sigma \in \mfT_+} \alpha(\tau) \beta(\sigma)
	    F\left( \frac{\sigma}{S\big( \sigma\big)} \star_2   \frac{\tau}{S(\tau)} \right)({\bf{u}}).
\end{equs}
We conclude from the fact that
\begin{equs}
	\frac{\sigma}{S(\sigma)} \star_2 \frac{\tau}{S(\tau)} = \sum_{\bar{\tau} \in \mfT} \frac{m(\sigma,\tau,\bar{\tau})}{S(\bar{\tau})} \bar{\tau}
\end{equs}
where the coefficients $ m(\sigma,\tau,\bar{\tau}) $ are defined as
	\begin{equs}
 	m(\sigma,\tau,\bar{\tau}) =	\big\langle \frac{\sigma}{S(\sigma)} \otimes \frac{\tau}{S(\tau)}, \Delta_2 \bar{\tau}  \big\rangle.
	\end{equs}
	 Indeed, one has using the innner product on decorated trees:
\begin{equs}
	\big\langle \frac{\sigma}{S(\sigma)} \star_2 \frac{\tau}{S(\tau)} , \bar{\tau}  \big\rangle = \big\langle \frac{\sigma}{S(\sigma)} \otimes \frac{\tau}{S(\tau)} , \Delta_2 \bar{\tau}  \big\rangle = 	m(\sigma,\tau,\bar{\tau}).  
\end{equs}
We obtain
\begin{equs}
	& B_-(\alpha,F, (B_+(\beta,F,a,{\bf{u}}))_{a \in \CD_+}) \\ & = \sum_{\tau \in \mfT}   \sum_{\sigma \in \mfT_+} \beta(\sigma) \alpha(\tau) 
	F\left( \sum_{\bar{\tau} \in \mfT} \frac{m(\sigma,\tau,\bar{\tau})}{S(\bar{\tau})} \bar{\tau} \right)({\bf{u}})
	\\ & = \sum_{\bar{\tau} \in \mfT} \sum_{\tau \in \mfT}   \sum_{\sigma \in \mfT_+}  \frac{m(\sigma,\tau,\bar{\tau})}{S(\bar{\tau})}  \beta(\sigma)\alpha(\tau)
	F\left( \bar{\tau} \right)({\bf{u}})
	\\ & = \sum_{\bar{\tau} \in \mfT}  \frac{(\beta \star_2 \alpha)(\bar{\tau})}{S(\bar{\tau})}  
	F\left( \bar{\tau} \right)({\bf{u}})
	\\ &= B_-(\beta \star_2 \alpha ,F, {\bf{u}})
\end{equs}
where we have used
\begin{equs}
	\sum_{\tau \in \mfT}   \sum_{\sigma \in \mfT_+}  m(\sigma,\tau,\bar{\tau})  \beta(\sigma)\alpha(\tau) = (\beta \star_2 \alpha)(\bar{\tau}).
\end{equs}
	\end{proof}
\begin{remark}
	A specific case of this theorem is when
	one considers the following B-series:
	\begin{equs}
		f^{\mathfrak{t}}({\bf{u}}) = \sum_{\mathfrak{l} \in \Lab_-} F^{\mathfrak{l}}_{\mathfrak{t}}({\bf{u}}) \alpha(\Xi_{\mathfrak{l}}).
	\end{equs}
In this B-series, we consider only noise trees that are the $ \Xi_{\mathfrak{l}} $. If $ \alpha(\Xi_{\mathfrak{l}}) = \xi_{\mathfrak{l}} $, it corresponds exactly to the right hand side of \eqref{main_equation}. 
	One has
	\begin{equs} \label{composition_solution}
			f^{\mathfrak{t}}(\cdot) \circ B_+(\alpha,F, \cdot, {\bf{u}}) = 	 \sum_{\tau \in \mfT} \sum_{\mathfrak{l} \in \Lab_-}   \sum_{\sigma \in \mfT_+}  \frac{m(\sigma,\Xi_{\mathfrak{l}},\bar{\tau})}{S(\bar{\tau})}  \alpha(\sigma)\alpha(\Xi_{\mathfrak{l}}) 	F\left( \tau \right)({\bf{u}}).
	\end{equs}
One notices that
\begin{equs}
	\sigma \star_2 \Xi_{\mathfrak{l}} = \sigma \Xi_{\mathfrak{l}}, \quad m(\sigma,\Xi_{\mathfrak{l}},\sigma \Xi_{\mathfrak{l}}) = 1.
\end{equs}
Therefore
\begin{equs}
		f^{\mathfrak{t}}(\cdot) \circ B_+(\alpha,F, \cdot, {\bf{u}}) = 	 \sum_{\mathfrak{l} \in \Lab_-}   \sum_{\sigma \in \mfT_+}  \frac{1}{S(\sigma\Xi_{\mathfrak{l}})}  \alpha(\sigma)\alpha(\Xi_{\mathfrak{l}}) 	F\left( \sigma\Xi_{\mathfrak{l}} \right)({\bf{u}}).
	\end{equs}
So if $ \alpha $ is not multiplicative, one has
	\begin{equs}
		f^{\mathfrak{t}}(\cdot) \circ B_+(\alpha,F, \cdot, {\bf{u}}) =  B_-(\alpha',F, {\bf{u}})
	\end{equs}
	where $ \alpha' $ is given for $ \tau = X^k \prod_{i=1}^n \CI_{a_i}(\tau_i) \Xi_{\mathfrak{l}} $ by
	\begin{equs}
		\alpha'(\tau) = 
		\sum_{\sum_j k_j = k }  \prod_j   \alpha(X^{k_j})
		\prod_{i=1}^n \alpha(\CI_{a_i}(\tau_i)) \alpha(\Xi_{\mathfrak{l}})
	\end{equs}
If $ \alpha $ is fully multiplicative then, one has
\begin{equs}
	f^{\mathfrak{t}}(\cdot) \circ B_+(\alpha,F, \cdot, {\bf{u}}) =  B_-(\alpha,F, {\bf{u}}) = \sum_{\tau \in \mfT} \frac{\alpha(\tau)}{S(\tau)} F(\tau)({\bf{u}}).
\end{equs}
One can integrate this B-series (moving from $\alpha(\tau)$ to $\alpha(\CI_{(\mathfrak{t},0)}(\tau))$) and add the polynomial part to produce the following B-series:
\begin{equs}
		B_+(\alpha,F, (\mathfrak{t},0), {\bf{u}}) = \sum_{k} \alpha(X^k) \frac{u_{(\mathfrak{t},k)}}{k!} + \sum_{\CI_a(\tau) \in \CP\CT_{(\mathfrak{t},0)}} \frac{\alpha(\CI_{(\mathfrak{t},0)}(\tau))}{S(\tau)} F_{\mathfrak{t}}(\tau)({\bf{u}}).
\end{equs}
Now assuming an identity of the form $ \alpha(\CI_{(\mathfrak{t},0)}(\tau)) = K_{\mathfrak{t}} * \alpha(\tau) $ modulo some polynomial in $ z $ for the recentering then we can say that $ B_+(\Pi_z \circ \mathcal{P}_{\leq \gamma},F, (\mathfrak{t},0), {\bf{u}}(z)) $ is a good ansatz for finding approximations of the solution of \eqref{mild_formulation}. This B-series form of the solutions was first proposed in \cite{BCCH} where one checks the coherence of the B-series which is similar to checking that the composition \eqref{composition_solution} followed by an integration step gives back the B-series we started with.
	\end{remark}

\begin{remark}
	The product $ \star_2 $ is at the core of the composition of B-series. It is also a key component in the construction of the recentering map $ \Pi_z $. Indeed, one has the following identity:
	\begin{equs}
		\Pi_z =  \left(  f_x \circ \mathcal{P}_+ \otimes \Pi \right) \Delta_2
	\end{equs} 
where the projection $ \mathcal{P}_+  $ keeps only elements $ \tau = X^k \prod_{i=1}^n \CI_{a_i}(\tau_i)  $ such that for every $ i $
\begin{equs}
	\deg(\CI_{a_i}(\tau_i) ) >0.
\end{equs}
The map $ \Pi : \CT \rightarrow \mathcal{D}'(\mathbb{R}^{d+1}, \mathbb{R}) $ is recursively defined as
\begin{equs}
	(\Pi X^k)(z)   = z^k, \quad (\Pi 	\xi_{\mathfrak{l}} )(z) = \xi_{\mathfrak{l}}(z), \quad
	(\Pi \CI_{(\mathfrak{t},m)}(\tau))(z)  = (D^{m} K_{\mathfrak{t}} * \Pi \tau)(z) 
\end{equs}
and extended multiplicatively for the tree product. The map $ f_z$ is defined in terms of $ \Pi_z $:
\begin{equs}
	f_z(X_i) = z_i, \quad f_z(\CI_{(\mathfrak{t},m)}(\tau))) = -  \sum_{\ell \in \mathbb{N}^{d+1}} 
	\frac{(-z)^\ell}{\ell !} \one_{\lbrace \deg(\CI_{(\mathfrak{t},m+ \ell)}(\tau)) > 0 \rbrace} D^{m+ \ell} K_{\mathfrak{t}} * \Pi_z \tau
\end{equs}
and extended multiplicatively for the tree product. These definitions can be found in \cite{reg,BHZ}.
	\end{remark}

We define the substitution of a B-series into another by:
\begin{equs} \label{substitution_B_series}
	B_-(\alpha,F,{\bf{u}}) \circ_s B_-(\beta,F,{\bf{u}}) = \sum_{\tau \in \mfT} \frac{\alpha(\tau)}{S(\tau)} \hat{F}(\tau)({\bf{u}})
	\end{equs}
where for $ \tau = X^k \prod_{i=1}^n \CI_{a_i}(\tau_i) \Xi_{\mathfrak{l}} $, one has for $ \mathfrak{l} \neq 0 $
\begin{equs}
	\hat{F}(\tau)({\bf{u}}) = 
	\prod_{i=1}^n \hat{F}(\tau_i)({\bf{u}}) \,  \partial^k \prod_{i=1}^n D_{a_i} F(\Xi_{\mathfrak{l}})({\bf{u}}). \, 
\end{equs}
When $ \mathfrak{l} = 0 $, one has
\begin{equs}
	\hat{F}(\tau)({\bf{u}}) = \sum_{\tau' \in \mfT} \frac{\beta(\tau')}{S(\tau')}
\prod_{i=1}^n \hat{F}(\tau_i)({\bf{u}}) \,  \partial^k \prod_{i=1}^n D_{a_i} F(\tau')({\bf{u}}).
\end{equs}
One can observe in both definitions the recursive construction and the fact that the substitution is happening only for $ \mathfrak{l} = 0 $. Indeed, in the context of singular SPDEs, this amounts to producing some renormalisation at the right hand side of \eqref{main_equation}. Only the non-linearity associated with $ \xi_{0} $, which is in general the constant noise equal to one, is changed. 
The renormalisation constants are given by the $ \beta(\tau') $ that select only decorated trees with a negative degree. The first general statement for the renormalised equation has been obtained in \cite{BCCH}. A big simplification for the proof and an extension of this statement are given in \cite{BB21,BB21b}. 
In the next theorem, we extend these results as a general substitution of Regularity Structures B-series: 
\begin{theorem} \label{thm_substition}
	One has for $ \alpha, \beta $ characters on $ \CT_- $ with finite support on $ \mathfrak{T} $
	\begin{equs} \label{substitution_reg}
			B_-(\alpha,F,{\bf{u}}) \circ_s B_-(\beta,F,{\bf{u}}) = B_-(\beta \star_1 \alpha ,F,{\bf{u}}). 
	\end{equs}
	\end{theorem}
\begin{proof}
	One has  from \eqref{M alpha}
	\begin{equs}
			\beta \star_1 \alpha  & = \left(\beta \hattimes \alpha \right) \Delta_1
			\\ &= \left(\beta \otimes \alpha \right) \Delta_1
			\\ &= \alpha\left( \left(\beta \otimes \id \right) \Delta_1 \right)
		\\ & = \alpha( M_{\beta} \cdot)
	\end{equs}
where we have used the fact that $ \beta $ has a finite support on trees for replacing $ \hattimes $ by the usual tensor product $ \otimes $. We have also made the following identification $  \mathbb{R} \otimes \alpha \equiv \alpha$.
Then,  one starts with the right hand side of \eqref{substitution_reg} and gets
\begin{equs}
	 B_-(\beta \star_1 \alpha ,F,{\bf{u}})
	  & =  B_-(\alpha( M_{\beta} \cdot) ,F,{\bf{u}})
	 \\ &  = \sum_{\tau \in \mfT} \frac{\alpha(M_{\beta} \tau)}{S(\tau)} F(\tau)({\bf{u}})
	 \\ & = \sum_{\tau \in \mfT} \frac{\alpha( \tau)}{S(\tau)} F(M^{*}_{\beta}\tau)({\bf{u}}).
	\end{equs}
One uses the fact that $ M_{\beta}^* $ is a morphism for the product $ \star_2 $ (see \eqref{productM}) to get for $  \tau = X^k \prod_{i=1}^n \CI_{a_i}(\tau_i) \Xi_{\mathfrak{l}} $
\begin{equs}
	 M_{\beta}^* \tau & =  M_{\beta}^* \left( X^k \prod_{i=1}^n \CI_{a_i}(\tau_i) \star_2 \Xi_{\mathfrak{l}} \right)
	 \\ &=  \left( \hat{M}_{\beta}^* \left(  X^k \prod_{i=1}^n \CI_{a_i}( \tau_i) \right) \right) \star_2  M_{\beta}^* \Xi_{\mathfrak{l}} 
	 \\ & =  \left(  X^k \prod_{i=1}^n \CI_{a_i}( M_{\beta}^* \tau_i) \right) \star_2  M_{\beta}^* \Xi_{\mathfrak{l}} 
\end{equs}
where in the last line we have used identity \eqref{morphism_property_M}.
Then, using Proposition~\ref{star_morphism}, one gets
 for $ \tau = X^k \prod_{i=1}^n \CI_{a_i}(\tau_i) \Xi_{\mathfrak{l}} $ where $a_i = (\mathfrak{t}_i,m_i)  $
\begin{equs}
	F(M^{*}_{\beta}\tau)({\bf{u}}) & = 	F( X^k \prod_{i=1}^n \CI_{a_i}( M_{\beta}^* \tau_i) \star_2  M_{\beta}^* \Xi_{\mathfrak{l}} )({\bf{u}})
	\\ & = \Big\{\partial^k D_{a_1} ... D_{a_n} F(M_{\beta}^* \Xi_{\mathfrak{l}} )({\bf u})\Big\}\,\prod_{j=1}^n F_{\mathfrak{t}_j}(M_{\beta}^* \tau_j)({\bf u}).
\end{equs}
We apply an induction hypothesis to get
\begin{equs}
	F_{\mathfrak{t}_j}(M_{\beta}^* \tau_j)({\bf u}) = \hat{F}_{\mathfrak{t}_j}( \tau_j)({\bf u})
\end{equs}
and we know that 
\begin{equs}
	M_{\beta}^* \Xi_{0} = \sum_{\tau' \in \mfT} \frac{\beta(\tau')}{S(\tau')} \tau', \quad 	M_{\beta}^* \Xi_{\mathfrak{l}} = \Xi_{\mathfrak{l}}, \, \mathfrak{l} \neq 0.  
\end{equs}
This allows us to conclude on the fact that
\begin{equs}
		F(M^{*}_{\beta}\tau)({\bf{u}}) = \hat{F}(\tau)({\bf{u}}).
\end{equs}
	\end{proof}

One can define the root substitution of a B-series as
\begin{equs} \label{substitution_root}
	B_-(\alpha,F,{\bf{u}}) \circ_{s,r} B_-(\beta,F,{\bf{u}}) = \sum_{\tau \in \mfT} \frac{\alpha(\tau)}{S(\tau)} \tilde{F}(\tau)({\bf{u}})
\end{equs}
where for $ \tau = X^k \prod_{i=1}^n \CI_{a_i}(\tau_i) \Xi_{\mathfrak{l}} $, one has for $ \mathfrak{l} \neq 0 $
\begin{equs}
	\tilde{F}(\tau)({\bf{u}}) = 
	\prod_{i=1}^n F(\tau_i)({\bf{u}}) \,  \partial^k \prod_{i=1}^n D_{a_i} F(\Xi_{\mathfrak{l}})({\bf{u}}). \, 
\end{equs}
When $ \mathfrak{l} = 0 $, one has
\begin{equs}
	\tilde{F}(\tau)({\bf{u}}) = \sum_{\tau' \in \mfT} \frac{\beta(\tau')}{S(\tau')}
	\prod_{i=1}^n F(\tau_i)({\bf{u}}) \,  \partial^k \prod_{i=1}^n D_{a_i} F(\tau')({\bf{u}}).
\end{equs}
We also define $F_{\circ}$ for $ \tau = X^k \prod_{i=1}^n \CI_{a_i}(\tau_i) \Xi_{\mathfrak{l}} $ by
\begin{equs}
		F_{\circ}(\tau)({\bf{u}}) = 
	\prod_{i=1}^n \hat{F}(\tau_i)({\bf{u}}) \,  \partial^k \prod_{i=1}^n D_{a_i} F(\Xi_{\mathfrak{l}})({\bf{u}}).
\end{equs}
The main difference between the root substitution and the normal substitution is that in the full substitution, we change the non-linearities associated to the noises $ \Xi_0 $ everywhere whereas with the root substitution, we change only the one  associated to the root. In the next theorem, we develop key identities between root substitution, substitution and composition.

\begin{theorem} \label{thm_root_substitution}
	One has the following identities 
	\begin{equs} \label{root_composition}
			B_-(\alpha,F,{\bf{u}}) \circ_{s,r} B_-(\beta,F,{\bf{u}}) = 	B_-(\beta,F,\cdot) \circ B_+(\alpha,F,\cdot,{\bf{u}})
	\end{equs}
where $ \alpha, \beta: \CT \rightarrow \mathbb{R}  $ are such that $ \alpha $ is multiplicative for the tree product and it is supported on decorated trees of the form $ X^k \prod_{i=1}^n \CI_{a_i}(\tau_i) \Xi_{\mathfrak{0}} $ that could be identified with an element of $  \CT_+  $. 
One has also 
\begin{equs} \label{Mcirc_R}
		B_-(\beta,F_{\circ},{\bf{u}}) \circ_{s,r} B_-(\beta,F,{\bf{u}}) = B_-(\beta,F,{\bf{u}}) \circ_{s} B_-(\beta,F,{\bf{u}}).
\end{equs}
	\end{theorem}

\begin{proof}
	For \eqref{root_composition}, we start by writing the definition of the root substitution \eqref{substitution_root}:
	\begin{equs}
		 B_-(\alpha,F,{\bf{u}}) & \circ_{s,r} B_-(\beta,F,{\bf{u}}) \\ & = \sum_{\tau \in \mfT} \frac{\alpha(\tau)}{S(\tau)} \tilde{F}(\tau)({\bf{u}})
		\\ & = \sum_{\tau = X^k \prod_{i=1}^n \CI_{a_i}(\tau_i) } \sum_{\tau' \in \mfT} \frac{\alpha(\tau)}{S(\tau)} \frac{\beta(\tau')}{S(\tau')}
		\prod_{i=1}^n F(\tau_i)({\bf{u}})   \partial^k \prod_{i=1}^n D_{a_i} F(\tau')({\bf{u}})
		\\ & = \sum_{\tau \in \mfT_+ } \sum_{\tau' \in \mfT} \frac{\alpha(\tau)}{S(\tau)} \frac{\beta(\tau')}{S(\tau')}
		 F(\tau \star_2 \tau')({\bf{u}})
		 \\ & =  \sum_{\tau \in \mfT} \frac{(\alpha \star_2 \beta)(\tau)}{S(\tau)}  F(\tau )({\bf{u}}) 
		 \\ &= B_-(\beta,F,\cdot) \circ B_+(\alpha,F,{\bf{u}}).
		\end{equs}
	where we have used the morphism property of $ \star_2 $ given in Proposition~\ref{star_morphism}. The last identities correspond to the same steps performed in the proof of Theorem \ref{main_composition}.
	The second identity follows directly from the definition of the products $ \circ_s $ and $ \circ_{r,s} $.
	\end{proof}

Let us explain how Theorem \ref{thm_substition} and Theorem \ref{thm_root_substitution} were used in the context of Regularity Structures. The solution of order $ \gamma $ of \eqref{main_equation} can be expressed using a B-series:
\begin{equs}
	u_{\mathfrak{t}} = B_+(\Pi_z \circ \mathcal{P}_{\leq \gamma},F, (\mathfrak{t},0), {\bf{u}}(z)).
\end{equs} 
Also the right hand side of \eqref{main_equation} can be expanded with a B-series:
\begin{equs}
	u_{\mathfrak{t}} = K_{\mathfrak{t}} * v_{\mathfrak{t}}, \quad
	v_{\mathfrak{t}} = B_-(\Pi_z \circ \mathcal{P}_{\leq 0},F_{\mathfrak{t}}, {\bf{u}}(z)).
\end{equs}
where for the well-posedness it is necessary to construct this expansion up to order $ \gamma = 0 $ for $ 	v_{\mathfrak{t}}  $. The main issue is that the iterated integrals given by $ \alpha = \Pi_z \circ \mathcal{P}_{\leq 0} $ may fail to converge when one starts with a smooth approximation of a singular noise and wants to pass to the limit. The main remedy is to renormalise them by introducing the correct counter-terms that will make these integrals converge. This is performed by using maps of the form $ M_{\beta} $. Then, one considers the renormalised B-series $ \hat{v}_{\mathfrak{t}} $ given by
\begin{equs}
	 \hat{v}_{\mathfrak{t}} = B_-(\alpha(M_\beta \cdot),F_{\mathfrak{t}}, {\bf{u}}(z)) = B_-(\alpha,F,{\bf{u}}(z)) \circ_s B_-(\beta,F,{\bf{u}}(z))
\end{equs}
where the last identity comes from Theorem \ref{thm_substition}. In general, one fails to get the analytical bounds for the term $ \alpha(M_\beta \tau) $ for $ \tau \in \CT $ which means that we do not get the correct local behaviour around the point $ z $.
One has to work with an extended space of decorated trees where extra decorations are added on the vertices to turn $ M_{\beta} $ into a degree-preserving map satisfying 
\begin{equs}
	M_{\beta} \tau = \tau + \sum_i c_i \tau_i, \quad \deg(\tau_i) = \deg(\tau).
\end{equs}
 This has been the strategy developed in \cite{BHZ} and it is used extensively for deriving the renormalised equation in \cite{BCCH}. In fact, one can get around this by using preparation maps that are right morphisms for the $ \star_2 $ product:
 \begin{equs}
 	R (\sigma \star_2 \tau) = 	\sigma \star_2 R \tau.
 	\end{equs}
 One main example are the maps of the form $ R = R_{\beta} $.
  The renormalised iterated integrals $\Pi^R_z, \Pi^{R,\times}_z $ are defined  recursively by
  \begin{equs} \label{def_Pi} \begin{aligned}
  	(\Pi_z^{R}  \tau)(z')  & = (\Pi_z^{R,\times} R \tau)(z'), \quad (\Pi_z^{R,\times}  \tau	\bar{\tau} )(z') = (\Pi_z^{R,\times}  \tau	 )(z') (\Pi_z^{R,\times}	\bar{\tau} )(z'),
  	\\
  	(\Pi_z^{R,\times} X^k)(z')  & = (z'-z)^k, \quad (\Pi_z^{R,\times} 	\xi_{\mathfrak{l}} )(z') = \xi_{\mathfrak{l}}(z'),
  	\\ (\Pi^{R,\times}_z \CI_{(\mathfrak{t},m)}(\tau))(z')  & = (D^{m} K_{\mathfrak{t}} * \Pi^{R}_z \tau)(z') \\ & - \sum_{|k|_{\s} \leq \deg(\CI_{(\mathfrak{t},m)}(\tau))} \frac{(z'-z)^k}{k!} (D^{m +k} K_{\mathfrak{t}} * \Pi^{R}_z \tau)(z).
  	\end{aligned}
  \end{equs}
The main point of this definition is to be valid without using extended decorations and it makes appear a multiplicative map $ \Pi_z^{R,\times} $ for the tree product. Then, one considers the following  character and its renormalisation via $ R_{\beta} $
\begin{equs}
	\tilde{\alpha} = \Pi_z^{R,\times}, \quad 	\tilde{\alpha}(R_{\beta} \cdot) = \Pi_z^{R}
\end{equs}
then by using \eqref{root_composition}, one gets
\begin{equs}
	B_-(\tilde{\alpha}( R_{\beta} \cdot) ,F,{\bf{u}}(z))  
	& =	B_-(\beta \star_2 \tilde{\alpha} ,F,{\bf{u}}(z))
	 \\ &  = B_-(\beta,F,\cdot) \circ B_+(\tilde{\alpha},F,\cdot,{\bf{u}}(z)) 
\end{equs}
The character $ \tilde{\alpha} $ contains some renormalisation but now it is multiplicative for the tree product. This simple remark is the reason why one can perform a simple proof in \cite{BB21}.

\end{document}